\documentclass{article}
\usepackage{amsfonts,amsmath,amsthm,latexsym,amscd,graphicx,xypic,psfrag}
\usepackage{epic, eepic}
\usepackage[T1]{fontenc}
\usepackage{bbold, mathabx}
\usepackage{verbatim}
\usepackage{setspace}
\usepackage[utf8]{inputenc}
\usepackage[colorlinks=true]{hyperref}

\def\<#1,#2>{\langle #1,#2 \rangle}

 \def\bothID{\rlap{\hbox to.97\wd0{\hss\vrule height.06\ht0 width.82\wd0}}
 \copy0\rlap{\kern-.36\wd0\vrule height1.05\ht0 width.05\ht0}\kern.14\wd0}

 \DeclareMathOperator{\spec}{spec}

 \DeclareMathOperator{\supp}{supp}

\begin{document}

\title{Renormalized Four Dimensional Gauge Invariant Quantum Yang-Mills Theory and Mass Gap}
\author{Simone Farinelli
        \thanks{Simone Farinelli, Aum\"ulistrasse 20,
                CH-8906 Bonstetten, Switzerland, e-mail simone.farinelli@alumni.ethz.ch}
        }
\maketitle

\begin{abstract}
A quantization procedure for the Yang-Mills equations for the Minkowski space $\mathbf{R}^{1,3}$ is carried out in such a way
that field maps satisfying Wightman axioms of Constructive Quantum Field Theory can be obtained.
Moreover, by removing the infrared and ultraviolet cutoffs, the spectrum of the corresponding (non-local) QCD Hamilton operator is proven to be positive and
bounded away from zero, except for the case of the vacuum state, which has vanishing energy level. The whole construction is invariant
for all gauge transformations preserving the Coulomb gauge.
As expected from QED, if the coupling constant converges to zero, then so does the mass gap.
This is the case for the running coupling constant leading to asymptotic freedom.\\\\
\vspace{0.2cm}
\noindent{\bf Mathematics Subject Classification (2020):} 	81T08 $\cdot $  81T13\\
\vspace{0.2cm}
\noindent{\bf Keywords:} Constructive Quantum Field Theory, Yang-Mills Theory, Mass Gap
\end{abstract}

\newtheorem{theorem}{Theorem}[section]
\newtheorem{proposition}[theorem]{Proposition}
\newtheorem{lemma}[theorem]{Lemma}
\newtheorem{corollary}[theorem]{Corollary}
\theoremstyle{definition}
\newtheorem{ex}{Example}[section]
\newtheorem{rem}{Remark}[section]
\newtheorem*{nota}{Notation}
\newtheorem{defi}{Definition}[section]
\newtheorem{conjecture}{Conjecture}
\newtheorem{counterex}{Counterexample}[section]

\tableofcontents

\section{Introduction}
Yang-Mills fields, which are also called gauge fields, are used in modern physics to describe physical fields that play the role of carriers
of an interaction (cf. \cite{EoM02}). Thus, the electromagnetic field in electrodynamics, the field of vector bosons, carriers of the weak interaction
 in the Weinberg-Salam theory of electrically weak interactions, and finally, the gluon field, the carrier of the strong interaction, are described
 by Yang-Mills fields. The gravitational field can also be interpreted as a Yang-Mills field (see \cite{DP75}).\par
The idea of a connection as a field was first developed by H. Weyl (1917), who also attempted to describe the electromagnetic field in terms of
a connection. In 1954, C.N. Yang and R.L. Mills (cf. \cite{MY54}) suggested that the space of intrinsic degrees of freedom of elementary particles
(for example, the isotropic space describing the two degrees of freedom of a nucleon that correspond to its two pure states, proton and neutron)
depends on the points of space-time, and the intrinsic spaces corresponding to different points are not canonically isomorphic.\par
In geometrical terms, the suggestion of Yang and Mills was that the space of intrinsic degrees of freedom is a vector bundle over space-time
that does not have a canonical trivialization, and physical fields are described by cross-sections of this bundle. To describe the differential
evolution equation of a field, one has to define a connection in the bundle, that is, a trivialization of the bundle along the curves in the base.
Such a connection with a fixed holonomy group describes a physical field, usually called a Yang-Mills field. The equations for a free Yang-Mills field
can be deduced from a variational principle. They are a natural non-linear generalization of Maxwell's equations (cf.\cite{Br03}).\par
Field theory does not give the complete picture. Since the early part of the 20th century, it has been understood that the description of nature at the subatomic scale requires quantum mechanics, where classical observables correspond to typically non commuting self-adjoint operators on a Hilbert space, and classic notions as ``the trajectory of a particle'' do not apply. Since fields interact with particles, it became clear by the late 1920s that an internally coherent account of nature must incorporate quantum concepts for fields as well as for particles. Under this approach components of fields at different points in space-time become non-commuting operators.\par
The most important Quantum Field Theories describing elementary particle physics are gauge theories formulated in terms of a principal fibre bundle over the Minkowskian space-time with particular choices of the structure group. They are depicted in Table \ref{GT}.\par
\begin{table}[!]
\begin{center}
\begin{tabular}{|l|l|l|}
  \hline
  \textbf{Gauge Theory} & \textbf{Fundamental Forces}& \textbf{Structure Group}\\
  \hline\hline
  Quantum Electrodynamics  & Electromagnetism &  $U(1)$\\
  (QED) & & \\
  \hline
  Electroweak Theory & Electromagnetism  & $\text{SU}(2)\times U(1)$\\
  (Glashow-Salam-Weinberg) & and weak force & \\
  \hline
  Quantum Chromodynamics  & Strong force & $\text{SU}(3)$\\
  (QCD) & and electromagnetism  & \\
  \hline
  Standard Model  & Strong, weak forces    & $\text{SU}(3)\times\text{SU}(2)\times U(1)$\\
   & and electromagnetism & \\
   \hline
   Georgi-Glashow Grand   & Strong, weak forces    & $\text{SU}(5)$\\
   Unified Theory (GUT1)& and electromagnetism & \\
  \hline
   Fritzsch-Georgi-Minkowski    & Strong, weak forces    & $\text{SO}(10)$\\
   Grand Unified Theory (GUT2)& and electromagnetism & \\
  \hline
   Grand Unified  & Strong, weak forces    & $\text{SU}(8)$\\
   Theory (GUT3)& and electromagnetism & \\
  \hline
  Grand Unified  & Strong, weak forces    & $\text{O}(16)$\\
   Theory (GUT4)& and electromagnetism & \\
  \hline
\end{tabular}
\end{center}
\caption{Gauge Theories}\label{GT}
\end{table}
As shown in \cite{JW04}, in order for Quantum Chromodynamics to completely explain the observed world of strong interactions, the theory must
imply:
\begin{itemize}
\item \textbf{Mass gap:} There must exist some positive constant $\eta$ such that the excitation of the vacuum state has energy at least $\eta$. This would explain why the nuclear force is strong but short-ranged, by providing the mathematical evidence that the corresponding exchange particle, the gluon, has non vanishing rest mass.
\item \textbf{Quark confinement:} The physical particle states corresponding to proton, neutron and pion must be $\text{SU}(3)$-invariant. This would explain why individual quarks are never observed.
\item \textbf{Chiral symmetry breaking:} In the limit for vanishing quark-bare masses the vacuum is invariant under a certain subgroup of the full symmetry group acting on the quark fields. This is required in order to account for the ``current algebra'' theory of soft pions.
\end{itemize}

\noindent The Seventh CMI-Millenium prize problem is the following conjecture.
\begin{conjecture}\label{CMI}
For any compact simple Lie group $G$ there exists a nontrivial Yang-Mills theory on the Minkowskian $\mathbf{R}^{1,3}$, whose quantization satisfies Wightman axiomatic properties of Constructive Quantum Field Theory and has a mass gap $\eta>0$.
\end{conjecture}
The conjecture is explained in \cite{JW04} and commented in \cite{Do04} and in \cite{Fa05}. To our knowledge this conjecture is unproved.\par
The first rigorous program of study of this problem is the one by Balaban (\cite{Ba84}, \cite{Ba84Bis}, \cite{Ba85}, \cite{Ba85Bis}, \cite{Ba85Tris}, \cite{Ba85Quater}, \cite{Ba87}, \cite{Ba88}, \cite{Ba88Bis}, \cite{Ba89} and \cite{Ba89Bis}). This program defines a sequence of block-spin transformations for the pure Yang-Mills
theory in a finite volume on the lattice, a toroidal Euclidean space-time, and shows that, as the lattice spacing tends to
$0$ and these transformations are iterated many times, the resulting effective action on
the unit lattice remains bounded. From this result the existence of an ultraviolet limit
for gauge invariant observables such as ``smoothed Wilson loops'' should follow, at least through a compactness argument using a subsequence of approximations; but
the limit is not necessarily unique. From this point of view, one must verify the existence of an ultraviolet limit of appropriate
expectations of gauge-invariant observables as the lattice spacing tends to zero (ultraviolet cutoff removal) and the volume
tends to infinity (infrared cutoff removal). This program is applied for the $\Phi^4_3$ model by Dimock in the very readable
 \cite{Di13}, \cite{Di13Bis}, \cite{Di14}.\par
Magnen, Rivasseau and S\'{e}n\'{e}or  provide in \cite{MRS93} the basis for a rigorous construction of the Schwinger functions
of the pure $SU(2)$ Yang-Mills field theory in four dimensions (in the trivial topological
sector) with a fixed infrared cutoff and removed ultraviolet cutoff, in a regularized axial
gauge. They check the validity of the construction by
showing that Slavnov identities (which express infinitesimal gauge invariance) do
hold non-perturbatively. They do not analyze the spectral properties of the Hamilton operator.\par
Very recently, regularity structures have been successfully by Chandra-Chevyrev-Hairer-Shen (\cite{CCHS22}) to obtain a stochastic quantization of Yang–Mills–Higgs in $3$D.
Previously, regularity structures, which had been pioneered by Hairer in \cite{Ha14}, have been applied by Gubinelli-Hofmanov\'{a} (\cite{GuHo19})
to extend a stochastic quantization of the Euclidean $\Phi^4_3$ quantum field theory from the torus $T^3$ to $\mathbf{R}^3$. It would be interesting to see this approach applied to
the Yang-Mills problem in $(3+1)$D.\par
The main contributions of this paper are:
\begin{itemize}
\item the construction of a rigorous quantum Yang-Mills theory over the whole Minkowski space satisfying the Osterwalder-Schrader axioms
of Constructive Quantum Field Theory,
\item the proof that the quantum hamiltonian for QCD has a spectral gap.
\end{itemize}

This paper is organized as follows. Section 2 presents the classical Yang-Mills equations
and their Hamiltonian formulation for the Minkowskian $\mathbf{R}^{1,3}$.
Section 3 depicts the axioms of Constructive Quantum Field Theory and may be skipped by the acquainted reader.
In Sections 4 the Yang-Mills Equations are quantized, the Osterwalder-Schrader and hence the Wightman axioms are verified,
 and the existence of a positive mass gap proven. More in detail,
the construction of the Yang-Mills Quantum Field Theory in $3+1$ spacetime dimensions and  the proof that the corresponding Hamiton operator possesses a spectral gap is done by passing to the corresponding Euclidean QFT and progresses through the following steps:
\begin{enumerate}
\item Impose an infrared and an ultraviolet cutoff.
\item Construct a background “quasi-free” QFT as the solution to an SDE via the
Fokker-Planck equation using the part of the Euclidean action of the form
$-\Delta + B \cdot \nabla$. In contrast to the free QFT this part contains a non local term.
\item Use the Feynman-Ka\v{c}-Nelson formula to add in the interaction.
\item Remove the infrared cutoff and prove the Osterwalder-Schrader axioms for any fixed ultraviolet cutoff
\item Remove the ultraviolet cutoff and prove the Osterwalder-Schrader axioms for the limit of the ultraviolet cutoff parameter tending to infinity
\item Verify the gauge invariance.
\item Show the existence of the mass gap.
\end{enumerate}
\noindent Even more in detail, by verifying the Osterwalder-Schrader axioms and the reconstruction theorem for quantum mechanics,
the QFT Hamiltonian for the continuum theory, constructed via quantization of the classical Hamiltonian,
is proved to be a selfadjoint operator on the Hilbert Space of $L^2$-Hida distributions. The Hamilton operator is non-local and the proof of its
selfadjointness for a particular probability measure requires the construction of the infinitesimal generator of a stochastic process taking values
in a $L^2$-space over the physical space, and the extension of the Feynman-Ka\v{c}-Nelson formula for the $L^2$-Hida distributions.\par
The QFT Hamiltonian has a continuous spectrum, which can be expressed as the direct limit of the continuous spectrum of another selfadjoint operator,
the QFT Hamiltonian with ultraviolet cutoff, when the cutoff parameter tends to infinity. For both operators strictly positive lower bounds of the spectra can be inferred.
 These depend on the bare coupling constant for both Hamiltonians and on the cutoff magnitude for the Hamiltonian with cutoff.
 The gap for the cutoff case ``survives'' then in the continuum limit for a strictly positive bare coupling constant. As expected from QED, if the coupling constant converges to zero, then so does the mass gap.\par
 We remark, that, due to our particular choice of the ultraviolet cutoff differing substantially from the one chosen by Magnen-Rivasseau-S\'{e}n\'{e}or in \cite{MRS93},
 we do not observe divergences arising for the ultraviolet cutoff parameter tending to infinity, and, hence, no dependence of the bare coupling constant
 on the cutoff parameter and no renormalization have to be introduced in our basic construction at this stage. Note that in the $\Phi^4_3$ model
 renormalization is needed, but concerns other parameters than the bare
 coupling constant (see \cite{GJ73}, \cite{FO76} and \cite{MS76}). Later, in order to obtain asymptotic freedom we will study the running of the coupling constant
 with respect to the energy scale, showing that the running mass gap vanishes in the asymptotic freedom limit, because the running coupling constant does, too.\par
 The whole construction is proved to be invariant for those
 gauge transforms which preserve the Coulomb gauge.
 Section 5 concludes.

\section{Yang-Mills Connections}\label{YMC}
\subsection{Definitions, Existence and Uniqueness}\label{ptrel}
A Yang-Mills connection is a connection in a principal fibre bundle over a (pseudo-)Riemannian manifold whose curvature satisfies the harmonicity condition, i.e. the Yang-Mills equation.
\begin{defi}[\textbf{Yang-Mills Connection}]
Let $P$ be a principal $G$-fibre bundle over a pseudoriemannian $m$-dimensional manifold $(M,h)$,
and let $V$ be the vector bundle associated with $P$ and $\mathbf{R}^K$, induced by the representation
$\rho:G\rightarrow\text{GL}(\mathbf{R}^K)$, where $K:=\dim(G)$. A connection on the principal fibre bundle $P$ is a Lie-algebra $\mathfrak{g}$ valued $1$-form $\omega$ on $P$,
such that the following properties hold:
\begin{enumerate}
\item[(i)] Let $A\in\mathfrak{g}$ and $A^*$ the vector field on $P$ defined by
\begin{equation}
A_p^*:=\left.\frac{d}{dt}\right|_{t:=0}(p\exp(tA)).
\end{equation}
Then, $\omega(A^*_p)=A$.
\item[(ii)]For $g\in G$ let
\begin{equation}
\begin{split}
&\text{Ad}_g:G\rightarrow G, h\mapsto\text{Ad}_g(h):=L_g\circ R_{g^{-1}}(h)=ghg^{-1}\\
&\text{ad}_g:\mathfrak{g}\rightarrow \mathfrak{g}, A\mapsto\text{Ad}_g(A):=\left.\frac{d}{dt}\right|_{t:=0}(g\exp(tA)g^{-1})\\
\end{split}
\end{equation}
be the adjoint isomorphism and the adjoint representation, respectively.\\
 Then, $R_g^*\omega=\text{ad}_{g^{-1}}\omega$.
\end{enumerate}
The connection $\omega$ on $P$ defines a connection $\nabla$ for the vector bundle $V$, i.e. an operator acting on the space of cross sections of $V$. The vector bundle connection $\nabla$ can be extended to an operator $d:\Gamma(\bigwedge^p(M)\bigotimes V)\rightarrow \Gamma(\bigwedge^{p+1}(M)\bigotimes V)$, by the formula
\begin{equation}
d^{\nabla}(\eta\otimes v):= d\eta\otimes v + (-1)^{p}\eta\otimes \nabla v.
\end{equation}
The operator $\delta^{\nabla}:\Gamma(\bigwedge^{p+1}(M)\bigotimes V)\rightarrow \Gamma(\bigwedge^{p}(M)\bigotimes V)$, defined as the formal adjoint to $d$, is equal to
\begin{equation}
\delta^{\nabla}\eta= (-1)^{p+1}\ast d^{\nabla} \ast,
\end{equation}
\noindent where $\ast$ denotes the Hodge-star operator on the pseudoriemannian manifold $M$.\par
A connection $\omega$ in a principal fibre bundle $P$ is called a Yang-Mills field if the curvature $F:=d\omega+\omega\wedge \omega$,
considered as a $2$-form with values in the Lie algebra $\mathfrak{g}$, satisfies the Yang-Mills equations
\begin{equation}\label{YME}
 \delta^{\nabla}F=0,
\end{equation}
 or, equivalently,
\begin{equation}
 \delta^{\nabla}R^{\nabla}=0,
 \end{equation}
 where $R^{\nabla}(X,Y):=\nabla_X\nabla_Y-\nabla_Y\nabla_X-\nabla_{[X,Y]}$ denotes the curvature of the vector bundle $V$, and is a $2$-form with values in $V$.
\end{defi}
\begin{rem}[\textbf{Local Representations of Connections on Vector and Principle Fibre Bundles}]\label{rem-local}
The local section  $\sigma:U\subset M\rightarrow P$ is defines the local representation of the connection given on the open $U\subset M$ by $A:=\omega\circ \sigma:U\rightarrow\mathfrak{g}$ a Lie-algebra $\mathfrak{g}$ valued
$1$-form on $U$, the fields $A_j(x):=A(x)e_j=\sum_{k=1}^KA_j^k(x)t_k$ define by means of the tangential map $T_e\rho:\mathfrak{g}\rightarrow\mathcal{L}(\mathbf{R}^K)$ of the representation $\rho:G\rightarrow\text{GL}(\mathbf{R}^K)$, with fields of endomorphisms $T_e\rho A_1,\dots,T_e\rho A_m\in\mathcal{L}(V_x)$ for the bundle $V$. Given a basis of the Lie-algebra $\mathfrak{g}$ denoted by $\{t_1,\dots,t_K\}$, the endomorphisms $\{w_s:=T_e\rho.t_s\}_{s=1,\dots,K}$ in $\mathcal{L}(\mathbf{R}^K)$ have matrix representations with respect to a local basis $\{v_s(x)\}_{s=1,\dots,K}$ denoted by $[w_s]_{\{v_s(x)\}}$. Since $\rho$ is a representation, $T_e\rho$ has maximal rank and the endomorphisms are linearly independent. Given a local basis $\{e_j(x)\}_{j=1,\dots,m}$ for $x\in U\subset M$, the Christoffel symbols of the connection $\nabla$ are locally defined by the equation
\begin{equation}
\nabla_{e_j}v_s=\sum_{r=1}^K\Gamma_{j,s}^rv_r,
\end{equation}
holding true on $U$, and they ,satisfy the equalities
\begin{equation}
\Gamma_{j,s}^r=\sum_{a=1}^K[w_a]_s^rA^a_j.
\end{equation}
Given a local vector field $v=\sum_{s=1}^Kf^sv_s$ in $V|_U$ and a local vector field $e$ in $TM|_U$, the connection $\nabla$ has a local representation
\begin{equation}
\nabla_ev=\sum_{s=1}^K(df^s(e).v_s+f^s\omega(e).v_s),
\end{equation}
where $\omega$ is an element of $T^*U|_U\bigotimes \mathcal{L}(V|_U)$, i.e. an endomorphism valued $1-$form satisfying
\begin{equation}
\omega(e_j)v_s=\sum_{r=1}^K\Gamma_{j,s}^rv_r.
\end{equation}
\end{rem}
\begin{rem}The curvature $2$-form reads in local coordinates as
\begin{equation}
F=\sum_{1\le i<j\le M}\sum_{k=1}^KF_{i,j}^k\,t_k\,dx_i\wedge dx_j=\frac{1}{2}\sum_{i,j=1}^M\sum_{k=1}^K\left(\partial_j A^k_i-\partial_i A_j^k-\sum_{a,b=1}^KC^k_{a,b}A^a_iA_j^b\right)t_k\,dx_i\wedge dx_j,
\end{equation}
where $C=[C^c_{a,b}]_{a,b,c=1,\dots,K}$ are the \textit{structure constants} of the Lie-algebra $\mathfrak{g}$ corresponding to the basis $\{t_1,\dots,t_K\}$, which means that for any $a,b$
\begin{equation}
[t_a,t_b]=\sum_{c=1}^KC_{a,b}^ct_c.
\end{equation}
\end{rem}
The existence and uniqueness of solutions of the Yang-Mills equations in the Minkowski space have been first established by Segal (cf. \cite{Se78} and \cite{Se79}), who proves that the corresponding Cauchy problem encoding initial regular data has always a unique local and global regular solution. He proves as well that the temporal gauge ($A_0=0$) chosen to express the solution does not affect generality, because any solution of the Yang-Mills equation can be carried to one satisfying the temporal gauge. This subject has been undergoing intensive research, improving the original results. For example in \cite{EM82} and in \cite{EM82Bis} the Yang-Mills-Higgs equations, which generalize (\ref{YME}) and are non linear PDEs of order two, have been reformulated in the temporal gauge as a non-linear PDE of order one, satisfying a constraint equation. This PDE can be written as an integral equation solving (always and uniquely, locally and globally) the Cauchy data problem with improved regularity results. Existence, uniqueness and regularity of the Yang-Mills-Higgs equations under the MIT Bag boundary conditions have been investigated in \cite{ScSn94} and \cite{ScSn95}.

\subsection{Hamiltonian Formulation for the Minkowski Space}
 The Hamiltonfunction describes the dynamics of a physical system in classical mechanics by means of Hamilton's equations. Therefore, we have to reformulate the Yang-Mills equations in Hamiltonian mechanical terms. We focus our attention on the Minskowski $\mathbf{R}^4$ with the pseudoriemannian structure of special relativity $h=dx^0\otimes dx^0-dx^1\otimes dx^1-dx^2\otimes dx^2-dx^3\otimes dx^3$. The coordinate $x^0$ represents the time $t$, while $x^1,x^2,x^3$ are the space coordinates.\par We introduce Einstein's summation notation, and adopt the convention that indices for coordinate variables from the greek alphabet vary over $\{0,1,2,3\}$, and those from the latin alphabet vary over the space indices $\{1,2,3\}$. For a generic field $F=[F_\mu]_{\mu=0,1,2,3}$ let $\mathbf{F}:=[F_i]_{i=1,2,3}$ denote the ``space'' component. The color indices lie in $\{1,\dots,K\}$. Let \begin{equation}\varepsilon^{a,b,c}:=\left\{
                                                       \begin{array}{ll}
                                                         +1 & \hbox{($\pi$ is even)} \\
                                                         -1 & \hbox{($\pi$ is odd)} \\
                                                         \;\;\; 0 & \hbox{(two indices are equal),}
                                                       \end{array}
                                                     \right.
\end{equation}and any other choice of lower and upper indices, be the Levi-Civita symbol, defined by mean of the permutation $\pi:=\left(
                              \begin{array}{ccc}
                                1 & 2 & 3 \\
                                a & b & c \\
                              \end{array}
                            \right)$ in $\mathfrak{S}^3$.
\begin{rem}
If the Lie-group $G$ is simple, then the Lie-Algebra is simple, and the structure constants can be written as
\begin{equation}
C_{a,b}^c = g\varepsilon^c_{a,b},
\end{equation}
for a positive constant $g$ called \textit{(bare) coupling constant}, (see f.i. \cite{We05} Chapter 15, Appendix A).
  The components of the curvature then read
\begin{equation}
F_{\mu,\nu}^k=\frac{1}{2}(\partial_\nu A^k_\mu-\partial_\mu A_\nu^k-g\varepsilon^k_{a,b}A^a_\mu A^b_\nu).
\end{equation}
We will consider only simple Lie groups. As we will see, it is essential for the existence of a mass gap for the group $G$ to be non-abelian.\\
The number $C_2(G)$ is defined as
\begin{equation}
\delta^{k,l}C_2(G)=\sum_{a,b=1}^K\varepsilon^k_{a,b}\varepsilon^l_{a,b},
\end{equation}
which is the quadratic \textbf{Casimir operator} in the adjoint representation of the Lie algebra of $G$.
\end{rem}
\noindent We need to introduce an appropriate gauge for the connections we are considering.
\begin{defi}[\textbf{Coulomb Gauge}]
A connection $A$ over the Minkowski space satisfies the \textit{Coulomb gauge} if and only if
\begin{equation}
A_0^a=0\quad\text{and}\quad\partial_jA_j^a=0
\end{equation}
for all $a=1,\dots,K$ and $j=1,2,3$.
\end{defi}
\begin{defi}[\textbf{Transverse Projector}]
Let $\mathcal{F}$ be the Fourier transform on functions in $L^2(\mathbf{R}^3,\mathbf{R})$. The transverse projector $T:L^2(\mathbf{R}^3,\mathbf{R}^3)\rightarrow L^2(\mathbf{R}^3,\mathbf{R}^3)$ is defined as
\begin{equation}
(Tv)_i:=\mathcal{F}^{-1}\left(\left[\delta_{i,j}-\frac{p_ip_j}{|p|^2}\right]\mathcal{F}(v_j)\right),
\end{equation}
and the vector field $v$ decomposes into a sum of a \textit{transversal} ($v^{\perp}$) and a \textit{longitudinal} ($v^{\parallel}$) component:
\begin{equation}
v_i=v_i^{\perp}+v_i^{\parallel},\quad v_i^{\perp}:=(Tv)_i,\quad v_i^{\parallel}:= v_i-(Tv)_i.
\end{equation}
\end{defi}
\begin{rem}
The Coulomb gauge condition for the space part of a connection $A$ is equivalent to the vanishing of its longitudinal component:
\begin{equation}
{A_i^a}^{\parallel}(t,\cdot)=0
\end{equation}
for all $i=1,2,3$, all $a=1,\dots K$ and any $t\in\mathbf{R}$. The time part $A_0$ of the connection $A$ vanishes by definition of Coulomb gauge.
\end{rem}
\begin{proposition}\label{ModGreen}
For a simple Lie-group as structure group let $A$ be a connection over the Minkowskian $\mathbf{R}^4$ satisfying the Coulomb gauge, and assume that $A_i^a(t,\cdot)\in C^{\infty}(\mathbf{R}^3,\mathbf{R})\cap L^2(\mathbf{R}^3,\mathbf{R})$ for all $i=1,2,3$, all $a=1,\dots K$ and any $t\in\mathbf{R}$. The operator $L$ on the real Hilbert space $L^2(\mathbf{R}^3,\mathbf{R}^K)$ defined as
\begin{equation}
L=L(\mathbf{A};x)=[L^{a,b}(\mathbf{A};x)]:=[\delta^{a,b}\Delta_x^{\mathbf{R}^3}+g\varepsilon^{a,c,b}A_k^c(t,x)\partial_k]
\end{equation}
is essentially self adjoint and elliptic for any time parameter $t\in\mathbf{R}$. Its spectrum lies on the real line, and decomposes into discrete $\spec_d(L)$ and continuous spectrum $\spec_c(L)$.  If $0$ is an eigenvalue, then it has finite multiplicity, i.e. $\ker(L)$ is always finite dimensional .\par
The modified Green's function $G=G(\mathbf{A};x,y)=[G^{a,b}(\mathbf{A};x,y)]\in\mathcal{S}^{\prime}(\mathbf{R}^3,\mathbf{R}^{K\times K})$ for the operator $L$ is the distributional solution to the equation
\begin{equation}\label{L}
L^{a,b}(\mathbf{A};x)G^{b,d}(\mathbf{A};x,y)=\delta^{a,d}\delta(x-y)-\sum_{n=1}^{N}\psi_n^a(\mathbf{A};x)\psi_n^d(\mathbf{A};y),
\end{equation}
where $\{\psi_n(\mathbf{A}; \cdot)\}_n$ is an o.n. $L^2$-basis of $N$-dimensional $\ker(L)$. In equation (\ref{L}) $x$ is seen as variable, while $y$ is considered as a parameter.
This modified Green's function can be written as a Riemann-Stielties integral: For any $\varphi\in\mathcal{S}(\mathbf{R}^3,\mathbf{R}^K)\cap L^2(\mathbf{R}^3,\mathbf{R}^K)$
\begin{equation}\label{gen}
G(\mathbf{A};x,\cdot)(\varphi)=\int_{\lambda\neq0}\frac{1}{\lambda}d(E_{\lambda}\varphi)(x),
\end{equation}
where $(E_{\lambda})_{\lambda\in\mathbf{R}}$ is the resolution of the identity corresponding to $L$.
\end{proposition}
\begin{rem}
In \cite{Br03} and \cite{Pe78} the modified Green's function is constructed assuming that the operator $L$ has a discrete spectral resolution $(\psi_n(\mathbf{A};\cdot),\lambda_n)_{n\ge0}$ as
\begin{equation}\label{part}
G(\mathbf{A};x,y)=\sum_{n:\lambda_n\neq0}\frac{1}{\lambda_n}\psi_n(\mathbf{A};x)\psi_n^{\dagger}(\mathbf{A};y).
\end{equation}
In particular, we have the symmetry property
 \begin{equation}
G(\mathbf{A};x,y)^{\dagger}=G(\mathbf{A};y,x)
\end{equation}
\noindent for all $x,y,\mathbf{A}$ for which the expression is well defined. Since the discontinuity points of the spectral resolution $(E_{\lambda})_{\lambda\in\mathbf{R}}$ are the eigenvalues, i. e. the elements of $\spec_d(L)$ (cf. \cite{Ri85}, Chapter 9), the solution (\ref{gen}) extends (\ref{part}) to the general case.
\end{rem}

\begin{proof}[Proof of Proposition \ref{ModGreen}]
As long as the connection satisfies the Coulomb gauge condition, the operator $L$ is symmetric and essentially selfadjoint on the appropriate domain, as a direct computation involving integration by parts can show. As its leading symbol is elliptic, the operator $L$ is elliptic and restricted to $[-\frac{R}{2}, +\frac{R}{2}]^3$, under the Dirichlet boundary conditions it has a discrete spectral resolution (cf. \cite{Gi95}, Chapter 1.11.3). Every eigenvalue has finite multiplicity. The dimension of the eigenspaces is an integer-valued, continuous, and hence constant function of $R$. For $R\rightarrow+\infty$ the discrete Dirichlet spectrum of $L$ on $[-\frac{R}{2}, +\frac{R}{2}]^3$ clusters in the spectrum of $L$ on $\mathbf{R}^3$, which decomposes into a discrete and a continuous spectrum. Therefore, the eigenvalues must have finite multiplicity and, in particular, $\ker(L)$ is finite dimensional.\\
Equation (\ref{gen}) gives the modified Green's function as it can be verified by the following computation, which holds true for any $\varphi\in\mathcal{S}(\mathbf{R}^3,\mathbf{R}^K)\cap L^2(\mathbf{R}^3,\mathbf{R}^K)$:
\begin{equation}
\begin{split}
L(\mathbf{A};x)G(\mathbf{A};x,\cdot)(\varphi)&=L\int_{\lambda\neq0}\frac{1}{\lambda}d(E_{\lambda}\varphi)(x)=\int_{\lambda\neq0}\frac{1}{\lambda}Ld(E_{\lambda}\varphi)(x)=\\
&=\int_{\mathbf{R}}d(E_{\lambda}\varphi)(x)-\int_{0^-}^{0^+}d(E_{\lambda}\varphi)(x)=\phi(x)-P_{\ker(L)}\varphi(x)=\\
&=\delta(x-\cdot)(\varphi)-\sum_{n=1}^{N}\psi_n(\mathbf{A};x)\psi_n^{\dagger}(\mathbf{A};\cdot)(\varphi).
\end{split}
\end{equation}
\end{proof}
The existence of eigenvalues of $L$ depends on the additive perturbation to the Laplacian given by $g\varepsilon^{a,c,b}A_k^c(t,x)\partial_k$. For example, if $g=0$ or $\mathbf{A}=0$, the operator $L$ has no eigenvalues and $\spec(L)=\spec_c(L)=]-\infty,0]$. In general, the spectrum depends on the choice of the connection $A$. In \cite{BEP78} and in \cite{Pe78} special cases comprising pure gauges and Wu-Yang monopoles are computed explicitly. We are interested in a reformulation of the general solution (\ref{gen}), where the dependence on the connection becomes explicit. Inspired by \cite{ChHa99} we find
\begin{proposition}\label{prop}
For a simple Lie-group as structure group, if we assume that the coupling constant $g<1$, then a Green's function $K=K(\mathbf{A};x,y)=[K^{a,b}(\mathbf{A};x,y)]$ for the operator $L$ in Proposition \ref{ModGreen}, that is, a distributional solution to the equation
\begin{equation}
L^{a,b}(\mathbf{A};x)K^{b,d}(\mathbf{A};x,y)=\delta^{a,d}\delta(x-y),
\end{equation}
is given by the convergent series in $\mathcal{S}^{\prime}(\mathbf{R}^3,\mathbf{R}^{K\times K})$
\begin{equation}\label{ser}
\begin{split}
&K^{b,d}(\mathbf{A};x,y)=\frac{\delta^{b,d}}{4\pi|x-y|}+2g\varepsilon^{e_1,b,d}\int_{\mathbf{R}^3}d^3u_1\frac{1}{4\pi|x-u_1|}A_k^{e_1}(u_1)\partial_k\left(\frac{1}{4\pi|u_1-y|}\right)+\\
&-3g^2\varepsilon^{e_1,b,s_1}\varepsilon^{s_1,e_2,d}\int_{\mathbf{R}^3}d^3u_1\frac{1}{4\pi|x-u_1|}A_k^{e_1}(u_1)\partial_k\int_{\mathbf{R}^3}d^3u_2\frac{1}{4\pi|u_1-u_2|}A_i^{e_2}(u_2)\partial_i\left(\frac{1}{4\pi|u_2-y|}\right)\\
&+\dots+\\
&+(-1)^{n-1}(n+1)g^n\varepsilon^{e_1,b,s_1}\varepsilon^{s_1,e_2,s_2}\cdots\varepsilon^{s_{n-1},e_n,d}\int_{\mathbf{R}^3}d^3u_1\frac{1}{4\pi|x-u_1|}A_k^{e_1}(u_1)\cdot\\
&\cdot\partial_k\int_{\mathbf{R}^3}d^3u_2\frac{1}{4\pi|u_1-u_2|}A^{e_2}(u_2)\partial_j\dots\int_{\mathbf{R}^3}d^3u_n\frac{1}{4\pi|u_{n-1}-u_n|}A_l^{e_n}(u_n)\partial_l\left(\frac{1}{4\pi|u_n-y|}\right)+\\
&+\dots
\end{split}
\end{equation}
Note that $x$ is the variable and $y$ a parameter.
\end{proposition}
\begin{proof}
The series (\ref{ser}) converges because of the integrability of the connection $A$ and the fact that $g<1$. Recall that $\frac{1}{|x-y|}\in L^1_{loc}(\mathbf{R}^3)\subset\mathcal{S}^{\prime}(\mathbf{R}^3)$ for any fixed $y\in\mathbf{R}^3$. We now check that it represents a Green's function for $L$:
\begin{equation}
L(\mathbf{A};x)^{a,b}K^{b,d}(\mathbf{A};x,y)=\delta^{a,d}\delta(x-y)+\lim_{n\rightarrow+\infty}\text{Rest}_n,
\end{equation}
where, after having evaluated a ``telescopic sum'', the remainder part reads
\begin{equation}
\begin{split}
\text{Rest}_n=&(-1)^{n-1}(n+1)g^{n+1}\varepsilon^{e_1,b,s_1}\varepsilon^{s_1,e_2,s_2}\cdots\varepsilon^{s_{n-1},e_n,d}\varepsilon^{a,c,b}A^c_k(x)\int_{\mathbf{R}^3}d^3u_1\frac{-x_k}{4\pi|x-u_1|^3}A_k^{e_1}(u_1)\cdot\\
&\cdot\partial_k\int_{\mathbf{R}^3}d^3u_2\frac{1}{4\pi|u_1-u_2|}\dots\int_{\mathbf{R}^3}d^3u_n\frac{1}{4\pi|u_{n-1}-u_n|}A_l^{e_n}(u_n)\partial_l\left(\frac{1}{4\pi|u_n-y|}\right).
\end{split}
\end{equation}
Because of the integrability of the connection, there exists a constant $C>0$ such that for any $\varphi\in\mathcal{S}(\mathbf{R}^3,\mathbf{R})$
\begin{equation}
|\text{Rest}_n(\varphi)|\le Cg^{n+1}\|\varphi\|_{L^2(\mathbf{R}^3,\mathbf{R})}\rightarrow0\quad(n\rightarrow+\infty),
\end{equation}
and the proposition follows.
\end{proof}
\begin{rem}
For the Minkowskian $\mathbf{R}^4$ the assumption of a (dimensionless) bare coupling constant $g<1$ is well posed: for the Yang-Mills theory one is basically allowed to choose any value $g>0$ by appropriate rescaling of the energy scale (see \cite{SaSc10}). Moreover, we will later have to analyze the case where $g\rightarrow0^+$.
\end{rem}
\begin{corollary}\label{cor}
Under the same assumptions as Proposition \ref{prop} the distribution
\begin{equation}
G(\mathbf{A};x,y)= \frac{1}{2}(K(\mathbf{A};x,y)+K(\mathbf{A};x,y)^{\dagger})
-\frac{1}{2}\sum_{n:\lambda_n=0}\left(\psi_n(\mathbf{A};x)\psi_n^{\dagger}(\mathbf{A};y)+\psi_n(\mathbf{A};y)\psi_n^{\dagger}(\mathbf{A};x)\right)
\end{equation}
is a symmetric modified Green's function for the operator $L$.
\end{corollary}
\noindent After this preparation we can turn to the Hamiltonian formulation of Yang-Mills' equations, following the results
in \cite{Br03} and \cite{Pe78}, which just need to be adapted for the generic case that $L$ has a mixture of discrete
and continuous spectral resolution. Note that
\begin{equation}
\mathcal{S}(\mathbf{R}^3,\mathbf{R}^{K\times K})\subset L^2(\mathbf{R}^3,\mathbf{R}^{K\times K}, d^3x)\subset \mathcal{S}^{\prime}(\mathbf{R}^3,\mathbf{R}^{K\times K})
\end{equation}
is a rigged Hilbert space and $L$ a selfadjoint operator with a complete set of generalized eigenvectors (cf. Appendix \ref{AppA}).
\begin{theorem}\label{CHam}. For a simple Lie-group as structure group and for canonical variables  satisfying the Coulomb gauge condition, the Yang-Mills equations for the Minkowskian $\mathbf{R}^4$  can be written as Hamilton equations
\begin{equation}
\left\{
  \begin{array}{ll}
    \frac{d\mathbf{E}}{dt}&=-\frac{\partial H}{\partial \mathbf{A}}(\mathbf{A},\mathbf{E})\\
    \\
    \frac{d\mathbf{A}}{dt}&=+\frac{\partial H}{\partial \mathbf{E}}(\mathbf{A},\mathbf{E})
  \end{array}
\right.
\end{equation}
for the following choices:
\begin{itemize}
\item \textbf{Position variable:} $\mathbf{A}=[A_{i}^a(t,x)]_{\substack{ a=1,\dots,K \\ i=1,2,3}}$ also termed {\it potentials} ,
\item \textbf{Momentum variable:} $\mathbf{E}=[E_{i}^a(t,x)]_{\substack{ a=1,\dots,K \\ i=1,2,3}}$, whose entries are termed {\it chromoelectric fields},
\item \textbf{Hamilton function:} defined as a function of $\mathbf{A}$ and $\mathbf{E}$ as
 \begin{equation}\label{classicH}
  H=H(\mathbf{A},\mathbf{E}):=\frac{1}{2}\int_{\mathbf{R}^3}d^3x\left(E_i^a(t,x)^2+B_i^a(t,x)^2+f^a(t,x))\Delta f^a(t,x)+2\rho^c(t,x)A^c_0(t,x)\right),
 \end{equation}
 \noindent where $\mathbf{B}=[B_i^a]$, whose entries are termed {\it chromomagnetic fields}, is the matrix valued-function defined as
 \begin{equation}
     \begin{split}
       B_i^a&:=\frac{1}{4}\varepsilon_{i}^{j,k}\left(\partial_jA_k^a-\partial_kA^a_j+g\varepsilon^a_{b,c}A_j^bA_k^c\right),
     \end{split}
 \end{equation}
\noindent and $\rho=[\rho^a(t,x)]$, termed {\it charge density}, is  the vector valued function defined as
\begin{equation}\label{cd}
\rho^a:=g\varepsilon^{a,b,c}E^b_iA^c_i,
\end{equation}
\noindent and where
\begin{equation}\label{deff}
\begin{split}
f^a(t,x)&:=-\int_{\mathbf{R}^3}d^3y\,G^{a,b}(\mathbf{A};x,y)\rho^b(t,y)=-G^{a,b}(\mathbf{A};x,\cdot)(\rho^b(t,\cdot))\\
A_0^a(t,x)&:=\int_{\mathbf{R}^3}d^3y\,G^{a,b}(\mathbf{A};x,y)\Delta f^b(t,y)=G^{a,b}(\mathbf{A};x,\cdot)(\Delta f^b(t,\cdot))\\
\end{split}
\end{equation}
\noindent for the modified Green's function $G(\mathbf{A};x,y)$ for the operator $L(\mathbf{A};x)$.
\end{itemize}
We consider position $\mathbf{A}$ and momentum variable $\mathbf{E}$ as elements of $\mathcal{S}(\mathbf{R}^3,\mathbf{C}^{K \times 3})$ depending on the time parameter $t$, so that the  RHSs of equations (\ref{deff}) are well defined distributions applied to test functions.
\end{theorem}
\begin{rem}
As shown in \cite{Pe78} the ambiguities discussed by Gribov in \cite{Gr78} concerning the gauge fixing (see also \cite{Si78} and \cite{He97}) can be traced precisely to the existence of zero eigenfunctions of the operator $L$.
\end{rem}

\begin{corollary}\label{corClassicH}
The Hamilton function (\ref{classicH}) for the Yang-Mills equations can be written as
\begin{equation}
H=H_I+H_{II}+V,
\end{equation}
where
\begin{equation}\label{classicHamiltonFunction}
\begin{split}
H_I&=\frac{1}{2}\int_{\mathbf{R}^3}d^3x\,E_i^a(t,x)^2\\
& \\
H_{II}&=\frac{g^2}{2}\int_{\mathbf{R}^3}d^3x\left[\int_{\mathbf{R}^3}d^3y\,\partial_iG^{a,b}(\mathbf{A};x,y)\varepsilon^{b,c,d}A_k^d(t,y)E_k^c(t,y)\right]^2\\
& \\
V&=\frac{1}{16}\int_{\mathbf{R}^3}d^3x\,\varepsilon_i^{j,k}\varepsilon_i^{p,q}[(\partial_jA^a_k-\partial_kA_j^a+g\varepsilon^{a,b,c}A_j^bA_k^c)(\partial_pA^a_q-\partial_qA_p^a+g\varepsilon^{a,b,c}A_p^bA_q^c)](t,x).
\end{split}
\end{equation}
\end{corollary}

\begin{proof}
The expressions for the functions $H_I$ and $V$ are obtained by a straightforward calculation. For the function $H_{II}$ some more work is needed. Inserting (\ref{deff}) in the last addendum of (\ref{classicH}) we obtain
\begin{equation}\label{classicH3}
\begin{split}
&\int_{\mathbf{R}^3}d^3x\left(\frac{1}{2}f^a(t,x)\Delta f^a(t,x)+\rho^c(t,x)A_0^c(t,x)\right)=\\
&=\int_{\mathbf{R}^3}d^3x\left[\frac{1}{2}\left(-\int_{\mathbf{R}^3}d^3y\,G^{a,b}(\mathbf{A};x,y)\rho^b(t,y)\right)\Delta\left(-\int_{\mathbf{R}^3}d^3\bar{y}\,G^{a,b}(\mathbf{A};x,\bar{y})\rho^b(t,\bar{y})\right)\right.+\\
&\qquad\qquad\quad+\left.\rho^c(t,x)\int_{\mathbf{R}^3}d^3y\,G^{c,b}(\mathbf{A};x,y)\int_{\mathbf{R}^3}d^3\bar{y}\,\Delta G^{b,d}(\mathbf{A};y,\bar{y})\rho^d(t,\bar{y})\right]=\\
&=\int_{\mathbf{R}^3}d^3x\int_{\mathbf{R}^3}d^3y\int_{\mathbf{R}^3}d^3\bar{y}\left[\frac{1}{2}\left(G^{a,b}(\mathbf{A};x,y)\rho^b(t,y)\Delta G^{a,d}(\mathbf{A};x,\bar{y})\rho^d(t,\bar{y})\right)\right.+\\
&\qquad\qquad\quad-\left.\rho^c(t,x)G^{c,b}(\mathbf{A};x,y)\Delta G^{b,d}(\mathbf{A};y,\bar{y})\rho^d(t,\bar{y})\right]=\\
&=\int_{\mathbf{R}^3}d^3x\int_{\mathbf{R}^3}d^3y\int_{\mathbf{R}^3}d^3\bar{y}\left[\frac{1}{2}\left(G^{a,b}(\mathbf{A};x,y)\Delta G^{a,d}(\mathbf{A};x,\bar{y})\rho^b(t,y)\rho^d(t,\bar{y})\right)\right.+\\
&\qquad\qquad\quad-\left.G^{a,b}(\mathbf{A};y,x)\Delta G^{a,d}(\mathbf{A};x,\bar{y})\rho^b(t,y)\rho^d(t,\bar{y})\right]=\\
&=\frac{1}{2}\int_{\mathbf{R}^3}d^3x\int_{\mathbf{R}^3}d^3y\int_{\mathbf{R}^3}d^3\bar{y}\left[\partial_iG^{a,b}(\mathbf{A};x,y)\rho^b(t,y)\right]\left[\partial_iG^{a,d}(\mathbf{A};x,\bar{y})\rho^d(t,\bar{y})\right]=\\
&=\frac{1}{2}\int_{\mathbf{R}^3}d^3x\left[\int_{\mathbf{R}^3}d^3y\,\partial_iG^{a,b}(\mathbf{A};x,y)\rho^b(t,y)\right]^2,
\end{split}
\end{equation}
where in the last transformation formula  we have utilized integration by parts in the variables $x_1, x_2, x_3$,
and the fact that $G(\mathbf{A};x,y)=G(\mathbf{A};y,x)$. Inserting expression (\ref{cd}) for the charge density completes the proof.
\end{proof}

\section{Axioms of Constructive Quantum Field Theory}
\subsection{Wightman Axioms}
In $1956$ Wightman first stated the axioms needed for CQFT in his seminal work \cite{Wig56}, which remained not very widely spread in the scientific community till 1964, when the first edition of \cite{SW10} appeared. We will list the axioms in the slight refinement of \cite{BLOT89} and \cite{DD10}.
\begin{defi}[\textbf{Wightman Axioms}]
A scalar (respectively vectorial-spinorial) quantum field theory consists of a separable Hilbert space $\mathcal{E}$, whose elements are called states, a unitary representation $U$ of the Poincar\'{e} group $\mathcal{P}$ in $\mathcal{E}$, an operator valued distribution $\Phi$ (respectively $\Phi_1,\dots,\Phi_d$) on $\mathcal{S}(\mathbf{R}^4)$ with values in the unbounded operators of $\mathcal{E}$, and a dense subspace $\mathcal{D}\subset\mathcal{E}$ such that the following properties hold:
\begin{description}
  \item[(W1) Relativistic invariance of states:] The representation $U:\mathcal{P}\rightarrow\mathcal{U}(\mathcal{E})$ is strongly continuous.
  \item[(W2) Spectral condition:] Let $P_0,P_1,P_2,P_3$ be the infinitesimal generators of the one-parameter groups $t\mapsto U(te_\mu,I)$ for $\mu=0,1,2,3$. The operators $P_0$ and $P_0^2-P_1^2-P_2^2-P_3^2$ are positive. This is equivalent to the spectral measure $E_{\cdot}$ on $\mathbf{R}^4$ corresponding to the restricted representation $\mathbf{R}^4\ni a\mapsto U(a,I)$ having support in the positive light cone (cf. \cite{RS75}, Chapter IX.8).
  \item[(W3) Existence and uniqueness of the vacuum:] There exists a unique state $\Omega_0\in\mathcal{D\subset\mathcal{E}}$ such that $U(a,I)\Omega_0=\Omega_0$ for all $a\in\mathbf{R}^4$.
  \item[(W4) Invariant domains for fields:] The maps $\Phi:\mathcal{S}(\mathbf{R}^4)\rightarrow\mathcal{O}(\mathcal{E})$, and, respectively $\Phi_1,\dots,\Phi_d:\mathcal{S}(\mathbf{R}^4)\rightarrow\mathcal{O}(\mathcal{E})$, from the Schwartz space of test functions to (possibly) unbounded selfadjoint operators on the Hilbert space, satisfy following properties
      \begin{description}
      \item[(a)] For all $\varphi\in\mathcal{S}(\mathbf{R}^4)$ and all field maps, the domain of definitions $\mathcal{D}(\Phi(\varphi))$, $\mathcal{D}(\Phi(\varphi)^*)$, and respectively $\mathcal{D}(\Phi_j(\varphi))$, $\mathcal{D}(\Phi_j(\varphi)^*)$, all contain $\mathcal{D}$ and the restrictions of all operators to $\mathcal{D}$ agree.
      \item[(b)] $\Phi(\varphi)(\mathcal{D})\subset\mathcal{D}$, and, respectively $\Phi_j(\varphi)(\mathcal{D})\subset\mathcal{D}$.
      \item[(c)] For any $\psi\in\mathcal{D}$ fixed, the maps $\varphi\mapsto\Phi(\varphi)\psi$, and, respectively $\varphi\mapsto\Phi_j(\varphi)\psi$, are linear.
      \end{description}
  \item[(W5) Regularity of fields:] For all $\psi_1,\psi_2\in\mathcal{D}$, the map $\varphi\mapsto(\psi_1,\Phi(\varphi)\psi_2)$, and, respectively the maps $\varphi\mapsto(\psi_1,\Phi_j(\varphi)\psi_2)$, are tempered distributions, i.e. elements of $\mathcal{S}^{\prime}(\mathbf{R}^4)$.
  \item[(W6) Poincar\'{e} invariance:] For all $(a,\Lambda)\in\mathcal{P}$, $\varphi\in\mathcal{S}(\mathbf{R}^4)$, and $\psi\in\mathcal{D}$, the inclusion $U(a,\Lambda)\mathcal{D}\subset\mathcal{D}$ must hold and
  \begin{description}
  \item[(Scalar field case): ] The following equation must hold for all $\Lambda\in \text{O}(1,3)$
  \begin{equation}
  U(a,\Lambda)\Phi(\varphi)U(a,\Lambda)^{-1}\psi=\Phi((a,\Lambda)\varphi)\psi.
  \end{equation}
  \item[(Vectorial/Spinorial field case): ] There exists a representation of $\text{SL}(2,\mathbf{C})$ on $\mathbf{C}^d$ denoted by $\rho$, and satisfying $\rho(-I)=I$ or $\rho(I)=I$, such that for all $\Lambda\in \text{SO}^+(1,3)$ and $\Phi:=[\Phi_1,\dots,\Phi_d]^{\dagger}$
  \begin{equation}
  U(a,\Lambda)\Phi(\varphi)U(a,\Lambda)^{-1}\psi=\rho(s^{-1}(\Lambda))\Phi((a,\Lambda)\varphi)\psi,
  \end{equation}
  where $s$ denotes the spinor map $s:\text{SL}(2,\mathbf{C})\rightarrow\text{SO}^{+}(1,3)$ defined as below.
  The vector spaces of Hermitian matrices $\mathfrak{H}$ in $\mathbf{C}^{2\time 2}$ and $\mathbf{R}^4$ are isomorphically mapped  by
    \begin{equation}
     \begin{split}
      X:\mathbf{R}^4&\longrightarrow \mathfrak{H}\\
      x=(x^0, x^1, x^2, x^3)&\mapsto X(x):=\left[
                                          \begin{array}{cc}
                                            x^0+x^3 & x^1-ix^2 \\
                                            x^1+ix^2 & x^0-x^3 \\
                                          \end{array}
                                        \right].
      \end{split}
    \end{equation}
    The group $\text{SL}(2,\mathbf{C})$ acts on $\mathfrak{H}$ by
    \begin{equation}
     \begin{split}
      \text{SL}(2,\mathbf{C})\times\mathfrak{H} &\longrightarrow \mathfrak{H}\\
      (P,X)&\mapsto P.X:=PXP^*.
      \end{split}
    \end{equation}
    The spinor map is defined as
    \begin{equation}\label{spinormap}
     \begin{split}
      s: \text{SL}(2,\mathbf{C})&\longrightarrow \text{SO}^+(1,3) \\
        P&\mapsto s(P):x\mapsto s(P)x:=X^{-1}(PX(x)P^*).
      \end{split}
    \end{equation}
  \end{description}
  \item[(W7) Microscopic causality or local commutativity:] Let $\varphi,\chi\in\mathcal{S}(\mathbf{R}^4)$, whose supports are spacelike separated, i.e. $\phi(x)\chi(y)=0$, if $x-y$ is not in the positive light cone. Then,
      \begin{description}
      \item[(Scalar field case): ] The images of the test functions by the map field must commute
          \begin{equation}
          [\Phi(\varphi), \Phi(\chi)]=0.
          \end{equation}
      \item[(Vectorial/Spinorial field case): ] For any field maps $j,i=1,\dots,d$ either the commutations
          \begin{equation}
          [\Phi_j(\varphi), \Phi_i(\chi)]=0,\quad [\Phi_j^*(\varphi), \Phi_i(\chi)]=0,
          \end{equation}
          or the anticommutations
          \begin{equation}
          [\Phi_j(\varphi), \Phi_i(\chi)]_{+}=0,\quad[\Phi_j^*(\varphi), \Phi_i(\chi)]_{+}=0
          \end{equation}
          hold.
      \end{description}
  \item[(W8) Cyclicity of the vacuum:] \text{}
  \begin{description}
      \item[(Scalar field case):] The set
     \begin{equation}
        \mathcal{D}_0:=\{\Phi(\varphi_1)\cdots\Phi(\varphi_n)\Omega_0\,|\,\varphi_j\in\mathcal{S}(\mathbf{R}^4), n\in\mathbf{N}_0\}
     \end{equation}\noindent is dense in $\mathcal{E}$.
      \item[(Vectorial/Spinorial field case):]  The set \begin{equation}\mathcal{D}_0:=\{\Psi(\varphi_1)\cdots\Psi(\varphi_n)\Omega_0\,|\,\varphi_j\in\mathcal{S}(\mathbf{R}^4), \Psi\in\{\Phi_1,\dots,\Phi_d,\Phi_1^*,\dots,\Phi_d^*\},\,n\in\mathbf{N}_0\}\end{equation}\noindent is dense in $\mathcal{E}$.
      \end{description}
\end{description}

\end{defi}

\subsection{Gaussian Random Processes}
The Hilbert spaces utilized in quantum field theory are realized as $L^2$ spaces over the tempered distributions, the latter seen as probability space. This construction turns out to be isomorphic to that of the Fock space. We follow chapter 1 of \cite{Sim15}, chapter 5 of \cite{BHL11} and appendix $A.4$ of \cite{GJ87}. For a general overview see Kuo (\cite{Ku96}), and the framework for functional integration developed
by Cartier/DeWitt-Morette as in \cite{Ca97}, \cite{La04} and \cite{CDM10}. Although Feynman's integral can be provided a rigorous foundation by mean of functional integration, we prefer to use the Feynman-Ka\v{c} approach, since we will be working with Euclidean fields. \par
Let $(\Omega,\mathcal{A},\mu)$ be a probability space, and $\mathcal{M}(\Omega,\mathcal{A}):=\{f:\Omega\rightarrow\mathbf{R}\,|\,f\text{ measurable}\}$ the algebra of random variables. A random variable $f$ on $\Omega$ has a probability distribution function $\mu_f(U):=\mu(f^{-1}(U))$ defined on the measurable sets $U$ of $\mathbf{R}$, and a characteristic function
\begin{equation}
S_f(t):=\int_{\mathbf{R}}e^{itx}d\mu_f(x),
\end{equation}
defined as the Fourier inverse transform of the probability density $\rho_f:=\mu_f^{\prime}$. A real valued random variable $f$ has mean $0$ and variance $a\ge0$ if and only if $S_f(t)=e^{-\frac{a}{2}t^2}$. A well known result (see f.i. \cite{RS75}) is the following
\begin{theorem}[\textbf{Bochner}]
A function $S:\mathbf{R}\rightarrow\mathbf{C}$ is the characteristic function of a random variable $f:\Omega\rightarrow\mathbf{R}$ if and only if the following conditions are satisfied:
\begin{enumerate}
\item[(i)] $S(0)=1$.
\item[(ii)] $t\mapsto S(t)$ is continuous.
\item[(iii)]$\{t_i\}_{i=1,\dots,n}\subset\mathbf{R}, \{z_i\}_{i=1,\dots,n}\subset\mathbf{C}\Rightarrow\sum_{i,j=1}^nz_i\bar{z}_j S(t_i-t_j)\ge0$.
\end{enumerate}
\end{theorem}
\noindent The construction of Bochner's theorem can be lifted to generating functionals of random variables over the space of tempered distributions.
\begin{defi}
A random process indexed by a real vector space $V$ is a linear map $\Phi:V\rightarrow\mathcal{M}(\Omega,\mathcal{A})$. It is termed Gaussian random process if
\begin{enumerate}
\item[(i)] $\Phi(v)$ is a Gaussian random variable for all $v\in V$.
\item[(ii)] $\{\Phi(v)\,|\,v\in V\}$ is full, i.e. $\mathcal{A}$ is the smallest $\sigma$-algebra for which this is a family of measurable functions.
\end{enumerate}
\end{defi}
\begin{theorem}
Let $\mathcal{H}$ be a real Hilbert space. Up to isomorphism there exists exactly one Gaussian random process indexed by $\mathcal{H}$ such that
\begin{equation}
(\Phi(v),\Phi(w))_{L^2(\Omega,d\mu)}:=\int_\Omega\Phi(v)\overline{\Phi(w)}=\frac{1}{2}(v,w)_{\mathcal{H}}.
\end{equation}
\end{theorem}
\begin{proof} See Theorem I.9 (page 20) in \cite{Sim15} or Lemma 5.4 (page 258) and Proposition 5.6 (page 260) in \cite{BHL11}.
\end{proof}
Given a real Hilbert space there are different but isomorphic models for the probability space $(\Omega,\mathcal{A},\mu)$ admitting a Gaussian process. We choose the tempered distributions $\mathcal{S}^{\prime}(\mathbf{R}^N)$ as model space. Remark that any fixed test function $f\in\mathcal{S}(\mathbf{R}^N)$ and variable $\Phi\in\mathcal{S}^{\prime}(\mathbf{R}^N)$, the expression $\Phi(f)$ defines a random variable over $\mathcal{S}^{\prime}(\mathbf{R}^N)$.
As in appendix $A.6$ of \cite{GJ87} Bochner's theorem generalizes to
\begin{theorem}[\textbf{Milnos}]\label{Milnos}
Let $S:\mathcal{S}(\mathbf{R}^N)\rightarrow\mathbf{C}$ be a function. There exists a probability measure $\mu$ satisfying
\begin{equation}
S(f)=\int_{\mathcal{S}^{\prime}(\mathbf{R}^N)}e^{\imath\Phi(f)}d\mu(\Phi),
\end{equation}
for all $f\in\mathcal{S}(\mathbf{R}^N)$, if and only if
\begin{enumerate}
\item[(i)] $S(0)=1$.
\item[(ii)] $f\mapsto S(f)$ is continuous in the $\mathcal{S}(\mathbf{R}^N)\rightarrow\mathbf{C}$ topology.
\item[(iii)]$\{f_i\}_{i=1,\dots,n}\subset\mathcal{S}(\mathbf{R}^N), \{z_i\}_{i=1,\dots,n}\subset\mathbf{C}\Rightarrow\sum_{i,j=1}^nz_i\bar{z}_j S(f_i-f_j)\ge0$.
\end{enumerate}
\end{theorem}
\begin{rem}
The $\sigma$-algebra $\mathcal{A}$ is generated by the cylinder sets in $\mathcal{S}^\prime(\mathbf{R}^N)$, i.e. subsets of the tempered distribution space of the form
\begin{equation}
\left\{\Phi\in\mathcal{S}^\prime(\mathbf{R}^N)\,\left|\;(\Phi(f_1),\dots\Phi(f_n))\in U\right.\right\},
\end{equation}
where $U$ is a fixed Borel set in $\mathbf{R}^n$, and $f_1,\dots,f_n$ fixed test functions in $\mathcal{S}(\mathbf{R}^n)$ for a $n\in\mathbf{N}_0$.
\end{rem}
\begin{rem}
Minlos's theorem holds true for topological vector spaces,  stating that a cylindrical measure on the dual of a nuclear space is a Radon measure
if its Fourier transform is continuous, see (\cite{Sc73}).
\end{rem}
\noindent We utilize Minlos' theorem to define Gaussian measures on the tempered distributions. Let $c$ be a positive semidefinite quadratic form on $\mathcal{S}(\mathbf{R}^N)$. Applying Theorem \ref{Milnos} to the functional $S(f):=e^{-\frac{1}{2}c(f,f)}$ we can construct a measure $\mu$ on $\mathcal{S}^\prime(\mathbf{R}^N)$ such that
\begin{equation}
\int_{\mathcal{S}^{\prime}(\mathbf{R}^N)}e^{\imath\Phi(f)}d\mu(\Phi)=e^{-\frac{1}{2}c(f,f)}
\end{equation}
Therefore, $\Phi(f)$ is a Gaussian random variable with variance $c(f,f)$, because for all $t\in\mathbf{R}$
\begin{equation}
\int_{\mathcal{S}^{\prime}(\mathbf{R}^N)}e^{it\Phi(f)}d\mu(\Phi)=e^{-\frac{t^2}{2}c(f,f)}
\end{equation}
holds true. If $\mathcal{H}$ is a real Hilbert space such that the embedding $\mathcal{S}\hookrightarrow\mathcal{H}$ is continuous and dense, then the inner product in $\mathcal{H}$ restricts to a positive definite bilinear form on $\mathcal{S}$, which can be written as  $c(f,g)=(Cf,g)_{\mathcal{H}}$ for all $f,g\in\mathcal{S}$ for a positive definite operator $C$ on $\mathcal{H}$ with domain of definition $\mathcal{S}$.
\begin{defi}[\textbf{Gaussian measures}]
A measure $\mu$ on $\mathcal{S}^{\prime}(\mathbf{R}^N)$ defined by Milnos's theorem satisfying
\begin{equation}
\int_{\mathcal{S}^{\prime}(\mathbf{R}^N)}e^{\imath\Phi(f)}d\mu(\Phi)=e^{-\frac{1}{2}(Cf,f)_{\mathcal{H}}},
\end{equation}
where $C$ is a positive semidefinite operator $C$ on $\mathcal{H}$ with domain of definition $\mathcal{S}$, termed \textbf{covariance operator}, is called \textbf{Gaussian}.
\end{defi}
The operation of derivation can be defined for functionals of random variables as well.
\begin{defi}[\textbf{Functional Derivative}] Let $F\in L^2(\mathcal{S}^{\prime}(\mathbf{R}^N),\mu)$. Its \textbf{functional directional derivative} in the direction $\Upsilon\in\mathcal{S}^{\prime}(\mathbf{R}^N)$ is defined as G\^{a}teaux derivative as
\begin{equation}
\frac{\delta F}{\delta \Phi}(\Phi).\Upsilon:=\left.\frac{d}{d\epsilon}\right|_{\epsilon=0}F(\Phi+\epsilon \Upsilon).
\end{equation}
The special case $\Upsilon=\delta(\cdot-x)$ for $x\in\mathbf{R}^N$ arises often and is thus given a special notation
\begin{equation}
\frac{\delta }{\delta \Phi(x)}F(\Phi):=\frac{\delta F}{\delta \Phi}(\Phi).\delta(\cdot-x).
\end{equation}
\end{defi}

\subsubsection{Fock Space}
To extend a quantum mechanical model accounting for a fixed number of particles to one accounting for an arbitrary number, the following  procedure is required.

\begin{defi}[\textbf{Second Quantization}]
Let $\mathcal{H}$ be the Hilbert space whose unit sphere corresponds to the possible pure quantum states of the system with a fixed number of particles. The
\textbf{Fock space} $\mathcal{F}(\mathcal{H}):= \bigoplus_{n=0}^\infty\mathcal{H}^{(n)}$ where, $\mathcal{H}^{(0)}:=\mathbf{C}$ and $\mathcal{H}^{(n)}:=\mathcal{H}\otimes\cdots\otimes\mathcal{H}$ ($n$ times tensor product) is the vector space representing the states of a quantum system with a variable number of particles. The vector $\Omega_0:=(1,0\dots)\in\mathcal{F}(\mathcal{H})$ is called the \textbf{vacuum vector}. Given $\psi\in\mathcal{F}(\mathcal{H})$, we write $\psi^{(n)}$ for the orthogonal projection of $\psi$ onto $\mathcal{H}^{(n)}$. The set $F_0$ consisting of those $\psi$ such that $\psi^{(n)}=0$ for all sufficiently large $n$ is a dense subspace of the Fock space, called the \textbf{space of finite particles}. The \textbf{symmetrization} and  {anti-symmetrization} operators
\begin{equation}
    \begin{split}
     S_n(\psi_1\otimes\dots\otimes\psi_n)&:=\frac{1}{n!}\sum_{\sigma\in\mathfrak{S}^n}\psi_{\sigma(1)}\otimes\dots\otimes\sigma_{\sigma(n)}\\
     A_n(\psi_1\otimes\dots\otimes\psi_n)&:=\frac{1}{n!}\sum_{\sigma\in\mathfrak{S}^n}(-1)^{\text{sgn}(\sigma)}\psi_{\sigma(1)}\otimes\dots\otimes\sigma_{\sigma(n)}
    \end{split}
\end{equation}
extend by linearity to $\mathcal{H}^{(n)}$ and are projections. The state space for $n$ fermions is defined as $\mathcal{H}^{(n)}_a:=A_n(\mathcal{H}^{n})$ and that for $n$ bosons as $\mathcal{H}^{(n)}_s:=S_n(\mathcal{H}^{n})$. The \textbf{Fermionic Fock space} is defined as
\begin{equation}
\mathcal{F}_a(\mathcal{H}):=\bigoplus_{n=0}^{\infty}\mathcal{H}^{(n)}_a,
\end{equation}
and the \textbf{Bosonic Fock space} as
\begin{equation}
\mathcal{F}_s(\mathcal{H}):=\bigoplus_{n=0}^{\infty}\mathcal{H}^{(n)}_s.
\end{equation}
A \textit{unitary} operator $U:\mathcal{H}\rightarrow\mathcal{H}$ can be uniquely extended to a \textit{unitary} operator $\Gamma(U):\mathcal{F}(\mathcal{H})\rightarrow\mathcal{F}(\mathcal{H})$ as
\begin{equation}
\Gamma(U)|_{\mathcal{H}^{(n)}}:=\bigotimes_{j=1}^n U.
\end{equation}
 A \textit{selfadjoint} operator $A$ on $\mathcal{H}$ with dense subspace $\mathcal{D}(A)\subset\mathcal{H}$ can be uniquely extended to a \textit{selfadjoint} operator $d\Gamma(A)$ on $\mathcal{F}(\mathcal{H})$ as closure of the essentialy selfadjoint operator
\begin{equation}
d\Gamma(A)|_{\mathcal{D}(d\Gamma A)\cap\mathcal{H}^{(n)}}:=\bigoplus_{j=1}^n\mathbb{1}\otimes\dots\mathbb{1}\otimes\underbrace{A}_{j}\otimes\mathbb{1}\dots\mathbb{1},
\end{equation}
where
\begin{equation}
\mathcal{D}(d\Gamma(A)):=\left\{\psi\in F_0\left|\,\psi^{(n)}\in\bigotimes_{j=1}^n\mathcal{D}(A)\text{ for each }n\right.\right\}.
\end{equation}
The operator $d\Gamma(A)$ is called the \textbf{second quantization} of $A$.
\end{defi}
It is easy to prove that the spectrum of the second quantization can be inferred from the spectrum of the first.
\begin{proposition}\label{spec}
Let $A$ be a selfadjoint operator with a discrete spectral resolution, i.e. $A\varphi_j=\lambda_j\varphi_j$, where $\{\lambda_j\}_{j\ge0}\subset\mathbf{R}$ and $\{\varphi_j\}_{j\ge0}$ is a o.n.B in $\mathcal{H}$. Then, $d\Gamma(A)|_{\mathcal{H}^{(n)}}$ has a discrete spectral resolution given by
\begin{equation}
d\Gamma(A)|_{\mathcal{H}^{(n)}}\varphi_{i_1}\otimes\dots\otimes\varphi_{i_n}=\left(\sum_{j=1}^n\lambda_{i_j}\right)\varphi_{i_1}\otimes\dots\otimes\varphi_{i_n}\quad(i_j\ge0, 1\le j\le n).
\end{equation}
If  $A$ is a selfadjoint operator with continuous spectrum $\spec_c(A)$, then $d\Gamma(A)|_{\mathcal{H}^{(n)}}$ has a continuous spectrum given by
\begin{equation}
\spec_c\left(d\Gamma(A)|_{\mathcal{H}^{(n)}}\right)=\left\{\left.\sum_{j=1}^n\lambda_j\right|\,\lambda_j\in\spec_c(A)\text{ for }1\le j\le n\right\}.
\end{equation}
\end{proposition}

\begin{defi}[\textbf{Segal Quantization}]
Let $f\in\mathcal{H}$ be fixed. For vectors in $\mathcal{H}^{(n)}$ of the form $\eta=\psi_1\otimes\psi_2\otimes\dots\otimes\psi_n$ we define a map $b^{-}(f):\mathcal{H}^{(n)}\rightarrow \mathcal{H}^{(n-1)}$ by
\begin{equation}
b^{-}(f)\eta:=(f,\psi_1)\psi_2\otimes\dots\otimes\psi_n.
\end{equation}
\noindent The expression $b^{-}(f)$ extends by linearity to a bounded operator on $\mathcal{F}(\mathcal{H})$. The operator $N:=d\Gamma(I)$ is termed \textbf{number operator} and
\begin{equation}
a^{-}(f):=\sqrt{N+1}\,b^{-}(f)
\end{equation}
is called the \textbf{annihilation operator} on $\mathcal{F}_s(\mathcal{H})$. Its adjoint, $a^{-}(f)^*$ is called the \textbf{creation operator}. Finally, the real linear (but not complex linear) operator
\begin{equation}
\Phi_S(f):=\frac{1}{\sqrt{2}}\left(a^{-}(f)+a^{-}(f)^*\right)
\end{equation}
is termed \textbf{Segal field operator}, and the map
\begin{equation}
\begin{split}
\Phi_S:\mathcal{H}&\rightarrow\mathcal{O}(\mathcal{F}_s(\mathcal{H}))\\
f&\mapsto\Phi_S(f)
\end{split}
\end{equation}
the \textbf{Segal quantization} over $\mathcal{H}$.
\end{defi}
\begin{theorem}\label{Segal}
Let $\mathcal{H}$ be a complex Hilbert space and $\Phi_S$ the corresponding Segal quantization. Then:
\begin{description}
\item[(a)] For each $f\in\mathcal{H}$ the operator $\Phi_S(f)$ is essentially selfadjoint on $F_0$.
\item[(b)] The vacuum $\Omega_0$ is in the domain of all finite products $\Phi_S(f_1)\Phi_S(f_2)\dots\Phi_S(f_n)$ and the linear span of $\left\{\Phi_S(f_1)\Phi_S(f_2)\dots\Phi_S(f_n)\Omega_0\,|\,f_i\in\mathcal{H},n\ge0\right\}$ is dense in $\mathcal{F}_s(\mathcal{H})$.
\item[(c)] For each $\psi_0\in F_0$ and $f,g\in\mathcal{H}$
    \begin{equation}
    \begin{split}
     &\Phi_S(f)\Phi_S(g)\psi-\Phi_S(g)\Phi_S(f)\psi=\imath\text{Im}(f,g)\psi\\
     &\exp\left(\imath\Phi_S(f+g)\right)=\exp\left(-\frac{\imath}{2}\text{Im}(f,g)\right)\exp\left(\imath\Phi_S(f)\right)\exp\left(\imath\Phi_S(g)\right).
    \end{split}
    \end{equation}

\item[(d)]If $f_n\rightarrow f$ in $\mathcal{H}$, then:
    \begin{equation}
      \begin{split}
      \Phi_S(f_n)\rightarrow \Phi_S(f)\\
      \exp\left(\imath\Phi_S(f_n)\right)\rightarrow \exp\left(\imath\Phi_S(f)\right)\\
      \end{split}
    \end{equation}
\item[(e)]For every unitary operator $U$ on $\mathcal{H}$, $\Gamma(U):\mathcal{D}(\overline{\Phi_S(f)})\rightarrow\mathcal{D}(\overline{\Phi_S(Uf)})$ and for $\psi\in\mathcal{D}(\overline{\Phi_S(Uf)})$ and for all $f\in\mathcal{H}$
    \begin{equation}
    \Gamma(U)\overline{\Phi_S(f)}\Gamma(U)^{-1}\psi=\overline{\Phi_S(Uf)}\psi.
    \end{equation}
\end{description}
\end{theorem}
\begin{proof} See Theorem X.41 in \cite{RS75}.
\end{proof}

\subsubsection{Wick Products and Wiener-It\^{o}-Segal Isomorphism}

\begin{defi}[\textbf{Wick Products of Random Variables}]
Let $(\Omega,\mathcal{A},\mu)$ be a probability space and $f:\Omega\rightarrow\mathbf{R}$ a random variable with finite moments. Then, for $n\in\mathbf{N}_0$, the random variable $:f^n:$, termed \textbf{nth Wick power} of f is defined recursively by
\begin{equation}
\begin{split}
&:f^0:=1,\\
&\frac{\partial}{\partial f}:f^n:=n:f^{n-1}:\quad\text{ and }\quad\mathbb{E}^\mu[:f^n:]=0\quad\text{ for }n\ge1.
\end{split}
\end{equation}
Let $f_1,\dots,f_k:\Omega\rightarrow\mathbf{R}$ be random variables with finite moments. The Wick product $:f_1^{n_1}\dots f_k^{n_k}:$ is defined recursively in $n=n_1+\dots+n_k$ by
\begin{equation}
\begin{split}
&:f_1^0\dots f_k^0:=1,\\
&\frac{\partial}{\partial f_i}:f_1^{n_1}\dots f_k^{n_k}:=n_i:f_1^{n_1}\dots f_i^{n_i-1}\dots f_k^{n_k}:\quad\text{ and }\quad\mathbb{E}^\mu[:f_1^{n_1}\dots f_k^{n_k}:]=0\quad\text{ for }n\ge1.
\end{split}
\end{equation}
\end{defi}

\begin{defi}[\textbf{Wick Products of Segal Fields}]
Let $\mathcal{H}$ be a Hilbert space, and for any $f\in\mathcal{H}$ let $\Phi_S(f)$ be the Segal field on the bosonic Fock space $\mathcal{F}_s(\mathcal{H})$. Then, for $f,f_1,\dots,f_k\in\mathcal{H}$ the \textbf{Wick product} $:\Phi_S(f_1)\dots\Phi_S(f_k):$ is defined recursively by
\begin{equation}
\begin{split}
&:\Phi_S(f):=\Phi_S(f),\\
&:\Phi_S(f)\prod_{j=1}^k\Phi_S(f_j):=\Phi_S(f):\prod_{j=1}^k\Phi_S(f_j):-\frac{1}{2}\sum_{j=1}^k(f,f_j)_{\mathcal{H}}):\prod_{i\neq j}\Phi_S(f_j):
\end{split}
\end{equation}
\end{defi}

\begin{proposition}[\textbf{Wiener-It\^{o}-Segal Isomorphism}]\label{WIS}
If the probability measure $\mu$ on $\mathcal{S}^{\prime}(\mathbf{R}^N)$ is Gaussian, then the space $\mathcal{F}_s(\mathcal{H})$ is isomorphic to $L^2(\mathcal{S}^{\prime}(\mathbf{R}^N), d\mu)$ and the isomorphism $\theta_W:\mathcal{F}_s(\mathcal{H})\rightarrow L^2(\mathcal{S}^{\prime}(\mathbf{R}^N), d\mu)$, termed Wiener-It\^{o}-Segal isomorphism, satisfies the following properties:
\begin{enumerate}
\item[(i)]$\theta_W\Omega_0=1$,
\item[(ii)]$\theta_W\mathcal{H}_s^{(n)}=L^2_n(\mathcal{S}^{\prime}(\mathbf{R}^N), d\mu)$,
\item[(iii)]$\theta_W \Phi_S(f)\theta_W^{-1}\Phi=\Phi(f)$ for all $\Phi\in \mathcal{S}^{\prime}(\mathbf{R}^N)$ and $f\in \mathcal{S}(\mathbf{R}^N)$,
\end{enumerate}
where $\Omega_0=(1,0,0,\dots)\in\mathcal{F}_s(\mathcal{H})$ and $L^2_n(\mathcal{S}^{\prime}(\mathbf{R}^N), d\mu)$ is the closure of all linear combinations of Wick products of random variables over $\mathcal{S}^{\prime}(\mathbf{R}^N)$ up to order $n$.
\end{proposition}
\begin{proof} See Proposition 5.7 in \cite{BHL11}.
\end{proof}

\subsection{Osterwalder-Schrader Axioms}

Osterwalder and Schrader (see \cite{OS73}, \cite{OS73Bis}) utilized the Wick rotation technique to pass from the Minkowskian to the Euclidean space and formulate axioms equivalent to Wightman in terms of Euclidean Green functions. In \cite{OS75} they defined free Bose and Fermi fields and proved a Feynman-Ka\v{c} formula for boson-fermion models.\par
The dynamics of the quantized system satisfying Wightman axioms is given by the Schr\"odinger, or, equivalently by the heat equation,
where an $H$ unbounded, selfadjoint Hamilton operator $H$ on a physical Hilbert space $\mathcal{H}$ appears. The Hamilton operator can be extracted from the Osterwalder-Schrader axioms and ideally coincides with an operator obtained by a quantization procedure of the Hamiltonfunction in the classical description of the physical system.\par Following chapter 6 of \cite{GJ87} we introduce the
\begin{defi}[\textbf{Osterwalder-Schrader Axioms}] The primitive of a quantum field model is a Borel probability measure $d\mu$ on $\mathcal{S}^{\prime}(\mathbf{R}^N)$, whose inverse Fourier transform gives the generating functional
\begin{equation}
S(f):=\int_{\mathcal{S}^{\prime}(\mathbf{R}^N)}e^{\imath\Phi(f)}d\mu(\Phi),
\end{equation}
where $f\in\mathcal{S}(\mathbf{R}^N)$. Formally, we write $\Phi(f)=\int_{\mathbf{R}^N}\Phi(x)f(x)d^Nx$.
\begin{description}
\item[(OS0) Analyticity:] The functional $S(f)$ is entire analytic. For every finite set of test functions $f_j\in\mathcal{S}(\mathbf{R}^N)$, $j=1,\dots, n$ and complex numbers $z:=(z_1,z_2,\dots,z_n)\in\mathbf{C}^n$, the function
    \begin{equation}
    z\mapsto S\left(\sum_{j=1}^nz_jf_j\right)
    \end{equation}
\noindent is entire on $\mathbf{C}^n$.
\item[(OS1) Regularity:] For some $p\in[1,2]$, some constant $c$ and for all $f\in\mathcal{S}(\mathbf{R}^N)$
\begin{equation}
|S(f)|\le e^{c(\|f\|_{L^1(\mathbf{R}^N)}+\|f\|_{L^p(\mathbf{R}^N)}^p)}.
\end{equation}
\item[(OS2) Euclidean Invariance:] The functional $S(f)$ is invariant under Euclidean symmetries $E$ of $\mathbf{R}^N$. This means that for any translation, rotation and reflection, for all $f\in\mathcal{S}(\mathbf{R}^N)$
    \begin{equation}
     S(Ef)=S(f),
    \end{equation}
where $Ef(x):=f(E^{-1}(x))$.
\item[(OS3) Reflection Positivity:] Let
\begin{equation}
\mathcal{L}:=\left\{\psi\left|\;\psi(\Phi)=\sum_{j=1}^nc_je^{\Phi(f_j)}, c_j\in\mathbf{C}, f_j\in\mathcal{S}(\mathbf{R}^N), j=1,\dots,n\right.\right\}
\end{equation}
be the algebra of exponential functionals on tempered distributions and
\begin{equation}
\mathcal{L}_+:=\left\{\psi\left|\;\psi(\Phi)=\sum_{j=1}^nc_je^{\Phi(f_j)}, c_j\in\mathbf{C}, f_j\in\mathcal{S}(\mathbf{R}^N), \supp(f_j)\subset\{(t,x)\in\mathbf{R}^N\,|\;t>0\}\right.\right\}
\end{equation}
the subalgebra of exponential functionals whose defining functions are supported in the positive time half space. Euclidean transforms E on $\mathbf{R}^N$ act on $\mathcal{S}^{\prime}(\mathbf{R}^N)$ via
\begin{equation}
(E\psi)(\Phi):=\psi(E\Phi)\text{ and } E\Phi(f):=\Phi(Ef),
\end{equation}
for all $\psi\in\mathcal{L}$, $\Phi\in\mathcal{S}^{\prime}(\mathbf{R}^N)$, $f\in\mathcal{S}(\mathbf{R}^N)$.\par
We assume that the time reflection
\begin{equation}
(t,x)\mapsto\theta(t,x):=(-t,x)
\end{equation}
satisfies
\begin{equation}
\int_{\mathcal{S}^{\prime}(\mathbf{R}^N)}\theta\Psi(\Phi)\bar{\Psi}(\Phi)d\mu(\Phi)\ge0.
\end{equation}
for all $\Psi\in\mathcal{L}_+$.
\item[(OS4) Ergodicity:] The Euclidean time translation subgroup $\{T(t)\}_{t\ge0}$ acts ergodically on the measure space $(\mathcal{S}^{\prime}(\mathbf{R}^N), d\mu)$, i.e.
\begin{equation}
\lim_{t\mapsto+\infty}\frac{1}{t}\int_0^tT(s)\Psi(\Phi)T(s)^{-1}ds=\int_{\mathcal{S}^{\prime}(\mathbf{R}^N)}\Psi(\Phi)d\mu(\Phi),
\end{equation}
for all $\Psi\in L^1(\mathcal{S}^{\prime}(\mathbf{R}^N), d\mu)$.
\end{description}
\end{defi}
\noindent From the Osterwalder-Schrader axioms we can reconstruct quantum field theory as described by the Wightman axioms and derive the quantum mechanical dynamics. Table \ref{OS} depicts the formal scheme of this construction.
\begin{table}[!]
\begin{center}
\begin{tabular}{|l|l|}
  \hline
  &\\
  Path space  & $\mathcal{S}^{\prime}(\mathbf{R}^{N})$ \\
  &\\
  \hline
  &\\
  Configuration space  & $\mathcal{S}^{\prime}(\mathbf{R}^{N-1})$ \\
  &\\
  \hline
  &\\
  Measure on path space  & $d\mu$\\
  &\\
  \hline
  &\\
  Measure on configuration space  & $d\nu=d\mu|_{t:=0}$\\
  &\\
  \hline
  &\\
  Path space for quantum operators  & $\mathcal{E}:=L^2(\mathcal{S}^{\prime}(\mathbf{R}^{N}), d\mu)$\\
  &\\
  \hline
   &\\
  Physical Hilbert space  & $\mathcal{H}\cong L^2(\mathcal{S}^{\prime}(\mathbf{R}^{N-1}), d\nu)$\\
  &\\
  \hline
\end{tabular}
\end{center}
\caption{Quantum field theory from Osterwalder-Schrader axioms}\label{OS}
\end{table}
The proper definition of the quantum mechanical Hilbert space relies on the reflection positivity axiom. The exponential functionals $\mathcal{L}$ are dense in $\mathcal{E}:=L^2(\mathcal{S}^{\prime}(\mathbf{R}^{N}), d\mu)$. Let $\mathcal{E}_+$ be the closure of $\mathcal{L}_+$ in $\mathcal{E}$, termed as \textbf{positive time subspace} of $\mathcal{E}$, and define the bilinear form on $\mathcal{E}_+\times\mathcal{E}_+$
\begin{equation}\label{scalar}
b(\Psi, \Upsilon):=\int_{\mathcal{S}^{\prime}(\mathbf{R}^N)}\Psi(\Phi)\overline{\Upsilon(\Phi)}d\mu(\Phi).
\end{equation}
By $(OS3)$ the bilinear form $b$ is positive semidefinite. Let
\begin{equation}
\mathcal{N}:=\{\psi\in\mathcal{E}_+\,|\,b(\Psi,\Psi)=0\}
\end{equation}
\noindent the subspace of vectors for which the bilinear form degenerates, and define the quantum mechanical Hilbert space as
\begin{equation}
\mathcal{H}:=\overline{\mathcal{E}_+/\mathcal{N}},
\end{equation}
with hermitian scalar product defined by (\ref{scalar}) independently of the representative of the equivalence class. We denote the canonical embedding of the positive time subspace unto the quantum mechanical Hilbert space as
\begin{equation}
\begin{split}
\curlywedge:&\mathcal{E}_+\rightarrow\mathcal{H}\\
&\Psi\mapsto\Psi^{\curlywedge}:=\Psi+\mathcal{N},
\end{split}
\end{equation}
and transfer operators $S$ acting on $\mathcal{E}_+$ to operators $S^\curlywedge$ acting on $\mathcal{H}$ by
\begin{equation}
S^\curlywedge\Psi^{\curlywedge}:=(S\Psi)^{\curlywedge},
\end{equation}
which is well defined for all $\Psi\in\mathcal{E}_+$, if
\begin{equation}\label{condS}
S:\mathcal{D}(S)\cap\mathcal{E}_+\rightarrow\mathcal{E}_+\text{ and }S:\mathcal{D}(S)\cap\mathcal{N}\rightarrow\mathcal{N}.
\end{equation}
\begin{theorem}[Reconstruction of quantum mechanics] \label{rec}
If the probability measure $\mu$ on $\mathcal{S}^{\prime}(\mathbf{R}^N)$ satisfies the reflection and time translation invariance axiom (OS2),
and the reflection positivity axiom (OS3), then for all $t\ge0$ the time translation $T(t)$ satisfies (\ref{condS}) and
\begin{equation}
T(t)^\curlywedge=e^{-tH},
\end{equation}
where $H=H^*\ge0$ is a (possibly unbounded) selfadjoint operator on $\mathcal{H}$ with ground state $\Omega_0:=1^\curlywedge$, i.e. $H\Omega_0=0$.
\end{theorem}
\begin{proof} See Theorem 6.13 in \cite{GJ87}.
\end{proof}
\noindent The reflection positivity axiom is not easy to verify. The following proposition provides a useful criterion.
\begin{proposition}\label{posprop}
The probability measure $\mu$ on $\mathcal{S}^{\prime}(\mathbf{R}^N)$ satisfies
 the reflection positivity axiom (OS3) if and only if the matrix $M:=[S(\theta(f_i)-f_j)]$ has positive eigenvalues for all choices
 of $(f_i)_{i=1,\dots,n}\subset\mathcal{S}(\mathbf{R}^N)$ with support in the time positive half space.
\end{proposition}
\begin{proof}
See Corollary 3.4.4 in \cite{GJ87}.
\end{proof}
\noindent Gaussian measures play a prominent role in quantum field theory, because they originate free fields.
\begin{defi} A linear operator $C$ defined on $\mathbf{R}^N$ satisfies the reflection positivity property if and only if
\begin{equation}
(\theta f,Cf)_{L^2(\mathbf{R}^N, d^Nx)}\ge0,
\end{equation}
for all $f\in\mathbf{R}^N$ supported at positive times.
\end{defi}
\begin{theorem}
A Gaussian measure on the space of tempered distributions satisfies reflection positivity if and only if its covariance operator does.
\end{theorem}
\begin{proof}  See Theorem 6.22 in \cite{GJ87}.
\end{proof}
\begin{proposition}\label{RP}
The covariance operator $C:=(-\Delta_{\mathbf{R}^N}+m^2)^{-1}$ is reflection positive for all $m^2>0$.
\end{proposition}
\begin{proof}  See Proposition 6.2.5 in \cite{GJ87}.
\end{proof}

\noindent Finally, we highlight the equivalence of the Osterwalder-Schrader axioms with Wightman's ones.
\begin{theorem}
Wightman's axioms (W1)-(W8) are equivalent to Osterwalder-Scharder axioms (OS0)-(OS4).
\end{theorem}
\begin{proof}
See Theorem II.12 and Theorem II.13 in \cite{Sim15} or Theorem 6.1.5 and Chapter 19 in \cite{GJ87}.
\end{proof}
\section{Quantization of Yang-Mills Equations and Positive Mass Gap}\label{QYM}
There are several methods of designing a quantum theory for non-Abelian gauge fields. The Hamiltonian formulation is the approach used in the original work by Yang-Mills (\cite{MY54}), which was later abandoned in favour of an alternative method based on Feynman path integrals (\cite{FP67}). When it became clear that the Faddeev-Popov method must be incomplete beyond perturbation theory, Hamiltonian formulation enjoyed a partial renaissance. More recent, didactically accessible, examples of the Hamiltonian approach in the physical literature can be found in \cite{Sch08}.\par
In the first section of this chapter we construct a quantized Yang-Mills theory in dimension $3+1$ and in the second we compute spectral lower bounds for the Hamilton operator. In the third we summarize our findings and prove the main result, Theorem \ref{CMIThm}.


\subsection{Quantization}\label{quant}
In the Yang-Mills $3+1$ dimensional set up,  in order to account for functionals on transversal fields as required by the Coulomb gauge, we introduce the configuration space $\mathcal{S}^{\prime}_{\bot}(\mathbf{R}^4,\mathbf{C}^{K \times 3})$ and the path space $\mathcal{S}^{\prime}_{\bot}(\mathbf{R}^3,\mathbf{C}^{K \times 3})$. These are the duals in the sense of nuclear spaces of the test functions satisfying the transversal condition:
\begin{equation}
\begin{split}
L^2_{\bot}(\mathbf{R}^3,\mathbf{C}^{K \times 3},d^3x)&:=\{\left.\mathbf{A}\in L^2(\mathbf{R}^3,\mathbf{C}^{K \times 3},d^3x)\right|\,\mathbf{A}^{\parallel}=0\},\\
\mathcal{S}_{\bot}(\mathbf{R}^3,\mathbf{C}^{K \times 3})&:=\mathcal{S}(\mathbf{R}^3,\mathbf{C}^{K \times 3})\cap L^2_{\bot}(\mathbf{R}^3,\mathbf{C}^{K \times 3},d^3x)\\
\mathcal{S}_{\bot}(\mathbf{R}^4,\mathbf{C}^{K \times 3})&:=\{f\in\mathcal{S}(\mathbf{R}^4,\mathbf{C}^{K \times 3})\left|\;f(t,\cdot)\in L^2_{\bot}(\mathbf{R}^3,\mathbf{C}^{K \times 3},d^3x)\text{ for all }t\in\mathbf{R}\right.\}
\end{split}
\end{equation}
Note that we do not need to bother about the time component of the connection, because it vanishes in the Coulomb gauge. The tempered distributions are defined
\begin{equation}
\begin{split}
\mathcal{S}_{\bot}^{\prime}(\mathbf{R}^3,\mathbf{C}^{K \times 3})&:=\left\{\mathbf{A}:\mathcal{S}_{\bot}(\mathbf{R}^3,\mathbf{C}^{K \times 3})\rightarrow \mathbf{C}^{K \times 3}\text{ linear and continuous}\right\}\\
\mathcal{S}_{\bot}^{\prime}(\mathbf{R}^4,\mathbf{C}^{K \times 3})&:=\left\{\mathbf{A}:\mathcal{S}_{\bot}(\mathbf{R}^3,\mathbf{C}^{K \times 4})\rightarrow \mathbf{C}^{K \times 3}\text{ linear and continuous}\right\},
\end{split}
\end{equation}
and $\mathbf{A}\in\mathcal{S}_{\bot}^{\prime}(\mathbf{R}^4,\mathbf{C}^{K \times 3})$ is called \textbf{regular}
if there exists $\mathbf{a}=\mathbf{a}(t,x)\in L^2_{\text{loc}}(\mathbf{R}^4,\mathbf{C}^{K\times 3})$ such that
for all $f\in\mathcal{S}_{\bot}(\mathbf{R}^4,\mathbf{C}^{K \times 3})$
\begin{equation}
\mathbf{A}(f)=\int_{\mathbf{R}^4}d^4(t,x)a(x).f(x),
\end{equation}
where the dot denotes the pointwise multiplication.\par
We define the Hilbert space $\mathcal{E}:=L^2(\mathcal{S}^{\prime}_{\bot}(\mathbf{R}^4,\mathbf{C}^{K \times 3}), d\mu)$
and the physical Hilbert space as $\mathcal{H}:=L^2(\mathcal{S}^{\prime}_{\bot}(\mathbf{R}^3,\mathbf{C}^{K \times 3}), d\nu)$
for appropriate probability measures $\mu$ and $\nu$.\\
We introduce the Hamilton operator originated by the quantization of the Hamiltonian formulation of Yang-Mills equations.
This operator is the infinitesimal generator of a time inhomogeneous It\^{o}'s diffusion, whose probability density solves the heat kernel equation.
We construct a probability measure on the tempered distributions using this It\^{o}'s process as integrator and utilizing the Feynman-Ka\v{c} formula.
We prove that the necessary Osterwalder-Schrader axioms for the Hamilton operator to be selfadjoint on the probability space of the time zero
 tempered distributions are fulfilled. Then, we verify that the Wightman axioms are satisfied.\par
When we try to introduce the canonical quantization for the Hamilton function $H$ in (\ref{classicHamiltonFunction}),
we face the problem for $H_{II}$ and $V$ that the classical fields $\mathbf{A}(t,x)$ cannot be directly interpreted
as multiplication operators in $L^2(\mathcal{S}^{\prime}_{\bot}(\mathbf{R}^3,\mathbf{C}^{K \times 3}), d\nu)$,
because the multiplication of distributions cannot be properly defined. To circumvent this issue, following the idea expressed
f.i. in \cite{Sim05} (page 257) and applied to Nelson's model in \cite{BHL11} (page 297) in the case of ultraviolet divergence,
we introduce a cutoff test function to regularize fields.
\begin{defi}[\textbf{Regularization and Ultraviolet Cutoff}]\label{moll} Let us consider a test function $\varphi\in\mathcal{S}(\mathbf{R}^4,\mathbf{R}^1)$ such that
$\varphi_t(x):=\varphi(t,x)\in\mathcal{S}(\mathbf{R}^3,\mathbf{R}^1)$ satisfies for all $t$
\begin{equation}
\begin{split}
&\varphi_t(x)>0\text{  for all }x\in\mathbf{R}^3\\
&\int_{\mathbf{R}^3}\varphi_t(x)=1\\
&\supp(\varphi_t)\subset\subset \mathbf{R}^3.
\end{split}
\end{equation}
The test function $\varphi^{\Lambda}\in\mathcal{S}(\mathbf{R}^4,\mathbf{R}^1)$ defined as
$\varphi^{\Lambda}(t,x):=\Lambda\varphi_t(\Lambda x)$
is called \textbf{mollifier} for the ultraviolet cutoff level $\Lambda\ge0$.
\end{defi}
\noindent Note that for all $t\in\mathbf{R}$ in the limit we have $\mathcal{S}^{\prime}-\lim_{\Lambda\rightarrow+\infty}\varphi_t^{\Lambda}=\delta\in\mathcal{S}^{\prime}(\mathbf{R}^3,\mathbf{R}^1)$
and that $\mathbf{A}(\varphi^{\Lambda}_t(\cdot-x)I^{K\times 3})$ is a polynomially bounded function in $x\in\mathbf{R}^3$. By introducing the notation
\begin{equation}
 \mathbf{A}(\varphi^{\Lambda}_t(\cdot-x)) :=\mathbf{A}(\varphi^{\Lambda}_t(\cdot-x)I^{K\times 3}),
\end{equation}
 Theorem \ref{CHam} can be now quantized as follows.
\begin{proposition}\label{quantizationprop}
Let $H=H(A,E)$ be the Hamilton function of the Hamilton equations equivalent to the Yang-Mills equations as in Theorem \ref{CHam} for a simple Lie-group
as structure group. Let us consider a $\Lambda\ge0$ and the mollifier function
$\varphi^{\Lambda}\in\mathcal{S}(\mathbf{R}^4,\mathbf{R}^1)$.
For a probability measure $\nu$ on $\mathcal{S}^{\prime}_{\bot}(\mathbf{R}^3,\mathbf{C}^{K \times 3})$ the canonical
quantization of the position variable $A$, of the momentum variable $E$  and of the Hamilton function $H$ means the following substitution:
\begin{equation}\label{quantization}
\begin{split}
\mathbf{A} \in C^{\infty}(\mathbf{R}^3,\mathbf{C}^{K \times 3})&\longrightarrow \mathbf{A}(\varphi^{\Lambda}_t(\cdot-\cdot))\in\mathcal{O}(L^2(\mathcal{S}^{\prime}_{\bot}(\mathbf{R}^3,\mathbf{C}^{K \times 3}),\mathbf{C}^{K \times 3},d\nu))\\
\mathbf{E} \in C^{\infty}(\mathbf{R}^3,\mathbf{C}^{K \times 3})&\longrightarrow \frac{1}{\imath}\frac{\delta}{\delta \mathbf{A}(\varphi^{\Lambda}_t(\cdot-\cdot))}\in\mathcal{O}(L^2(\mathcal{S}^{\prime}_{\bot}(\mathbf{R}^3,\mathbf{C}^{K \times 3}),\mathbf{C}^{K \times 3},d\nu))\\
H \in C^{\infty}(\mathbf{R}^{2(K\times 3)},\mathbf{R})&\longrightarrow H^{\Lambda}:=H\left(\mathbf{A}(\varphi^{\Lambda}_t(\cdot-\cdot)),\frac{1}{\imath}\frac{\delta}{\delta \mathbf{A}(\varphi^{\Lambda}_t(\cdot-\cdot))}\right)\in\mathcal{O}(L^2(\mathcal{S}^{\prime}_{\bot}(\mathbf{R}^3,\mathbf{C}^{K \times 3}),\mathbf{C},d\nu)),\\
\end{split}
\end{equation}
\noindent where $\mathcal{O}(\cdot)$ denotes the set of all linear operators over the corresponding underlying vector space.
The cutoff Hamilton operator $H^{\Lambda}$ for the quantized Yang-Mills equations reads
\begin{equation}
H^{\Lambda,g}=H_I^{\Lambda}+H_{II}^{\Lambda,g}+V^{\Lambda,g}-V_0^{\Lambda,g},
\end{equation}
where
\begin{equation}\label{H-V}
\begin{split}
&(H_I^{\Lambda}\Psi)(\mathbf{A})=-\frac{1}{2}\int_{\mathbf{R}^3}d^3x\left[\frac{\delta}{\delta A_i^a(\varphi^{\Lambda}_t(\cdot-x))}\right]^2\Psi(\mathbf{A})\\
& \\
&(H_{II}^{\Lambda,g}\Psi)(\mathbf{A})=-\frac{g^2}{2}\int_{\mathbf{R}^3}d^3x\left[\int_{\mathbf{R}^3}d^3y\partial_iG^{a,b}(\mathbf{A}(\varphi^{\Lambda}_t(\cdot-\cdot));x,y)\varepsilon^{b,c,d}A_k^d(\varphi^{\Lambda}_t(\cdot-y))\frac{\delta}{\delta A_k^c(\varphi^{\Lambda}_t(\cdot-y))}\right]^2\Psi(\mathbf{A})\\
& \\
&(V^{\Lambda,g}\Psi)(\mathbf{A})=\int_{\mathbf{R}^3}d^3x\,V^{\Lambda,g}(t,x,\mathbf{A})\Psi(\mathbf{A}),\\
&\\
&V^{\Lambda,g}(t,x,\mathbf{A}):=\frac{1}{16}\varepsilon_i^{j,k}\varepsilon_i^{p,q}\left\{[\partial_jA^a_k(\varphi^{\Lambda}_t(\cdot-x))-\partial_kA_j^a(\varphi^{\Lambda}_t(\cdot-x))+\right.
g\varepsilon^{a,b,c}A_j^b(\varphi^{\Lambda}_t(\cdot-x))A_k^c(\varphi^{\Lambda}_t(\cdot-x))]\cdot\\
&\qquad\qquad\qquad\cdot[\partial_pA^a_q(\varphi^{\Lambda}_t(\cdot-x))-\partial_qA_p^a(\varphi^{\Lambda}_t(\cdot-x))
\left.+g\varepsilon^{a,b,c}A_p^b(\varphi^{\Lambda}_t(\cdot-x))A_q^c(\varphi^{\Lambda}_t(\cdot-x))]\right\},\\
&\\
&V^{\Lambda,g}_0:\text{ a real constant which will be chosen later},
\end{split}
\end{equation}
with the domain of definition
\begin{equation}
\mathcal{D}(H^{\Lambda,g}):=\left\{\Psi\in L^2(\mathcal{S}^{\prime}_{\bot}(\mathbf{R}^3,\mathbf{C}^{K \times 3}),\mathbf{C},d\nu)\left|\,H^{\Lambda}\Psi\in L^2(\mathcal{S}^{\prime}_{\bot}(\mathbf{R}^3,\mathbf{C}^{K \times 3}),\mathbf{C},d\nu)\right.\right\}.
\end{equation}
Note that the time $t$ is considered as a parameter in the expressions (\ref{quantization}) and (\ref{H-V}).
\end{proposition}
\begin{rem}
The first dot in $\partial_iG^{a,b}(\mathbf{A}(\varphi^{\Lambda}_t(\cdot-\cdot))$ in the expression for $H_{II}^{\Lambda,g}$ in (\ref{H-V}) refers to the application of the distribution $\mathbf{A}$ on the test function
i.e. the integration variable if $\mathbf{A}$ is regular. The second is the generic variable sign, since the modified Green function $G$ is a functional of a function in the first argument,
as explained in Proposition \ref{ModGreen}.
\end{rem}
\begin{proof}[Proof of Proposition \ref{quantizationprop}]
It is a straightforward consequence of Corollary \ref{corClassicH} where we introduce the quantization specified by (\ref{quantization}).
\end{proof}
\begin{rem}
Later, we will prove that $\Omega_0\in\mathcal{D}(H^{\Lambda,g})$ is the \textbf{ground state} of $H^{\Lambda,g}$,
which is an eigenvector for the eigenvalue $0$ with simple multiplicity, unique up to multiplication with a constant.
\end{rem}
\begin{rem}
The derivative
\begin{equation}
\frac{\delta}{\delta A_i^a(\varphi^{\Lambda}_t(\cdot-x))}\Psi(\mathbf{A})=\left.\frac{d}{d\varepsilon}\right|_{\varepsilon:=0}\Psi(\mathbf{A}+\varepsilon\varphi^{\Lambda}_t(\cdot-x)))
\longrightarrow\frac{\delta}{\delta A_i^a(t,x)}\Psi(\mathbf{A})\;(\Lambda\rightarrow+\infty)\end{equation}
 in line with fact that $A_i^a(\varphi^{\Lambda}_t(\cdot-x))$ tends to $A_i^a(t,x)$, because $\mathcal{S}^{\prime}-\lim_{\Lambda\rightarrow+\infty}\varphi^{\Lambda}_t(\cdot-x))=\delta(\cdot-x)$
 for all $t$.
 \end{rem}

\begin{rem} The operator $H_I^{\Lambda}$  tends for $\Lambda\rightarrow +\infty$ to the Laplace operator in infinite dimensions for functionals
of the potential fields. The operator $V^{\Lambda}$ is a multiplication operator corresponding to the fibrewise multiplication with the square
of the connection curvature. Both operators $H_I^{\Lambda}$ and $V^{\Lambda}$ do not vanish if $g=0$ and do not contribute to the existence
of a mass gap. The operator $H_{II}^{\Lambda}$ vanishes if $g=0$, f.i. when $G$ is an abelian groups, and, as we will see,
is responsible for the existence of a mass gap.
\end{rem}

\begin{rem}
In the physical literature (see f.i. \cite{Sch08}) the operator $C=C(\mathbf{A};x,y)$
\begin{equation}
C^{a,b}(\mathbf{A};x,y):=-\int_{\mathbf{R}^3}d^3\bar{y}G^{a,c}(\mathbf{A};x,\bar{y})\Delta G^{c,b}(\mathbf{A};y,\bar{y})
\end{equation}
is termed \textit{Coulomb operator} and $G=G(\mathbf{A};x,y)$ is also called the \textit{Faddeev-Popov operator}. Note that they make sense only after the substitution
\begin{equation}
\mathbf{A}(t,x)\rightarrow\mathbf{A}(\varphi^{\Lambda}_t(\cdot-x)),
\end{equation}
and passing to the limit $\Lambda\rightarrow+\infty$.
\end{rem}
\noindent A direct computation shows
\begin{proposition}[\textbf{Canonical Commutation Relations}]
The operators $Q_i^j(x)$ and $P_i^j(x)$ defined as
\begin{equation}
\begin{split}
Q_i^a(x)\Psi(\mathbf{A})&:=A_i ^a(\varphi^{\Lambda}_t(x-\cdot))\Psi(\mathbf{A})\\
P_i^a(x)\Psi(\mathbf{A})&:=\frac{1}{\imath}\frac{\delta}{\delta A_i ^a(\varphi^{\Lambda}_t(x-\cdot))}\Psi(\mathbf{A}),
\end{split}
\end{equation}
on the appropriate domains, are in $\mathcal{O}(L^2(\mathcal{S}^{\prime}_{\bot}(\mathbf{R}^3,\mathbf{C}^{K \times 3}),d\nu^{\Lambda,g}))$
for all $i,a,x$ and all measures $\nu^{\Lambda, g}$. Their commutators satisfy the equations
\begin{equation}
[Q_i^a(x),Q_j^b(y)]=0\qquad[P_i^a(x),P_j^b(y)]=0\qquad[P_i^a(x),Q_j^b(y)]=\frac{1}{\imath}\delta^{a,b}\delta_{i,j}\delta_{x,y},
\end{equation}
for all $i,j,a,b,x,y$.
\end{proposition}

\subsection{Construction of a Complete Set of Generalized Eigenvectors}
\begin{theorem}\label{pgs} For $\Lambda\ge0$ big enough there exists a probability measure $\nu^{\Lambda,g}$ such that $H^{\Lambda,g}$ is a selfadjoint operator on
$L^2(\mathcal{S}^{\prime}_{\bot}(\mathbf{R}^3,\mathbf{C}^{K \times 3}),\mathbf{C},d\nu^{\Lambda,g})$ for all real constants $V_0^{\Lambda,g}$
There exists a \textbf{ground state} $\Omega_0^{\Lambda,g}\in\mathcal{D}(\mathcal{H}^{\Lambda,g})$ of $H^{\Lambda,g}$ for the eigenvalue $0$
for one choice of the real constant $V_0^{\Lambda,g}$:
\begin{equation}\label{gs}
H^{\Lambda,g}\Omega_0^{\Lambda,g}=0.
\end{equation}
The ground state $\Omega_0^{\Lambda,g}$ is an eigenvector with simple multiplicity, and hence it is unique up to multiplication with a constant.
\end{theorem}
The proof of Theorem \ref{pgs} is an elaborated functional analytic construction of a probability measure making
the Hamilton operator selfadjoint and with a unique ground state.\par
On a finite dimensional vector space an operator possessing a basis of eigenvectors for real eigenvalues is selfadjoint only with respect
to the scalar product which makes that basis orthonormal. For an infinite  space the situation is similar but more complicated by the presence of generalized
eigenvalues which are not proper eigenvalues, i.e. by the continuous spectrum.
\begin{proposition}\label{generalized_ev}
Let $\nu_0$ be the standard Gaussian probability on $\mathcal{S}_{\bot}^{\prime}(\mathbf{R}^3,\mathbf{C}^{K \times 3})$. In the notation of Example \ref{ExKT}
\begin{equation}
\begin{split}
&\mathcal{E}(L^2(\mathcal{S}^{\prime}_{\bot}(\mathbf{R}^3,\mathbf{C}^{K \times 3}),\mathbf{C},d\nu_0)):=\left\{\varphi\in L^2(\mathcal{S}^{\prime}_{\bot}(\mathbf{R}^3,\mathbf{C}^{K \times 3}),\mathbf{C},d\nu_0)|\,\|\varphi\|_k<+\infty\right\}\\
&\mathcal{E}^{\prime}(L^2(\mathcal{S}^{\prime}_{\bot}(\mathbf{R}^3,\mathbf{C}^{K \times 3}),\mathbf{C},d\nu_0)): \text{ dual space of }(\mathcal{S}_{\bot}(\mathbf{R}^3,\mathbf{C}^{K \times 3},\mathbf{C},d\nu_0)).
\end{split}
\end{equation}
Then, we have a rigged Hilbert space
\begin{equation}\label{riggedE}
\mathcal{E}(L^2(\mathcal{S}^{\prime}_{\bot}(\mathbf{R}^3,\mathbf{C}^{K \times 3}),\mathbf{C},d\nu_0))\subset L^2(\mathcal{S}^{\prime}_{\bot}(\mathbf{R}^3,\mathbf{C}^{K \times 3}),\mathbf{C},d\nu_0)\subset\mathcal{E}^{\prime}(L^2(\mathcal{S}^{\prime}_{\bot}(\mathbf{R}^3,\mathbf{C}^{K \times 3}),\mathbf{C},d\nu_0))
\end{equation}
and the operators $H^{\Lambda}_I$, $H^{\Lambda, g}_{II}$, $V^{\Lambda, g}$ have a complete set of generalized eigenvectors in $\mathcal{E}^{\prime}(L^2(\mathcal{S}^{\prime}_{\bot}(\mathbf{R}^3,\mathbf{C}^{K \times 3}),\mathbf{C},d\nu_0))$
for real non-negative generalized eigenvalues. The operator $H^{\Lambda, g}$ is selfadjoint on $L^2(\mathcal{S}^{\prime}_{\bot}(\mathbf{R}^3,\mathbf{C}^{K \times 3}),\mathbf{C},d\nu_0)$ with a non negative spectrum for $V_0^{\Lambda,g}$ sufficiently small
and $\Lambda$ sufficiently big.
\end{proposition}
\noindent To prove this proposition we need to prove some intermediate Lemmata beforehand.

\begin{lemma}\label{DirichletB}
Let $R>0$ be a positive constant and let
\begin{equation}
D:=\sum_{k=1}^Nf_k(A)\partial_{A_k}
\end{equation}
be a first order PDO on $\mathbf{R}^N$ for given functions $f_1,f_2,\dots,f_N$ on $\mathbf{R}^N$. If there exists a diffeomeorphism $B$ mapping some compact subset of $\mathbf{R}^N$ (in the "A"-space) to $[\frac{R}{2},+\frac{R}{2}]^N$ (in the "B"-space), such that
\begin{equation}\label{gjs}
g_j(B):=\sum_{k=1}^Nf_k(A)\partial_{A_k}B_j
\end{equation}
depends only on $B_j$, so that $g_j(B)=g_j(B_j)$, then the Dirichlet eigenvalues of $D^2$ on $B^{-1}([\frac{R}{2},+\frac{R}{2}]^N)$ are
\begin{equation}\label{eigB}
-\sum_{j=1}^N\frac{\pi^2k_j^2}{\left[\int_{-\frac{R}{2}}^{+\frac{R}{2}}dB_j\,g_j(B_j)^{-1}\right]^2},
\end{equation}
where $k_j\in\mathbf{Z}^*$ for $j=1,\dots,N$, provided the integrals in the denominators of (\ref{eigB}) exists.
\end{lemma}
\noindent Lemma \ref{DirichletB} is proved by a direct computation utilizing second order ODE.
\begin{lemma}\label{Diffeo}
Let us assume that for all $k=1,\dots,N$
\begin{equation}
\int_{-\infty}^{+\infty}dA_k\,f_k(A)^{-1}<+\infty
\end{equation}
holds uniformly in $A$. Then, the function $B:\mathbf{R}^N\rightarrow\mathbf{R}^N$ defined as
\begin{equation}\label{diffeoB}
B_j:=\Phi^{-1}\left(L_j\left[\int_{-\infty}^{A_j}d\bar{A}_jf_j^{-1}(A_1,\dots,\bar{A}_j,\dots,A_N)\;+\;K_j(A_1,\dots,A_{j-1},A_{j+1},\dots,A_N)\right]\right),
\end{equation}
\noindent where
\begin{equation}\label{defLK}
\begin{split}
&K_j(A_1,\dots,A_{j-1},A_{j+1},\dots,A_N):=-\inf_{A_j}\int_{-\infty}^{A_j}d\bar{A}_jf_j^{-1}(A_1,\dots,\bar{A}_j,\dots,A_N)>-\infty\\
&L_j:=\left[\sup_{A}\left[\int_{-\infty}^{A_j}d\bar{A}_jf_j^{-1}(A_1,\dots,\bar{A}_j,\dots,A_N)+K_j(A_1,\dots,A_{j-1},A_{j+1},\dots,A_N)\right]\right]^{-1}<+\infty\\
&\Phi(v):=\frac{1}{\sqrt{\pi}}\int_{-\infty}^{v}du\,e^{-u^2},
\end{split}
\end{equation}
is a diffeomorphism satisfying the assumptions of Lemma \ref{DirichletB} for any constant $R>0$, and the functions $g_j$ read
\begin{equation}
g_j(B_j)=\sqrt{\pi}\left(\sum_{k=1}^NL_{k}\right)e^{B_j^2}.
\end{equation}
\end{lemma}
\begin{proof}
By definition (\ref{gjs}), if
\begin{equation}
f_k\partial_{A_k}B_j=C_{k,j}(B_j)
\end{equation}
for a function $C_{k,j}$ of one variable, then the function $g_j$ explicitly depends on the variable $B_j$ only. This leads to the differential equation
\begin{equation}
\frac{dB_j}{C_{j,k}(B_j)}=\frac{dA_k}{f_k(A)},
\end{equation}
which is fulfilled if
\begin{equation}
\int_{-\infty}^{B_j}\frac{d\bar{B}_j}{C_{j,k}(\bar{B}_j)}=\int_{-\infty}^{A_k}\frac{d\bar{A}_k}{f_k(A_1,\dots,A_{k-1},\bar{A}_k,A_k,\dots,A_N)}+K_k(A_1,\dots,A_{k-1},A_{k+1},\dots,A_N),
\end{equation}
where $K_k$ is a function of $A$ not depending on $A_k$. The choice
\begin{equation}
C_{j,k}(u):=\sqrt{\pi}L_{j}e^{u^2}
\end{equation}
for $L_j$ and $K_j$ defined in (\ref{defLK}) leads to the desired result.
\end{proof}
\noindent Now we can compute the generalized spectral decomposition of the Hamilton operator $H^{\Lambda,g}$.
\begin{proof}[\text{Proof of Proposition \ref{generalized_ev}}]
The rigged Hilberst space statement (\ref{riggedE}) follows from the Kubo-Takenaka construction explained in Example \ref{ExKT}.
We will construct generalized eigenvectors for real non negative eigenvalues of $H^{\Lambda,g}$ which decomposes as
\begin{equation}
H^{\Lambda,g}=H_I^{\Lambda}+H_{II}^{\Lambda,g}+V^{\Lambda,g}-V_0^{\Lambda,g}.
\end{equation}
First, we analyze the operator $H_I^{\Lambda}$, which can be written as
\begin{equation}
\begin{split}
&H_I^{\Lambda}=U^{-1}H_IU\\
&H_I=-\frac{1}{2}\int_{\mathbf{R}^3}d^3 x\Delta_{\mathbf{R}^{3K}}\\
&U\Psi(\mathbf{A})=\Psi(\mathbf{A}(\varphi^{\Lambda}_t(\cdot-\cdot)))\text{ for }\mathbf{A}\in \mathcal{S}_{\bot}^{\prime}(\mathbf{R}^4,\mathbf{C}^{K \times 3}).
\end{split}
\end{equation}
Hence it suffices to construct a generalized eigenvector for $H_I$ and one follows for $H_I^{g}$ for the same generalized eigenvalue.
 Let  $x\in\mathbf{R}^3$ and $t\in\mathbf{R}$ now be fixed.
 For any $R>0$ the Laplace operator $\Delta$ on $[-\frac{R}{2},+\frac{R}{2}]^{3K}$ under Dirichlet boundary conditions
 has a discrete spectral resolution $(\lambda_k,\psi_k)_{k\ge0}$, where $\lambda_k=-\frac{\pi^2}{R^2}(k+1)$,
 and $\psi_k=\psi_k(\mathbf{A})\in C^{\infty}_0([-\frac{R}{2},+\frac{R}{2}]^{3K},\mathbf{C})$.
 We can extend $\psi_k$ outside the cube by setting its value to $0$, obtaining an approximated eigenvector for the approximated eigenvalue
 $\lambda_k$, which is in line with the fact that the Laplacian on $L^2(\mathbf{R}^{3K},\mathbf{C})$ has solely a continuous spectrum,
 which is $]-\infty,0]$. The functional on $\mathcal{S}^{\prime}_{\bot}(\mathbf{R}^3,\mathbf{C}^{K \times 3})$
\begin{equation}
\Psi^{\mathbf{A};x_0}_k(\bar{\mathbf{A}}):=\delta(\bar{\mathbf{A}}-\mathbf{A})\delta(x-x_0)\psi_k(\mathbf{A})
\end{equation}
for $x_0\in\mathbf{R}^3, k\in\mathbf{N}$ and $\mathbf{A}\in\mathbf{R}^{3K}$ is a generalized eigenvector in
$\mathcal{E}^{\prime}(\mathcal{S}^{\prime}_{\bot}(\mathbf{R}^3,\mathbf{C}^{K \times 3}),d\nu_0)$ for the operator $H_I$
on the rigged Hilbert space $L^2(\mathcal{S}^{\prime}_{\bot}(\mathbf{R}^3,\mathbf{C}^{K \times 3}),d\nu_0)$ for the generalized eigenvalue
$\frac{\pi^2}{R^2}(k+1)$.
By varying the generalized eigenvector over $x_0$, $\mathbf{A}$, $k$ and $R$, we obtain a complete set of generalized eigenvectors for
$\mathcal{E}^{\prime}(L^2(\mathcal{S}^{\prime}_{\bot}(\mathbf{R}^3,\mathbf{C}^{K \times 3}),\mathbf{C},d\nu_0))$, because, if for any
$\Phi\in \mathcal{E}(L^2(\mathcal{S}^{\prime}_{\bot}(\mathbf{R}^3,\mathbf{C}^{K \times 3}),\mathbf{C},d\nu_0))$
\begin{equation}
\int_{\mathcal{S}^{\prime}_{\bot}(\mathbf{R}^3,\mathbf{C}^{K \times 3})}
\Psi^{\mathbf{A};x_0}_k(\bar{\mathbf{A}})
\bar{\Phi}(\bar{\mathbf{A}})d\nu_0(\bar{\mathbf{A}})=0,
\end{equation}
that is
\begin{equation}
\delta(x-x_0)\psi_k(\mathbf{A})\bar{\Phi}(\mathbf{A})=0,
\end{equation}
which holds true for all $x_0\in\mathbf{R}^3$, $\mathbf{A}\in\mathbf{R}^{3K}$ and all $k\ge0$ iff $\Phi=0$.
By Corollary \ref{invGK} the operator $H_I$ on $L^2(\mathcal{S}^{\prime}_{\bot}(\mathbf{R}^3,\mathbf{C}^{K \times 3}),\mathbf{C},d\nu_0)$
is selfadjoint with non negative spectrum. The same holds true for $H^{\Lambda}_I$\\
Next, we analyze the operator $H_{II}^{\Lambda,g}$, which can be written as
\begin{equation}
\begin{split}
&H_{II}^{\Lambda,g}=U^{-1}H_{II}^gU\\
&H_{II}=-\frac{g^2}{2}\int_{\mathbf{R}^3}d^3x\left[\int_{\mathbf{R}^3}d^3y\,D_i^a(\mathbf{A};x,y)\right]^2\\
&U\Psi(\mathbf{A})=\Psi(\mathbf{A}(\varphi^{\Lambda}_t(\cdot-\cdot)))\text{ for }\mathbf{A}\in \mathcal{S}_{\bot}^{\prime}(\mathbf{R}^4,\mathbf{C}^{K \times 3}),
\end{split}
\end{equation}
\noindent
for the operator $D=D(\mathbf{A;x,y})$ defined  as
\begin{equation}
D_i^a(\mathbf{A};x,y):=\partial_iG^{a,b}(\mathbf{A}(t,y);x,y)\varepsilon^{b,c,d}A_k^d(t,y)\frac{\delta}{\delta A_k^c(t,y)}.
\end{equation}
Hence it suffices to construct a generalized eigenvector for $H_{II}^g$ and one follows for $H_{II}^{\Lambda,g}$ for the same generalized eigenvalue.
Let  $x_0,y_0\in\mathbf{R}^3$ now be fixed. We set
\begin{equation}
f_{i, k}^{a,c}(\mathbf{A};x_0,y_0):= \partial_iG^{a,b}(\mathbf{A}(t,y_0);x_0,y_0)\varepsilon^{b,c,d}A_k^d(t,y_0)
\end{equation}
and apply Lemma \ref{DirichletB} and Lemma \ref{Diffeo}. Assuming that for all indices $c,k$
\begin{equation}\label{bound}
\int_{-\infty}^{+\infty}dA_k^c\,f_{i, k}^{a,c}(\mathbf{A}(t,y_0);x_0,y_0)^{-1}<+\infty
\end{equation}
uniformly in $\mathbf{A}$, we can find a diffeomeorphism $\mathbf{B}:\mathbf{R}^{3K}\rightarrow \mathbf{R}^{3K}$ in the form of formula (\ref{diffeoB}), such that
for any $R>0$ the  operator $D_i^a(\mathbf{A}(t,y_0);x_0,y_0)^2$ on $B^{-1}([-\frac{R}{2},+\frac{R}{2}]^{3K})$ under Dirichlet boundary conditions
has a discrete spectral resolution $(\lambda_{i,s}^a(x_0,y_0),\psi_{i,s}^a(\mathbf{A}(t,y_0);x_0,y_0))_{s\ge0}$, where
\begin{equation}\label{lambdas}
\lambda_{i,s}^a(x_0,y_0)=-\sum_{j=1}^3\sum_{c=1}^K\frac{\pi^2k_{j,c,s}^2}{\left[\int_{-\frac{R}{2}}^{+\frac{R}{2}}dB_{j}^{c}\,g_{i, j}^{a, c}(B_{j}^{c};x_0,y_0)^{-1}\right]^2},
 \end{equation}
where $k_{j,c,s}\in\mathbf{Z}^*$ for all indices $s\in\mathbf{N}_0$, $j\in\{1,2,3\}$ and $c\in\{1,\dots, K\}$, and, by Lemma \ref{Diffeo} we defined
\begin{equation}
g_{i,j}^{a,l}(B_j^l;x_0,y_0):=\left(\sum_{k=1}^3\sum_{c=1}^KL_{i, k}^{a,c}(x_0,y_0)\right)e^{{B_j^l}^2}
\end{equation}
for
\begin{equation}\label{LiKi}
\begin{split}
&L_{i,j}^{a,c}(x_0,y_0)=\left[\sup_{\mathbf{A}}\left[\int_{-\infty}^{A_j^c}\,d\bar{A}_j^c\,f_{i, j}^{a,c}(\mathbf{A};x_0,y_0)^{-1}+K_{i,j}^{a,c}(\mathbf{A};x_0,y_0)\right]\right]^{-1}\\
&K_{i,j}^{a,c}(\mathbf{A}\,;x_0,y_0):=-\inf_{A_j^c}\int_{-\infty}^{A_j^c}d\bar{A}_j^c\,f_{i, j}^{a,c}(\mathbf{A};x_0,y_0)^{-1}.
\end{split}
\end{equation}
Note that for any $R>0$ the  operator $D_i^a(\mathbf{A}(t,y_0);x_0,y_0)$ on $B^{-1}([-\frac{R}{2},+\frac{R}{2}]^{3K})$
under Dirichlet boundary conditions has a discrete spectral resolution with the same eigenvectors as $D_i^a(\mathbf{A}(t,y_0);x_0,y_0)^2$ but other eigenvalues $(\zeta_{i,s}^a(x_0,y_0),\psi_{i,s}^a(\mathbf{A}(t,y_0);x_0,y_0))_{s\ge0}$, where
\begin{equation}\label{lambdas}
\zeta_{i,s}^a(x_0,y_0)=-\sum_{j=1}^3\sum_{c=1}^K\frac{\imath\pi k_{j,c,s}^2}{\int_{-\frac{R}{2}}^{+\frac{R}{2}}dB_{j}^{c}\,g_{i, j}^{a, c}(B_{j}^{c};x_0,y_0)^{-1}},
 \end{equation}
where $k_{j,c,s}\in\mathbf{Z}^*$ for all indices $s\in\mathbf{N}_0$, $j\in\{1,2,3\}$ and $c\in\{1,\dots, K\}$.\\
Since $B^{-1}([-\frac{R}{2},+\frac{R}{2}]^{3K})\uparrow\mathbf{R}^{3K}$ for $R\uparrow+\infty$, we can extend $\psi_{i,s}^a(\cdot;x_0,y_0)$ outside the cube by setting its value to $0$, obtaining an approximated eigenvector for the approximated eigenvalue $\lambda_{i,s}^a(x_0,y_0)$ for the operator $D_i^a(\mathbf{A};x_0,y_0)^2$ on $L^2(\mathbf{R}^{3K}, \mathbf{C})$, which means that $\lambda_{i,s}^a(x_0,y_0)\in\spec_c(D_i^a(\mathbf{A};x_0,y_0)^2)$. For fixed $i,s$ and $a$ the functional
\begin{equation}\label{basis}
\Psi^{a,x_0,y_0, \mathbf{A}}_{i,k}(\bar{\mathbf{A}}):=\delta(\bar{\mathbf{A}}-\mathbf{A})\delta(x-x_0)\delta(y-y_0)\delta(\bar{y}-y_0)\psi_{i,k}^a(\mathbf{A};x_0,y_0)
\end{equation}
is a generalized eigenvector in $\mathcal{E}^{\prime}(\mathcal{S}^{\prime}_{\bot}(\mathbf{R}^3,\mathbf{C}^{K \times 3}),d\nu_0)$ for the operator
\begin{equation}
\begin{split}
H_{i,II}^{a,g}&:=-\frac{g^2}{2}\int_{\mathbf{R}^3}d^3x\left[\int_{\mathbf{R}^3}d^3y\,D_i^a(\mathbf{A};x,y)\right]^2=\\
&=-\frac{g^2}{2}\int_{\mathbf{R}^3}d^3x\left[\int_{\mathbf{R}^3}d^3y\,D_i^a(\mathbf{A};x,y)\right]\left[\int_{\mathbf{R}^3}d^3\bar{y}\,D_i^a(\mathbf{A};x,\bar{y})\right]
\end{split}
\end{equation}
on the rigged Hilbert space $L^2(\mathcal{S}^{\prime}_{\bot}(\mathbf{R}^4,\mathbf{C}^{K \times 3}),d\nu_0)$
for the strictly positive generalized eigenvalue
\begin{equation}\label{genev}
\lambda_{i,s}^{a,g}(x_0,y_0)=g^2\sum_{j=1}^3\sum_{c=1}^K\frac{\frac{\pi^2}{2}k_{j,c,s}^2}{\left[\int_{-\frac{R}{2}}^{+\frac{R}{2}}dB_{j}^{c}\,g_{i, j}^{a, c}(B_{j}^{c};x_0,y_0)^{-1}\right]^2}.
\end{equation}
By varying the generalized eigenvector over $x_0,y_0$, $\mathbf{A}$, $k$ and $R$, we obtain a complete set of generalized eigenvectors for
$\mathcal{E}^{\prime}(L^2(\mathcal{S}^{\prime}_{\bot}(\mathbf{R}^3,\mathbf{C}^{K \times 3}),\mathbf{C},d\nu_0))$, because, if for any
$\Phi\in \mathcal{E}(L^2(\mathcal{S}^{\prime}_{\bot}(\mathbf{R}^3,\mathbf{C}^{K \times 3}),\mathbf{C},d\nu_0))$
\begin{equation}
\int_{\mathcal{S}^{\prime}_{\bot}(\mathbf{R}^3,\mathbf{C}^{K \times 3})}
\Psi^{a,x_0,y_0, \mathbf{A}}_{i,k}(\bar{\mathbf{A}})
\bar{\Phi}(\bar{\mathbf{A}})d\nu_0(\bar{\mathbf{A}})=0,
\end{equation}
that is
\begin{equation}
\delta(x-x_0)\delta(y-y_0)\psi_{i,k}^a(\mathbf{A};x_0,y_0)=0,
\end{equation}
which holds true for all $x_0,y_0\in\mathbf{R}^3$, $\mathbf{A}\in\mathbf{R}^{3K}$ and all $k\ge0$ iff $\Phi=0$.
By Corollary \ref{invGK} the operators $H_{i,II}^a$ on $L^2(\mathcal{S}^{\prime}_{\bot}(\mathbf{R}^3,\mathbf{C}^{K \times 3}),\mathbf{C},d\nu_0)$
is selfadjoint with non negative spectrum. The same holds true for and thus $H_{II}^g$ and  $H_{II}^{\Lambda, g}$ for all $\Lambda\ge0$.\\
Finally, we analyze the multiplication operator
\begin{equation}
\begin{split}
&(V^{\Lambda,g}\Psi)(\mathbf{A})=\int_{\mathbf{R}^3}d^3x\,V^{\Lambda,g}(t,x,\mathbf{A})\Psi(\mathbf{A}),\\
&\\
&V^{\Lambda,g}(t,x,\mathbf{A}):=\frac{1}{16}\varepsilon_i^{j,k}\varepsilon_i^{p,q}\left\{[\partial_jA^a_k(\varphi^{\Lambda}_t(\cdot-x))-\partial_kA_j^a(\varphi^{\Lambda}_t(\cdot-x))+\right.
g\varepsilon^{a,b,c}A_j^b(\varphi^{\Lambda}_t(\cdot-x))A_k^c(\varphi^{\Lambda}_t(\cdot-x))]\cdot\\
&\qquad\qquad\qquad\cdot[\partial_pA^a_q(\varphi^{\Lambda}_t(\cdot-x))-\partial_qA_p^a(\varphi^{\Lambda}_t(\cdot-x))
\left.+g\varepsilon^{a,b,c}A_p^b(\varphi^{\Lambda}_t(\cdot-x))A_q^c(\varphi^{\Lambda}_t(\cdot-x))]\right\}.
\end{split}
\end{equation}
Let $\mathbf{A}\in L^2_{\bot}(\mathbf{R}^3,\mathbf{C}^{K \times 3},d^3x)$ now be fixed.
\begin{equation}\label{Vconv}
\begin{split}
&V^{\Lambda,g}(t,x,\mathbf{A})\longrightarrow\frac{1}{16}\varepsilon_i^{j,k}\varepsilon_i^{p,q}\left\{[\partial_jA^a_k(t,x)-\partial_kA_j^a(t,x)+\right.
g\varepsilon^{a,b,c}A_j^b(t,x)A_k^c(t,x)]\cdot\\
&\qquad\qquad\qquad\cdot[\partial_pA^a_q(t,x)-\partial_qA_p^a(t,x)
\left.+g\varepsilon^{a,b,c}A_p^b(t,x)A_q^c(t,x)]\right\}=|R^{\nabla^\mathbf{A}}(t,x)|^2\quad(\Lambda\rightarrow +\infty),
\end{split}
\end{equation}
where $R^{\nabla^\mathbf{A}}$ is the curvature operator associated to the connection $\mathbf{A}$.
Any non zero $\psi\in C^{\infty}_0(\mathbf{R}^{3K},\mathbf{C})$ is eigenvector of the multiplication operator on $L^2(\mathbf{R}^{3K},\mathbf{C})$ with the non negative real
$V^{\Lambda,g}(t,x,\mathbf{A})$ for $\Lambda$ big enough. The functional
\begin{equation}
\Psi^{\mathbf{A};x_0}_k(\bar{\mathbf{A}}):=\delta(\bar{\mathbf{A}}-\mathbf{A})\delta(x-x_0)\psi_k(\mathbf{A}),
\end{equation}
where $(\psi_k)_{k\ge0}$ is an orthonormal basis of $L^2(\mathbf{R}^{3K},\mathbf{C})$ , is a generalized eigenvector
in $\mathcal{E}^{\prime}(\mathcal{S}^{\prime}_{\bot}(\mathbf{R}^3,\mathbf{C}^{K \times 3}),d\nu_0)$
for the operator $V^{\Lambda, g}$ on the rigged Hilbert space $L^2(\mathcal{S}^{\prime}_{\bot}(\mathbf{R}^3,\mathbf{C}^{K \times 3}),d\nu)$
for the generalized eigenvalue $V^{\Lambda,g}(t,x_0,\mathbf{A})$.
By varying the generalized eigenvector over $x_0\in\mathbf{R}^3$, $\mathbf{A}\in\mathbf{R}^{3K}$, and $k\ge0$, we obtain a complete set of generalized eigenvectors for
$\mathcal{E}^{\prime}(L^2(\mathcal{S}^{\prime}_{\bot}(\mathbf{R}^3,\mathbf{C}^{K \times 3}),\mathbf{C},d\nu_0))$, because, if for any
$\Phi\in \mathcal{E}(L^2(\mathcal{S}^{\prime}_{\bot}(\mathbf{R}^3,\mathbf{C}^{K \times 3}),\mathbf{C},d\nu_0))$
\begin{equation}
\int_{\mathcal{S}^{\prime}_{\bot}(\mathbf{R}^3,\mathbf{C}^{K \times 3})}
\Psi^{x_0, \mathbf{A}}_{k}(\bar{\mathbf{A}})
\bar{\Phi}(\bar{\mathbf{A}})d\nu_0(\bar{\mathbf{A}})=0,
\end{equation}
that is
\begin{equation}
\delta(x-x_0)\psi_{k}(\mathbf{A};x_0)=0,
\end{equation}
which holds true for all $x_0\in\mathbf{R}^3$, $\mathbf{A}\in\mathbf{R}^{3K}$ and all $k\ge0$ iff $\Phi=0$.
By Corollary \ref{invGK} the operator $V^{\Lambda,g}$ on $L^2(\mathcal{S}^{\prime}_{\bot}(\mathbf{R}^3,\mathbf{C}^{K \times 3}),\mathbf{C},d\nu_0)$
is selfadjoint with non negative spectrum if $\Lambda$ is big enough.\par
We conclude that $H^{\Lambda,g}$ is a selfadjoint operator as a sum of selfadjoint operators.
For $\Lambda$ big enough, if we choose $V_0^{\Lambda,g}\le\inf\spec((V^{\Lambda,g})$,
then the spectrum of $H^{\Lambda,g}$ lies in $[0, +\infty[$. The proof is concluded.\\
\end{proof}
\noindent Next, in view of the proof of Theorem \ref{pgs} we consider the commutative version of Theorem 4 in \cite{Gro72}
considering the remarks on page 59 therein.

\begin{theorem}\label{Gross}
Let $H_0$ be a nonnegative selfadjoint operator on $L^2(X,\mathbf{C},d\nu)$, where $(X,\mathcal{A},\nu)$ is a probability space.
Assume
\begin{itemize}
\item[(i)] $\exp(-tH_0)$ is a contraction in $L^p(X,\mathbf{C},d\nu)$ norm for all $t > 0$ and
all $p\in [1, +\infty]$ and $\exp(-TH_0)$ is a contraction from $L^2(X,\mathbf{C},d\nu)$ to $L^4(X,\mathbf{C},d\nu)$
for some real number $T > 0$.
\item[(ii)] $\exp(-tH_0)$ is positivity preserving for all $t > 0$.
\item[(iii)] The null space of $H_0$, $\ker(H_0)$ is spanned by the identity element of the algebra of bounded measurable functions on $(X,\mathcal{A})$.
\end{itemize}
Let $V$ be a selfadjoint operator given by the multiplication by some measurable real function $v$ on  $L^2(X,\mathbf{C},d\nu)$. Assume
\begin{itemize}
\item[(iv)]  $v\in L^k(X,\mathbf{C},d\nu)$ for some real number $k > 2$ and $\exp(-v)\in L^p(X,\mathbf{C},d\nu)$ is for all $p < +\infty$.
\end{itemize}
Then:
\begin{itemize}
\item[(a)] $H_0 + V$ is essentially selfadjoint and its closure $H$ is bounded from
below.
\item[(b)] If $\lambda = \inf\spec(H)$, then $\lambda$ is an eigenvalue of $H$ of
multiplicity one and there exists a corresponding non-negative eigenvector.
\end{itemize}
\end{theorem}

\begin{proof}[\text{Proof of Theorem \ref{pgs}}]
If $\Lambda$ big enough and $V_0^{\lambda,g}$ small enough, by Proposition \ref{generalized_ev} the operator $H^{\Lambda,g}$ is selfadjoint with non negative spectrum on
$\mathcal{H}:=L^2(X,\mathbf{C},d\nu_0)$ for $X:=\mathcal{S}^{\prime}_{\bot}(\mathbf{R}^3,\mathbf{C}^{K \times 3})$.
With the choice $V_0^{\Lambda,g}:=\inf\spec(H_I^{\Lambda}+H_{II}^{\Lambda,g}+V^{\Lambda,g})$ we obtain
\begin{equation}
0=\inf\spec(H^{\Lambda,g}),
\end{equation}
which, by Theorem \ref{Gross} (b), is a simple eigenvalue for an eigenvector $\Omega^{\Lambda,g}_0$.
With the choices
\begin{equation}
H_0:=H^{\Lambda,g}\text{ and } v:=0,
\end{equation}
the assumptions of Theorem \ref{Gross}
are satisfied, because, being $H_0$ selfadjoint with non-negative spectrum, we have
the spectral representation by of the projection valued measure $E:\mathcal{B}(\mathbf{R})\rightarrow\mathcal{L}(\mathcal{H})$ as
\begin{equation}
H_0\varphi=\int_0^{+\infty}\lambda dE(\lambda)\varphi\qquad\text{for }\varphi\in\mathcal{D}(H_0).
\end{equation}
Therefore, for all $t\ge0$
\begin{equation}
\begin{split}
&\exp(-tH_0)\varphi=\int_0^{+\infty}\exp(-t\lambda) dE(\lambda)\varphi\\
&\text{for }\varphi\in\mathcal{D}(\exp(-tH_0)):=\left\{\varphi\in\mathcal{H}\left|\,\int_0^{+\infty}|\exp(-t\lambda)|^2 d(E(\lambda)\varphi,\varphi)<+\infty\right.\right\}
\end{split}
\end{equation}

\begin{itemize}
\item[(i)] For $p=2$ we have for all $\varphi\in\mathcal{H}$
\begin{equation}
\begin{split}
&\|\exp(-tH_0)\varphi\|_{\mathcal{H}}^2=
\left(\exp(-tH_0)\varphi,\exp(-tH_0)\varphi\right)=\left(\exp(-2tH_0)\varphi,\varphi\right)=\\
&\quad=\int_0^{+\infty}\underbrace{\exp(-2t\lambda)}_{\le 1} d(E(\lambda,\varphi,\varphi)\le\|\varphi\|_{\mathcal{H}}^2.
\end{split}
\end{equation}
Hence, $\exp(-tH_0)$ is a contraction on $L^2$.\\
For $p\in[1,+\infty[$ we have
\begin{equation}
\|\exp(-tH_0)\varphi\|_{L^p}^p=\int_X\left|\int_0^{+\infty}\exp(-t\lambda) dE(\lambda)\varphi\right|^pd\nu
\le\underbrace{\left\|\int_0^{+\infty}\underbrace{\exp(-t\lambda)}_{\le 1} dE(\lambda)\right\|_p^p}_{\le 1}\|\varphi\|_{p}^p.
\end{equation}
Hence, $\exp(-tH_0)$ is a contraction on $L^p$ for $p\in[1,+\infty[$.\\
For $p=+\infty$ we have
\begin{equation}
\|\exp(-tH_0)\varphi\|_{L^\infty}=\int_X\sup_X\left|\int_0^{+\infty}\exp(-t\lambda) dE(\lambda)\varphi\right|d\nu
\le\underbrace{\left\|\int_0^{+\infty}\underbrace{\exp(-t\lambda)}_{\le 1} dE(\lambda)\right\|_{\infty}}_{\le 1}\|\varphi\|_{\infty}.
\end{equation}
Hence, $\exp(-tH_0)$ is a contraction on $L^{\infty}$.
\item[(ii)] Let us write $\exp(-tH_0)$ as integral operator
\begin{equation}\label{repK}
\exp(-tH_0)\varphi(x)=\int_XK(x,y)\varphi(y)d\nu(y)
\end{equation}
for an  appropriate with integral kernel $K=K(x,y)$. Therefore, for all $\varphi\in\mathcal{H}$
\begin{equation}
(\exp(-tH_0)\varphi,\varphi)=
\int_{X^2}K(x,y)|\varphi(y)|^2d\nu(x)d\nu(y)=
\left\|\exp\left(-\frac{t}{2}H_0\right)\varphi\right\|^2_{\mathcal{H}}\ge
0.
\end{equation}
This can only be true if $K\ge0$. Hence, for $\varphi>0$, (\ref{repK}) shows that $\exp(-tH_0)\varphi>0$, meaning that $\exp(-tH_0)\varphi$
is positivity preserving.
\item[(iii)] Every element of $\ker(H_0)\subset\mathcal{E}(L^2(\mathcal{S}^{\prime}_{\bot}(\mathbf{R}^3,\mathbf{C}^{K \times 3}),\mathbf{C},d\nu_0))$ must be a bounded measurable function on $X$,
because $H_0$ is an infinite dimensional elliptic operator \cite{BeKo12}.
\item[(iv)] The choice $v=0$ satisfies this assumption.
\end{itemize}
\noindent The proof is completed.\\
\end{proof}
\begin{rem}
The ground state $\Omega_0^{\Lambda,g}$ of $H^{\Lambda,g}$ has been constructed in $L^2(\mathcal{S}^{\prime}_{\bot}(\mathbf{R}^3,\mathbf{C}^{K \times 3}),\mathbf{C},d\nu_0)$, where $\nu_0$ is the standard Gaussian measure.
This is NOT the measure for which we will verify the Osterwalder-Schrader axioms.
\end{rem}

\subsection{Construction of the Probability Measure under the Ultraviolet and Infrared Cut Offs and Regularization}
To construct appropriate probability measures  on $\mathcal{S}^{\prime}_{\bot}(\mathbf{R}^4,\mathbf{C}^{K \times 3})$ we need results about infinitesimal generators of time inhomogeneous It\^{o}' s diffusions.
\begin{proposition}\label{infGen}
Let $(W_t)_{t\ge0}$ be a $M$-dimensional standard $\mathbb{P}-$Brownian motion with respect to the filtration $(\mathcal{A}_t)_{t\ge0}$. Let
$b:[0,+\infty[\times\mathbf{R}^N\rightarrow\mathbf{R}^N$, $\sigma:[0,+\infty[\times\mathbf{R}^N\rightarrow\mathbf{R}^{N\times M}$ be Borel measurable and locally bounded functions,
satisfying
\begin{equation}
\left|b(t,x)-b(t,y)\right|_{\mathbf{R}^N}+\left|\sigma(t,x)-\sigma(t,y)\right|_{\mathbf{R}^{N\times M}}\le K \left|x-y\right|_{\mathbf{R}^N}
\end{equation}
for a positive constant $K$, meaning Lipschitz-continuity with respect to $x$ uniform in $t$. The solution of the SDE
\begin{equation}\label{Ito}
\left\{
  \begin{array}{l}
    dX_t=b(t,X_t)dt+\sigma(t,X_t)dW_t \\
    X_0=x_0\in\mathbf{R}^N,
  \end{array}
\right.
\end{equation}
is a time inhomogeneous It\^{o}'s diffusion, whose infinitesimal generator is given by the PDO with variable coefficients
\begin{equation}
\begin{split}
&L_t=\frac{1}{2}\sum_{i,i=1}^Na_{i,j}(t,x)\frac{\partial^2}{\partial x_i\partial x_j}+\sum_{i=1}^Nb_i(t,x)\frac{\partial}{\partial x_j}\\
&\mathcal{D}(L_t):=C^{\infty}_0(\mathbf{R}^N,\mathbf{R}^N)\subset L^2(\mathbf{R}^N,\mathbf{R}^N, d^Nx)\rightarrow L^2(\mathbf{R}^N,\mathbf{R}^N, d^Nx),
\end{split}
\end{equation}
where
\begin{equation}
a(t,x):=\sigma(t,x)\sigma(t,x)^{\dagger}.
\end{equation}
Conversely, if the operator $L_t$ is elliptic for all $t\ge0$, then for $\sigma(t,x):=a^{\frac{1}{2}}(t,x)$ and $M=N$, there exists an It\^{o}'s diffusion as (\ref{Ito}) whose transition density $\kappa_t(x_0,x)$ is the heat kernel of $L_t$, i.e. the solution of
\begin{equation}
\left\{
  \begin{array}{l}
    \frac{\partial}{\partial t}u(t,x)=L_tu(t,x) \\
    u(0,x)=\delta(x-x_0)\in\mathcal{S}^{\prime}(\mathbf{R}^N,\mathbf{R}^N).
  \end{array}
\right.
\end{equation}
It follows that the solution of
\begin{equation}
\left\{
  \begin{array}{l}
    \frac{\partial}{\partial t}u(t,x)=L_tu(t,x) \\
    u(0,x)=f(x)\in C^{\infty}_0(\mathbf{R}^N,\mathbf{R}^N)
  \end{array}
\right.
\end{equation}
is given by
\begin{equation}
u(t,x_0)=\mathbb{E}[f(X_t)|\mathcal{A}_t]=\int_{\mathbf{R}^N} f(x)\kappa_t(x_0,x)d^Nx
\end{equation}
\end{proposition}
\begin{proof}
See chapters 8.3-8.5 of \cite{CCFI11} and chapters VII.1-VII.2 \cite{RJ99}.
\end{proof}
\noindent What is the situation in the infinite dimensional case?
\begin{proposition}\label{diffInf}
Let $\mathcal{F}$ and $\mathcal{G}$ be a real separable Hilbert spaces and $\mathcal{L}_{\text{HS}}(\mathcal{G},\mathcal{F})$ denote
the vector space of all Hilbert-Schmidt operators from $\mathcal{G}$ to $\mathcal{F}$.
Let $(W_t)_{t\ge0}$ be a standard $\mathbb{P}-$Wiener process with respect to the filtration $(\mathcal{A}_t)_{t\ge0}$ taking values  $\mathcal{G}$ and assume that
$b:[0,+\infty[\times\mathcal{F}\rightarrow\mathcal{F}$, $\sigma:[0,+\infty[\times\mathcal{F}\rightarrow\mathcal{L}_{\text{HS}}(\mathcal{G},\mathcal{F})$ be Borel measurable and locally bounded functions,
satisfying
\begin{equation}
\left|b(t,x)-b(t,y)\right|_{\mathcal{F}}+\left|\sigma(t,x)-\sigma(t,y)\right|_{\mathcal{L}_{\text{HS}}(\mathcal{G},\mathcal{F})}\le K \left|x-y\right|_{\mathcal{F}}
\end{equation}
for a positive constant $K$, meaning Lipschitz-continuity with respect to $x$ uniform in $t$. The (strong) solution of the SDE
\begin{equation}\label{Ito}
\left\{
  \begin{array}{l}
    dX_t=b(t,X_t)dt+\sigma(t,X_t)dW_t \\
    X_0=x_0\in\mathcal{F}
  \end{array}
\right.
\end{equation}
is a time inhomogeneous It\^{o}'s diffusion. In particular, it satifies the Markov property and thus is a Markov process.
\end{proposition}
\begin{proof}
It is a special case of Theorem 7.4 of \cite{DZ92} for identity covariance of the Wiener process and vanishing linear operator in the drift.
The Markov property follows from Theorem 9.8 of \cite{DZ92}.
\end{proof}

 We can now proceed with the construction of a probability measure on $\mathcal{S}^{\prime}_{\bot}(\mathbf{R}^4,\mathbf{C}^{K \times 3})$.
Inspired by the treatment of the quantum field associated to a particle in a potential as depicted in chapter 3 of \cite{GJ87},
we adapt the ideas therein to the Yang-Mills fields with a cutoff. Using the Feynman-Ka\v{c} formula, we construct probability
measures on $\mathcal{E}$ and $\mathcal{H}$ satisfying those Osterwalder-Schrader axioms implying the reconstruction theorem
of quantum mechanics, and thus the selfadjointness and non-negativity of the cutoff Hamilton operator. However, due to the
presence of a non-local term in $H_{II}^{\lambda, g}$ created by modified green function $G$, it is not possible, unless $g=0$ to
realize a fibrewise construction over every $x\in\mathbf{R}^3$, and then integrate over $x$.\\
If we exclude for the moment the part of the Hamiltonian containing the potential, from Proposition \ref{quantizationprop} formula (\ref{H-V}) we can write
\begin{equation}
\begin{split}
&H_I^{\Lambda}+H_{II}^{\Lambda,g}=\int_{\mathbf{R}^3}d^3x\,
\left\{-\frac{1}{2}\left[\frac{\delta}{\delta A_i^a(\varphi^{\Lambda}_t(\cdot-x))}\right]^2+\right.\\
&\qquad\qquad\qquad\;\left.-\frac{g^2}{2}\left[\int_{\mathbf{R}^3}d^3y\,\partial_iG^{a,b}(\mathbf{A}(\varphi^{\Lambda}_t(\cdot-\cdot));x,y)
\varepsilon^{b,c,d}A_k^d(\varphi^{\Lambda}_t(\cdot-y))\frac{\delta}{\delta A_k^c(\varphi^{\Lambda}_t(\cdot-y))}\right]^2\right\}.
\end{split}
\end{equation}
For any probability $\mathbb{P}$, every time $t\in\mathbf{R}$ (seen as parameter) let us now consider
the operator $h_0^g$ on the Hilbert space $L^2(L^2_{\bot}(\mathbf{R}^3,\mathbf{C}^{K \times 3}),\mathbf{C}, d\mathbb{P})$ defined
 as
\begin{equation}
\begin{split}
&H_0^g:=\int_{\mathbf{R}^3}d^3x\,h_0^g(x),\text{ where }\\
&h_0^g(x):=\left\{-\frac{1}{2}\left[\frac{\delta}{\delta A_i^a(t,x)}\right]^2
-\frac{g^2}{2}\left[\int_{\mathbf{R}^3}d^3y\,\partial_iG^{a,b}(\mathbf{A}(t,\cdot);x,y)\varepsilon^{b,c,d}A_k^d(t,y)\frac{\delta}{\delta A_k^c(t,y)}\right]^2\right\}.
\end{split}
\end{equation}
\noindent Since
\begin{equation}\label{lim}
\left[\frac{\delta}{\delta A_i^a(t,x)}\right]^2=
-\lim_{\Xi\rightarrow+\infty}\int_{\mathbf{R}^3\times\mathbf{R}^3}d^3y\,d^3\bar{y}\,\psi^{\Xi}(y-x)\psi^{\Xi}(\bar{y}-x)
\delta_d^{\bar{d}}\delta_k^{\bar{k}}\frac{\delta} {\delta A_k^d(t,y)}
\frac{\delta}{\delta A_{\bar{k}}^{\bar{d}}(t,\bar{y})},
\end{equation}
where $(\psi^{\Xi})_{\Xi\ge0}\subset\mathcal{S}(\mathbf{R}^3)$ is a delta sequence,
 i.e. $\mathcal{S}^{\prime}-\lim_{\Xi\rightarrow+\infty}\psi^{\Xi}=\delta\in\mathcal{S}^{\prime}(\mathbf{R}^3)$, for which $\psi^{\Xi}>0$ for all $\Xi$,
 and the limit is pointwise on the domain of definition of the operator on the l. h. s. of (\ref{lim}),
we can express the operator $h_0^g(x)$ as
\begin{equation}\label{exp}
\begin{split}
h_0^{g}(x)&=
-\lim_{\Xi\rightarrow+\infty}\left[
\int_{\mathbf{R}^3}d^3y\,
b_k^d(\Xi,x,y,g;\mathbf{A}(t,\cdot))
\frac{\delta}{\delta A_{k}^{d}(t,y)}\right.+\\
&\qquad\qquad\qquad+\left.
\int_{\mathbf{R}^3\times\mathbf{R}^3}d^3y\,d^3\bar{y}\;a_{k,\bar{k}}^{d,\bar{d}}(\Xi,x, y,\bar{y},g;\mathbf{A}(t,\cdot))
\frac{\delta}{\delta A_k^d(t,y)}
\frac{\delta}{\delta A_{\bar{k}}^{\bar{d}}(t,\bar{y})}
\right],
\end{split}
\end{equation}
\noindent for appropriate matrix $a(\Xi, x, y, \bar{y},g;\mathbf{A}(t,\cdot))$ and vector $b(\Xi,x,y,g;\mathbf{A}(t,\cdot))$ valued coefficient functions.
The horizontal vector $b=[b^d_k]$ has entries ordered by the bi-index $(d,k)$. The quadratic matrix $a=[a^{d,\bar{d}}_{k,\bar{k}}]$ has row index $(d,k)$
and column index $(\bar{d},\bar{k})$. Note that we have written the matrices
$\mathbf{A}(t,y)$ and $\mathbf{A}(t,\bar{y})$
in their equivalent column vector forms. The limit in (\ref{exp}) holds pointwise on the domain of definition of $h_0^g$.   Furthermore, equation (\ref{exp}) becomes
\begin{equation}\label{xi2}
h_0^g(x)=
-\lim_{\Xi\rightarrow+\infty} \left<b(\Xi,x,g;\mathbf{A}(t,\cdot)),\frac{\delta}{\delta\mathbf{A}(t,\cdot)}\right>_{\mathbf{A}}+\left<a(\Xi,x,g;\mathbf{A}(t,\cdot))\frac{\delta}{\delta\mathbf{A}(t,\cdot)},\frac{\delta}{\delta\mathbf{A}(t,\cdot)}\right>_{\mathbf{A}},
\end{equation}
\noindent where
\begin{equation}
b(\Xi,x,g;\mathbf{A}(t,\cdot))=y\mapsto[b^{d}_{k}(\Xi, x, y,g;\mathbf{A}(t,\cdot))]\\
\end{equation}
\noindent and
\begin{equation}
\begin{split}
a(\Xi,x,g;\mathbf{A}(t,\cdot))&=y\mapsto[a^{d,\bar{d}}_{k,\bar{k}}(\Xi,x,y,g;\mathbf{A}(t,\cdot))]\\
\left(a(\Xi,y,g;\mathbf{A}(t,\cdot)))\frac{\delta}{\delta \mathbf{A}(t,y)}\right)^{\bar{d}}_{\bar{k}}&=\int_{\mathbf{R}^3}d^3\bar{y}\,a^{d,\bar{d}}_{k,\bar{k}}(\Xi,y,\bar{y},g;\mathbf{A}(t,\cdot))\frac{\delta}{\delta A^d_k(t,\bar{y})}
\end{split}
\end{equation}

\noindent with the notation, for any $\mathbf{B}\in L^2_{\bot}(\mathbf{R}^3,\mathbf{C}^{K \times 3})$
\begin{equation}
\left<b_1, b_2\right>_{\mathbf{B}}:=\int_{\mathbf{R}^3}d^3y\;b_1^r(\mathbf{B}(y))b_2^r(\mathbf{B}(y))
\end{equation}
 for $b_{1,2}\in L^2(\mathbf{C}^{K \times 3},\mathbf{R})$. \par


Both $b$ and $a$ are functionals of $\mathbf{A}$
and depend on the parameters $\Xi$.
There exists a $g_0\in[0,1[$ (not depending on $\Lambda$!), such that, if the coupling constant $g\in[0,g_0[$, then
 the expression $a(\Xi,x, g;\mathbf{A})$
represents a positive definite operator valued functional of $\mathbf{A}=\mathbf{A}(t,\cdot)$ seen as $\mathbf{C}^{3K}$-valued function in the
 separable Hilbert space $L^2(\mathbf{R}^3,\mathbf{C}^{3K})$. By Proposition \ref{diffInf}, we can construct the diffusion
 $\mathfrak{A}_t=\mathfrak{A}_t(\Xi,x,g)$
\begin{equation}
    d\mathfrak{A}_t=b(\Xi,x,g;\mathfrak{A}_t)dt+a^{\frac{1}{2}}(\Xi,x,g;\mathfrak{A}_t)dW_t,
\end{equation}
where $(W_t)_{t\ge0}$ is the standard Wiener process adapted to the filtration $(\mathcal{A}_t)_{t\ge0}$ in the Hilbert space $L^2(\mathbf{R}^3,\mathbf{C}^{3K})$. Note that $a^{\frac{1}{2}}$ is a Hilbert-Schmidt
operator because $a$ is a Hilbert-Schmidt integral operators. Both $b$ and $a^{\frac{1}{2}}$ are Lipschitz-continuous with respect to $\mathbf{A}$, because they are Fr\'{e}chet differentiable with continuous derivative. The Lipschitz constant does not depend on $t$, because $b$ and $a^{\frac{1}{2}}$ do not either.\par
For fixed $x\in\mathbf{R}^3$ the
$\mathbf{R}^{3K}$-valued process $\mathfrak{A}_t(\Xi,x,g)$ is a Markov process with stochastic kernel
$\kappa_{t,s}(\Xi,x,g)$ such that for all $\mathbf{A}_t\in\mathbf{C}^{3K}$ and all measurable $B\in\mathcal{A}_s$
\begin{equation}
\mathbb{P}[\mathfrak{U}_s(\Xi,g;x)\in B\,|\,\mathfrak{U}_t(\Xi,g,x)=\mathbf{A}_t]= \kappa_{t,s}(\Xi,g,x;\mathbf{A}_t,B).
\end{equation}


Let $\mathcal{W}_{\bot}(\mathbf{A}, \mathbf{A}^\prime, t,x)$ be the set of continuous paths $\mathbf{A}(s,x)$ in $\mathbf{C}^{3K}$
which take the values $\mathbf{A}(-t/2,x)=\mathbf{A}$ and $\mathbf{A}(+t/2,x)=\mathbf{A}^\prime$ at their endpoints such that their transversal component $\mathbf{A}^{\parallel}(s,\cdot)$ vanishes for all $s$.
The cylinder sets of $\mathcal{W}_{\bot}(\mathbf{A}, \mathbf{A}^\prime, t,x)$ have the form
\begin{equation}
\begin{split}
Z_{\bot}(\mathbf{A},\mathbf{A}^\prime,t,\{I_j\}_j,x)&=\left\{\mathbf{A}(s,x)\,\left|\mathbf{A}\in\mathcal{W}_{\bot}(\mathbf{A}, \mathbf{A}^\prime, t,x),\,\mathbf{A}(-t/2,x)=\mathbf{A},\mathbf{A}(+t/2,x)=\mathbf{A}^\prime, \right.\right.\\
&\qquad\left.\mathbf{A}(t_j,x)\in I_j,\text{ for all }j=1,\dots,n\right\},
\end{split}
\end{equation}
where $-t/2<t_1<t_2<\dots<t_n<t/2$ and $I_j$ are Borel subsets of $\mathbf{C}^{3K}$. On these cylinder sets we can define the measure given by
\begin{equation}\label{U_measure}
U_{\mathbf{A},\mathbf{A}^\prime}^{t, x,g,\Xi}(Z):=
\kappa_{-t/2,t_1}(\Xi,g,x;\mathbf{A},\cdot)\otimes\kappa_{t_1,t_2}(\Xi,g,x)\otimes\dots\otimes\kappa_{t_{n-1},t_{n}}(\Xi,g,x)\otimes\kappa_{t_n, +t/2}(\Xi,g,x;\cdot,\mathbf{A}^\prime),
\end{equation}
which is countably additive and has a unique extension to the Borel subsets of $\mathcal{W}_{\bot}(\mathbf{A}, \mathbf{A}^\prime, t,x)$, being the tensor product of stochastic kernels.\par
The proof of Proposition \ref{generalized_ev} contains the proof that the operator $h_0(x)$ on $L^2(\mathbf{C}^{3K},\mathbf{C})$ is selfadjoint,
and from Theorem \ref{Gross}, $h_0^{g}(x)$ has a unique ground state that we denote by
$\omega_0^{x,g}=\omega_0^{ x,g}(\mathbf{A})\in L^2(\mathbf{C}^{K \times 3},\mathbf{C})$
. For any $R>0$, termed \textbf{infrared cutoff}, the expression
\begin{equation}\label{defxi}
d\xi^{R,g}_{t}:=\int_{|x|\le R}d^3x\int_{|s|\le \frac{t}{2}}d^1s\int_{\mathbf{R}^{3K}\times\mathbf{R}^{3K}}\;dU_{\mathbf{A},\mathbf{A}^\prime}^{s, x,g,R}
\;\omega_0^{x,g}(\mathbf{A})\,\omega_0^{x,g}(\mathbf{A}^\prime)
\end{equation}
\noindent defines a finite measure on $\mathcal{W}_{\bot}(\mathbf{A}, \mathbf{A}^\prime, t,x)$ 
which pushes forward by means of the inclusion to a finite measure $\xi^{g,R}_{t}$ on $\mathcal{S}_{\bot}(\mathbf{R}^4,\mathbf{C}^{K \times 3})$ first and to $\mathcal{S}^{\prime}_{\bot}(\mathbf{R}^4,\mathbf{C}^{K \times 3})$ next.
\noindent Note that from (\ref{U_measure}) to (\ref{defxi}) we have chosen $\Xi=R>0$, which is legitimate, because in our construction they both will tend to infinity. Therefore, inspired by \cite{Ja82}, we can introduce the
\begin{defi}[\textbf{Infrared/Ultraviolet Cutoff Measure}]\label{IUCutoffM}
With
\begin{equation}\label{def_mu_Lambda_0}
\begin{split}
Z_{t}^{R,g}&:=\int_{\mathcal{S}_{\bot}^{\prime}(\mathbf{R}^4,\mathbf{C}^{K \times 3})}\,d\xi^{R,g}_{t}\\
d\mu^{R, g}_{t}&:=\frac{1}{Z_{t}^{R,g,}}d\xi^{R,g}_{t},
\end{split}
\end{equation}
which is a probability measure on on $\mathcal{S}^{\prime}_{\bot}(\mathbf{R}^4,\mathbf{C}^{K \times 3})$,
we can define
\begin{equation}\label{def_mu_Lambda}
\begin{split}
Z_{t}^{\Lambda,R, g}&:=\int_{\mathcal{S}_{\bot}^{\prime}(\mathbf{R}^4,\mathbf{C}^{K \times 3})}\,\left[\exp\left(-\int_{|x|\le R}d^3x\int_{-\frac{t}{2}}^{+\frac{t}{2}}\,ds\,V^{\Lambda,g}(s,x,\mathbf{A})\right)\right]d\mu^{R,g}_t\\
d\mu^{\Lambda,R, g}_{t}&:=\frac{1}{Z_{t}^{\Lambda,R, g}}\left[\exp\left(-\int_{|x|\le R}d^3x\int_{-\frac{t}{2}}^{+\frac{t}{2}}\,ds\,V^{\Lambda,g}(s,x,\mathbf{A})\right)\right]d\mu^{R,g}_t
\end{split}
\end{equation}
as a probability measure on $\mathcal{S}_{\bot}^{\prime}(\mathbf{R}^4,\mathbf{C}^{K \times 3})$.
\end{defi}
\noindent We remark that $\mu_t^{\Lambda,R,g}$ is not gauge invariant and not translation invariant.
\subsection{Infrared Cutoff Removal and Reconstruction of a Selfadjoint Hamiltonian}
\begin{defi}[\textbf{Ultraviolet Measure}]\label{UCutoffM}
For any measurable $A\subset\mathcal{S}_{\bot}^{\prime}(\mathbf{R}^4,\mathbf{C}^{K \times 3})$ let
\begin{equation}\label{UVMD}
\begin{split}
\mu_t^{\Lambda,g}(A)&:=\limsup_{R\rightarrow+\infty}\mu_t^{\Lambda,R,g}(A)\\
\mu^{\Lambda,g}(A)&:=\limsup_{t\rightarrow+\infty}\mu_{t}^{\Lambda,g}(A).
\end{split}
\end{equation}
\end{defi}
\noindent Are $\mu_t^{\Lambda,g}$ and $\mu^{\Lambda,g}$ measures? The answer is yeas and requires several steps. F.i. in \cite{Doo94}, chapter IX.10 we can find the proof of
\begin{theorem}[\textbf{Vitali-Hahn-Saks}]\label{VHS}
Let $(X,\mathcal{B})$ be a measurable space and $(\mu_j)_{j\ge0}$ a sequence of probability measures such that $(\mu_j(A))_{j\ge0}$
converges for all measurable $A\in\mathcal{B}$. Then, $\mu(A):=\lim_{j\rightarrow+\infty}\mu_j(A)$ defines a probability on $(X,\mathcal{B})$.
\end{theorem}
\begin{rem}
In (\ref{UVMD}) for any $A$ we can always find a sequence $R_j\uparrow+\infty$ as $j\rightarrow+\infty$ such that $\mu_{R_j}^{\Lambda,g}(A)\rightarrow \mu^{\Lambda,g}(A)$ as $j\rightarrow+\infty$.
But this sequence can a priori depend on $A$, so that we cannot apply immediately the Vitali-Hahn-Saks theorem.
\end{rem}
\noindent Nevertheless, we have
\begin{proposition}\label{muLProb}
The expressions $\mu_t^{\Lambda,g}$ and $\mu^{\Lambda,g}$ in (\ref{UVMD}) define for all $\Lambda\ge0$ big enough, $g\in[0,1[$ small enough
and $t>0$ probability measures on $\mathcal{S}_{\bot}^{\prime}(\mathbf{R}^4,\mathbf{C}^{K \times 3})$.
\end{proposition}
\begin{proof}
Let us drop the $\Lambda$ and $g$ to ease the notation. We first prove the claim for the case $V=0$. We have
\begin{equation}
\mu_t^{R}(A)=\frac{\xi_t^{R}(A)}{\xi_t^{R}(\mathcal{S}^{\prime})},\text{ where }\xi_t^{R}=\int_{\mathcal{S}^{\prime}} d\xi_t^{R}
\end{equation}
By construction (\ref{defxi}) and (\ref{xibound}) for any measurable $A$ the measure $\xi_t^{R}(A)$ is monotone increasing in $R$ and bounded from above:
\begin{equation}
\xi_t^{R}(A)\le\xi_t^{S}(A)\le\xi_t^{+\infty}(A)<+\infty\quad(R\le S).
\end{equation}
Therefore, for $R\le S$
\begin{equation}
\begin{split}
&\frac{\xi_t^{R}(A)}{\xi_t^{+\infty}(A)}\le\frac{\xi_t^{S}(A)}{\xi_t^{+\infty}(A)}\\
&\frac{\xi_t^{R}(A)}{\xi_t^{+\infty}(A)}\underbrace{\frac{\xi_t^{R}(\mathcal{S}^{\prime})}{\xi_t^{R}(\mathcal{S}^{\prime})}}_{=1}\le\frac{\xi_t^{S}(A)}{\xi_t^{+\infty}(A)}\underbrace{\frac{\xi_t^{R}(\mathcal{S}^{\prime})}{\xi_t^{R}(\mathcal{S}^{\prime})}}_{=1}\\
&\frac{\xi_t^{R}(\mathcal{S}^{\prime})}{\xi_t^{+\infty}(A)}\mu_t^{R}(A)\le \frac{\xi_t^{S}(\mathcal{S}^{\prime})}{\xi_t^{+\infty}(A)}\mu_t^{S}(A)\\
&\underbrace{\frac{\xi_t^{R}(\mathcal{S}^{\prime})}{\xi_t^{S}(\mathcal{S}^{\prime})}}_{\in]0,1]}\le \frac{\mu_t^{S}(A)}{\mu_t^{R}(A)}.
\end{split}
\end{equation}
and
\begin{equation}
\mu_t^R(A)\le\mu_t^S(A).
\end{equation}
We conclude that for any measurable $A$ and for any sequence $R_j\uparrow+\infty$ as $j\rightarrow+\infty$, the sequence $(\mu_t^{R_j}(A))
_{j\ge0}$ is monotone increasing.
By Theorem \ref{VHS} $\mu_t$ defines a probability measure on $\mathcal{S}^{\prime}$. The proof for $\mu$ goes analogously, because, by construction,
 it is the limit of the monotone increasing sequence $(\mu_{t_j})_{j\ge0}$ for $t_j\uparrow+\infty$ as $j\rightarrow+\infty$.\\
We know consider the general case, when $V$ does not vanish. We have
\begin{equation}
\begin{split}
&\mu_{t}^{\Lambda, R, g}(A)=\frac{\xi_t^{\Lambda, R, g}(A)}{\xi_t^{\Lambda, R, g}(\mathcal{S}^{\prime})},\text{ where }\\
&\xi_t^{\Lambda, R, g}(A):=\int_{A} \,\left[\exp\left(-\int_{|x|\le R}d^3x\int_{-\frac{t}{2}}^{+\frac{t}{2}}\,ds\,V^{\Lambda,g}(s,x,\mathbf{A})\right)\right]d\mu_t^{\Lambda,R,g}(\mathbf{A}).
\end{split}
\end{equation}
By construction for any measurable $A$ the measure $\xi_t^{\Lambda,R,g}(A)$ is monotone increasing in $R$ and bounded from above:
\begin{equation}
\xi_t^{\Lambda,R,g}(A)\le\xi_t^{\Lambda,S,g}(A)\le\xi_t^{\Lambda,+\infty,g}(A)<+\infty\quad(R\le S).
\end{equation}
Therefore, for $R\le S$
\begin{equation}
\begin{split}
&\frac{\xi_t^{\Lambda,R,g}(A)}{\xi_t^{\Lambda, +\infty, g}(A)}\le\frac{\xi_t^{\Lambda,S,g}(A)}{\xi_t^{\Lambda, +\infty, g}(A)}\\
&\frac{\xi_t^{\Lambda,R,g}(A)}{\xi_t^{\Lambda, +\infty, g}(A)}\underbrace{\frac{\xi_t^{\Lambda,R,g}(\mathcal{S}^{\prime})}{\xi_t^{\Lambda,R,g}(\mathcal{S}^{\prime})}}_{=1}\le\frac{\xi_t^{\Lambda,S,g}(A)}{\xi_t^{\Lambda, +\infty, g}(A)}\underbrace{\frac{\xi_t^{\Lambda,R,g}(\mathcal{S}^{\prime})}{\xi_t^{\Lambda,R,g}(\mathcal{S}^{\prime})}}_{=1}\\
&\frac{\xi_t^{\Lambda,R,g}(\mathcal{S}^{\prime})}{\xi_t^{\Lambda, +\infty, g}(A)}\mu_t^{\Lambda,R,g}(A)\le \frac{\xi_t^{\Lambda,S,g}(\mathcal{S}^{\prime})}{\xi_t^{\Lambda, +\infty, g}(A)}\mu_t^{\Lambda,S,g}(A)\\
&\underbrace{\frac{\xi_t^{\Lambda,R,g}(\mathcal{S}^{\prime})}{\xi_t^{\Lambda,S,g}(\mathcal{S}^{\prime})}}_{\in]0,1]}\le \frac{\mu_t^{\Lambda,S,g}(A)}{\mu_t^{\Lambda,R,g}(A)}.
\end{split}
\end{equation}
and
\begin{equation}
\mu_t^{\Lambda,R,g}(A)\le\mu_t^{\Lambda,S,g}(A).
\end{equation}
We conclude that for any measurable $A$ and for any sequence $R_j\uparrow+\infty$ as $j\rightarrow+\infty$, the sequence $(\mu_t^{\Lambda, R_j,g}(A))
_{j\ge0}$ is monotone increasing.
By Theorem \ref{VHS} $\mu_t^{\Lambda, g}$ defines a probability measure on $\mathcal{S}^{\prime}$. The proof for $\mu^{\Lambda,g}$ goes analogously, because, by construction,
 it is the limit of the monotone increasing sequence $(\mu_{t_j}^{\Lambda,g})_{j\ge0}$ for $t_j\uparrow+\infty$ as $j\rightarrow+\infty$.\\
The proof is completed.
\end{proof}

\noindent By Fubini's theorem for distributions (cf. \cite{Tr06}), we can write the measures $\mu_t^{\Lambda, g}$  and $\mu^{\Lambda,g}$ on $\mathcal{S}^{\prime}_{\bot}(\mathbf{R}^4,\mathbf{C}^{K \times 3})$ as
\begin{equation}
\begin{split}
\mu_t^{\Lambda,g}(\mathbf{A}(s,x))&=\varrho^{\Lambda,g}_t(\mathbf{A}(\cdot,x))\nu^{\Lambda,g}(\mathbf{A}(s,\cdot))\\
\mu^{\Lambda,g}(\mathbf{A}(s,x))&=\varrho^{\Lambda,g}(\mathbf{A}(\cdot,x))\nu^{\Lambda,g}(\mathbf{A}(s,\cdot)),
\end{split}
\end{equation}
\noindent where $\varrho^{\Lambda, g}_t(\mathbf{A}(\cdot,x)):=\mu^{\Lambda, g}_t(\mathbf{A}(\cdot,x))$ and $\varrho^{\Lambda,g}(\mathbf{A}(\cdot,x)):=\mu^{\Lambda,g}(\mathbf{A}(\cdot,x))$
are probability measures on $\mathcal{S}^{\prime}_{\bot}(\mathbf{R}^1,\mathbf{C}^{K \times 3})$, and $\nu^{\Lambda, g}(\mathbf{A}(s,\cdot)):=\mu_t^{\Lambda, g}(\mathbf{A}(s,\cdot))=\mu^{\Lambda, g}(\mathbf{A}(s,\cdot))$ is a probability measure on
$\mathcal{S}^{\prime}_{\bot}(\mathbf{R}^3,\mathbf{C}^{K \times 3})$, respectively.
Remark that $\nu^{\Lambda,g}$ does not depend on $t$ by construction.\par

\begin{theorem}[\textbf{Feynman-Ka\v{c}-Nelson Formula}]\label{FKNT}
The operator $H^{\Lambda,g}=H_{I}+H_{II}^{\Lambda,g}+V^{\Lambda,g}$ has a domain in the Hilbert space
$L^2(\mathcal{S}_{\bot}^{\prime}(\mathbf{R}^3,\mathbf{C}^{K \times 3}), d\nu^{\Lambda,g})$, has
fundamental state $\Omega_0^{\Lambda,g}=\Omega_0^{\Lambda,g}(\mathbf{A}(s,x))$, i.e. $H^{\Lambda,g}\Omega_0^{\Lambda,g}=0$, and satisfies the Feynman-Ka\v{c}-Nelson formula
\begin{equation}\label{repthm}
\begin{split}
&\left(
\Omega_0^{\Lambda,g},B_1e^{-(s_2-s_1)H^{\Lambda,g}}B_2e^{-(s_3-s_2)H^{\Lambda,g}}\cdot\dots\cdot
B_N\Omega_0^{\Lambda,g}
\right)^{\Lambda,g}=\\
&\qquad\qquad\qquad
=\int_{\mathcal{S}_{\bot}^{\prime}(\mathbf{R}^4,\mathbf{C}^{K \times 3})}\hspace{-0.5cm}\Pi_{k=1}^NB_k(\mathbf{A}(s_k,\cdot))d\mu^{\Lambda,g}(\mathbf{A}),
\end{split}
\end{equation}
\noindent where the scalar product $\left(\cdot,\cdot\right)^{\Lambda,g}$ on  $\mathcal{S}_{\bot}^{\prime}(\mathbf{R}^3,\mathbf{C}^{K \times 3})$ is defined as
\begin{equation}
\left(\Upsilon,\Theta\right)^{\Lambda,g}:=\int_{\mathcal{S}_{\bot}^{\prime}(\mathbf{R}^3,\mathbf{C}^{K \times 3})}\;\Upsilon(\mathbf{A})\bar{\Theta}(\mathbf{A})d\nu^{\Lambda,g}(\mathbf{A}),
\end{equation}
the functionals $(B_k=B_k(\mathbf{A}))_{k=1,\dots,N}$ are in
 $L^2(\mathcal{S}_{\bot}^{\prime}(\mathbf{R}^3,\mathbf{C}^{K \times 3}), \mathbf{C}, d\nu^{\Lambda,g})$,
 and $(s_k)_{k=1,\dots,N}$ is a partition of the interval $[\tau,T]$ defined as $s_k:=\tau+k\frac{T-\tau}{N}$
 for $\tau:=-\frac{t}{2}$ and $T:=+\frac{t}{2}$,.
\end{theorem}

\begin{proof}
We show now that the argument in the proof of Theorem 3.4.1 in \cite{GJ87} can be utilized this set up, so that
the limit of the integrand on the r.h.s. of (\ref{es_lambda_t}) for $t\rightarrow+\infty$ exists.
Recall from Theorem \ref{pgs} that the operator $H^{\Lambda,g}$ has
fundamental state $\Omega_0^{\Lambda,g}=\Omega_0^{\Lambda,g}(\mathbf{A}(s,x))$, i.e. $H^{\Lambda,g}\Omega_0^{\Lambda,g}=0$. The state $\Omega_0^{\Lambda,g}$ depends neither on $\Lambda$ nor on $g$.
Note that $s$ is seen as parameter.
By adapting the proof of Theorem 3.4.1 in \cite{GJ87}, we can write
\begin{equation}\label{repthm0}
\begin{split}
&\left(
\Omega_0^{\Lambda,g},B_1e^{-(s_2-s_1)H^{\Lambda,g}}B_2e^{-(s_3-s_2)H^{\Lambda,g}}\cdot\dots\cdot
B_N\Omega_0^{\Lambda,g}
\right)^{\Lambda,g}
=\\
&\qquad\qquad\qquad=\lim_{t\rightarrow+\infty}
\int_{\mathcal{S}_{\bot}^{\prime}(\mathbf{R}^4,\mathbf{C}^{K \times 3})}\hspace{-0.5cm}\Pi_{k=1}^NB_k(\mathbf{A}(s_k,\cdot))d\mu_t^{\Lambda,g}(\mathbf{A}).
\end{split}
\end{equation}

\noindent
Since $\mu^{\Lambda,g}=\lim_{t\rightarrow+\infty}\mu^{\Lambda,g}_t$,
we have
\begin{equation}
\begin{split}
&\left(
\Omega_0^{\Lambda,g},B_1e^{-(s_2-s_1)H^{\Lambda,g}}B_2e^{-(s_3-s_2)H^{\Lambda,g}}\cdot\dots\cdot
B_N\Omega_0^{\Lambda,g}
\right)^{\Lambda,g}
=\\
&\qquad\qquad\qquad=\int_{\mathcal{S}_{\bot}^{\prime}(\mathbf{R}^4,\mathbf{C}^{K \times 3})}\hspace{-0.5cm}\Pi_{k=1}^NB_k(\mathbf{A}(s_k,\cdot))d\mu^{\Lambda,g}(\mathbf{A}),
\end{split}
\end{equation}
with the same assumptions as for (\ref{repthm0}).\\
\end{proof}
\begin{rem}
Theorem \ref{FKNT} for the FKN formula on $3K$-dimensional distributions is consistent with Theorem 3.4.1 in \cite{GJ87} for the FKN formula on $1$-dimensional distributions.
\end{rem}

\begin{theorem}[\textbf{Ultraviolet Measure Properties}]\label{UVMPL} There exists a $g_0\in[0,1[$ not dependent on $\Lambda$, such that
the generating functional
\begin{equation}\label{es_lambda_r}
S^{\Lambda,g}(f)=\int_{\mathcal{S}_{\bot}^{\prime}(\mathbf{R}^4,\mathbf{C}^{K \times 3})}e^{\imath\mathbf{A}(f)}d\mu^{\Lambda,g}(\mathbf{A}),
\end{equation}
\noindent for $f\in\mathcal{S}_{\bot}(\mathbf{R}^4,\mathbf{C}^{K \times 3})$ satisfies
satisfies the Osterwalder-Schadrer axioms (OS0)-(OS4) and hence the Wightman axioms (W1)-(W8). Note that $S^{\Lambda,g}(f)$ and $A(f)$ are $K\times 3$ complex matrices, and that the exponential is meant componentwise.
\end{theorem}

\begin{proof}\text{}\\
Without loss of generality we can assume that $S^{\Lambda,g}(f)$ and $A(f)$ are complex numbers throughout this proof, because the general proof can be reconstructed by iterating over the components of the complex $K\times 3$ matrices representing them.\par
We first prove (OS2) and (OS3), which will be immediately needed to apply the reconstruction theorem of quantum mechanics.
(OS2): The invariance of the generating functional $S^{\Lambda,g}$ under time translation and time reflection follows directly from the
definition of $\mu^{\Lambda,g}$ in (\ref{UVMD}). More exactly, by (\ref{def_mu_Lambda_0}) the infrared/ultraviolet cutoff measure $\mu^{\Lambda,R, g,0}_{t}$
is invariant under the space rotations and reflections in O(3), and hence the ultraviolet cutoff measure in invariant under all space-time rotations
and reflections, as well as translations.\\
(OS3):
we have to show that the complex matrix $M^{\Lambda,g}:=[M^{\Lambda,g}_{i,j}]$, where
\begin{equation}\label{defM}
M^{\Lambda,g}_{i,j}:=S^{\Lambda,g}(\theta f_i-f_j),
\end{equation}
is positive definite for all choices of $(f_i)_{i=1,\dots,n}\subset\mathbf{S}(\mathbf{R}^4,\mathbf{R})$,
such that $\supp(f_i)\subset[0,+\infty[\times\mathbf{R}^3$,
and $(\theta f)(s,x):=f(-s,x)$ denotes the time reflection.\\
\noindent With the choice of functionals $B_k$ as
\begin{equation}
B_k(\mathbf{A}):=\exp\left(\frac{\imath}{N}\mathbf{A}\left(\theta(f_i)(s_k,\cdot)-f_j(s_k,\cdot)\right)\right),
\end{equation}
the r.h.s. of (\ref{repthm}) leads to
\begin{equation}\label{p1}
\lim_{N\rightarrow+\infty}\int_{\mathcal{S}_{\bot}^{\prime}(\mathbf{R}^4,\mathbf{C}^{K \times 3})}\hspace{-0.5cm}\Pi_{k=1}^NB_k(\mathbf{A}(s_k,\cdot))d\mu^{\Lambda,g}(\mathbf{A})=
\int_{\mathcal{S}_{\bot}^{\prime}(\mathbf{R}^4,\mathbf{C}^{K \times 3})}\hspace{-0.5cm}\exp\left(-\imath\mathbf{A}\left(\theta (f_i)-f_j\right)\right)d\mu^{\Lambda,g}(\mathbf{A}),
\end{equation}
when $\tau\rightarrow-\infty$ and $T\rightarrow+\infty$. We have
 approximated the time integration in $\mathbf{A}(f)$ by means of a Riemann sum for the partition $(s_k)_{k=1,\dots,N}$. For the l.h.s of (\ref{repthm}) we obtain
 in the limit
\begin{equation}\label{p2}
\begin{split}
&\lim_{N\rightarrow+\infty}\left(\Omega_0^{\Lambda,g},B_1e^{-(t_2-t_1)H^{\Lambda,g}}B_2e^{-(t_3-t_2)H^{\Lambda,g}}\dots B_N\Omega_0^{\Lambda,g}\right)^{\Lambda, g}
=\\
&\quad=\left(\Omega_0^{\Lambda,g},\exp\left(-\imath\mathbf{A}\left(\theta (f_i)-f_j\right)\right)\Omega_0^{\Lambda,g}\right)^{\Lambda, g}=\\
&\quad=\left( \exp\left(-\imath\mathbf{A}\left(\theta (f_i)\right)\right)\Omega_0^{\Lambda,g},\exp\left(-\imath\mathbf{A}\left(f_j\right)\right)\Omega_0^{\Lambda,g}\right)^{\Lambda, g}=\\
&\quad=\left(\exp\left(-\imath\mathbf{A}\left(f_i\right)\right)\Omega_0^{\Lambda,g},\exp\left(-\imath\mathbf{A}\left(f_j\right)\right)\Omega_0^{\Lambda,g}\right)^{\Lambda, g},
\end{split}
\end{equation}
because the supports of $f_i$ and $f_j$ lie in the time positive half space. Putting (\ref{repthm})  with (\ref{p1}) and (\ref{p2})
together shows that the matrix $[M^{\Lambda,g}_{i,j}]$ with entries
\begin{equation}\label{Fub}
M^{\Lambda,g}_{i,j}=\int_{\mathcal{S}_{\bot}^{\prime}(\mathbf{R}^4,\mathbf{C}^{K \times 3})}\hspace{-0.5cm}\exp\left(-\imath\mathbf{A}\left(\theta (f_i)-f_j\right)\right)d\mu^{\Lambda,g}(\mathbf{A})
=\left(\exp\left(-\imath\mathbf{A}\left(f_i\right)\right)\Omega_0^{\Lambda,g},\exp\left(-\imath\mathbf{A}\left(f_j\right)\right)\Omega_0^{\Lambda,g}\right)^{\Lambda, g}
\end{equation}
is positive definite for all $g\in[0,g_0[$, and, by Proposition \ref{posprop} (or Corollary 3.4.4
in \cite{GJ87}), the reflection positivity axiom (OS3) is fulfilled.\\
\end{proof}
\noindent We can now prove
\begin{theorem}\label{Selfadjointness}
There exists a $g_0\in[0,1[$ not depending on $\Lambda$, such that, if the coupling constant $g\in[0,g_0[$, for the
probability measure $\mu^{\Lambda,g}$ on $\mathcal{S}^{\prime}_{\bot}(\mathbf{R}^4,\mathbf{C}^{K \times 3})$
 the Hamilton operator $H^{\Lambda,g}$ is selfadjoint for the choice $\nu^{\Lambda,g}$.
If the coupling constant $g$ vanishes, both measures $\mu^{\Lambda,0}$ and $\nu^{\Lambda,0}$ are Gaussian, otherwise not.
\end{theorem}
\begin{proof}
We know that all the assumptions of Theorem \ref{rec} are satisfied, because we have already verified the Osterwalder-Schrader axioms (OS2) ans (OS3).
Hence, the time translation operator $T(t)$ satisfies
\begin{equation}\label{convH}
T(t)^{\curlywedge^{\Lambda,g}}=e^{-t\tilde{H}^{\Lambda,g}},
\end{equation}
where $\tilde{H}^{\Lambda,g}$ is a selfadjoint operator on the Hilbert space $L^2\left(\mathcal{S}_{\bot}^{\prime}(\mathbf{R}^3,\mathbf{C}^{K \times 3}), d\nu^{\Lambda,g}\right)$.
Note that the canonical embedding $\curlywedge$ is defined utilizing the measure $\mu^{\Lambda,g}$, which is highlighted by the superscripts $\Lambda$ and $g$ in the notation $\curlywedge^{\Lambda,g}$.
To conclude the proof we have to show that $H^{\Lambda,g}=\tilde{H}^{\Lambda,g}$. A slight reformulation of (\ref{repthm}) provides the equality
\begin{equation}\label{H_x_eq}
\left(\Omega_0^{\Lambda,g},\Psi e^{-sH^{\Lambda,g}}\Phi\Omega_0^{\Lambda,g}\right)^{\Lambda, g}=
\int_{\mathcal{S}_{\bot}^{\prime}(\mathbf{R}^4,\mathbf{C}^{K \times 3})}\Psi\left(\mathbf{A}(0,\cdot)\right)\Phi\left(\mathbf{A}(s,\cdot)\right)d\mu^{\Lambda,g}(\mathbf{A}),
\end{equation}
where:
\begin{itemize}
\item the distribution $\mathbf{A}(s,\cdot)\in\mathcal{S}_{\bot}^{\prime}(\mathbf{R}^3,\mathbf{C}^{K \times 3})$ depends on the parameter $s\in\mathbf{R}^1$,
\item the functionals $\Psi,\Phi\in \mathcal{D}(H^{\Lambda,g})\subset L^2(\mathcal{S}^{\prime}_{\bot}(\mathbf{R}^3,\mathbf{C}^{K \times 3}),\mathbf{C},d\nu^{\Lambda,g})$.
\end{itemize}
We now compute the r.h.s of equation (\ref{H_x_eq}) and obtain
\begin{equation}\label{FKNExtra}
\begin{split}
&\int_{\mathcal{S}_{\bot}^{\prime}(\mathbf{R}^4,\mathbf{C}^{K \times 3})}\Psi\left(\mathbf{A}(0,\cdot)\right)\Phi\left(\mathbf{A}(s,\cdot)\right)d\mu^{\Lambda,g}(\mathbf{A})=\\
&\quad=\int_{\mathcal{S}_{\bot}^{\prime}(\mathbf{R}^4,\mathbf{C}^{K \times 3})}\Psi\left(\mathbf{A}(0,\cdot)\right)\Phi\left(\mathbf{A}(s, \cdot)\right)d(\varrho^{\Lambda, g}\otimes \nu^{\Lambda,g})(\mathbf{A})=\\
&\quad=\int_{\mathcal{S}_{\bot}^{\prime}(\mathbf{R}^4,\mathbf{C}^{K \times 3})}\Psi\left(\mathbf{A}(0,\cdot)\right)\Phi\left(\mathbf{A}(s,\cdot)\right)d\nu^{\Lambda,g}(\mathbf{A}(t,\cdot))d\varrho^{\Lambda, g}(\mathbf{A}(\cdot,x))=\\
&\quad=\int_{\mathcal{S}_{\bot}^{\prime}(\mathbf{R}^1,\mathbf{C}^{K \times 3})}\left[\int_{\mathcal{S}_{\bot}^{\prime}(\mathbf{R}^3,\mathbf{C}^{K \times 3})}\Psi\left(\mathbf{A}(0,\cdot)\right)\Phi\left(e^{-s\tilde{H}^{\Lambda,g}}\mathbf{A}(0,\cdot)\right)d\nu^{\Lambda,g}(\mathbf{A}(t,\cdot))\right]d\varrho^{\Lambda, g}(\mathbf{A}(\cdot,x))=\\
&\quad=\left(\int_{\mathcal{S}_{\bot}^{\prime}(\mathbf{R}^3,\mathbf{C}^{K \times 3})}\Psi\left(\mathbf{A}(0,\cdot)\right)\Phi\left(e^{-s\tilde{H}^{\Lambda,g}}\mathbf{A}(0,\cdot)\right)d\nu^{\Lambda,g}(\mathbf{A})\right)\underbrace{\left(\int_{\mathcal{S}_{\bot}^{\prime}(\mathbf{R}^1,\mathbf{C}^{K \times 3})}d\varrho^{\Lambda, g}(\mathbf{A})\right)}_{=1}=\\
&\quad=\int_{\mathcal{S}_{\bot}^{\prime}(\mathbf{R}^3,\mathbf{C}^{K \times 3})}\Psi\left(\mathbf{A}\right)(e^{-s\tilde{H}^{\Lambda,g}})\Phi\left(\mathbf{A}\right)d\nu^{\Lambda,g}(\mathbf{A}),
\end{split}
\end{equation}
where we have utilized Fubini's theorem for distributions (cf. \cite{Tr06}), the fact that the integrand does not depend on the time $t$, and that $\int d\varrho^{\Lambda, g}=1$,
being $\varrho^{\Lambda, g}$ a probability measure.\\
Therefore,
\begin{equation}\label{H_x_eq_int}
\left(\Omega_0^{\Lambda,g},\bar{\Psi} e^{-sH^{\Lambda,g}}\Phi\Omega_0^{\Lambda,g}\right)^{\Lambda, g}=\int_{\mathcal{S}_{\bot}^{\prime}(\mathbf{R}^3,\mathbf{C}^{K \times 3})}\Psi\left(\mathbf{A}\right)\overline{e^{-s\tilde{H}^{\Lambda,g}}\Phi\left(\mathbf{A}\right)}d\nu^{\Lambda,g}(\mathbf{A})
\end{equation}
We now take the derivative $-\left.\frac{d}{ds}\right|_{s:=0}$ on both sides  of (\ref{H_x_eq_int}) and obtain
\begin{equation}\label{h-h}
\left(\Omega_0^{\Lambda,g},\bar{\Psi} H^{\Lambda,g}\Phi\Omega_0^{\Lambda,g}\right)^{\Lambda, g}
=\int_{\mathcal{S}_{\bot}^{\prime}(\mathbf{R}^3,\mathbf{C}^{K \times 3})}\Psi\left(\mathbf{A}\right)\overline{\tilde{H}^{\Lambda,g}\Phi\left(\mathbf{A}\right)}d\nu^{\Lambda,g}(\mathbf{A}).
\end{equation}
Equation (\ref{h-h}) holds for all $\Psi,\Phi\in\mathcal{D}(H^{\Lambda,g})$, and
 we conclude that $H^{\Lambda,g}=\tilde{H}^{\Lambda,g}$. The proof is completed.\\
\end{proof}
\noindent We can now verify the remaining Osterwalder-Schrader axioms.
\begin{proof}[\text{Proof of Theorem \ref{UVMPL}, continuation}]\text{}\\
(OS0): Let us consider a finite set of test functions $f_j\in\mathcal{S}_{\bot}(\mathbf{R}^4,\mathbf{C}^{K \times 3})$, $j=1,\dots, n$
and complex numbers $z:=(z_1,z_2,\dots,z_n)\in\mathbf{C}^n$, the complex partial derivative of $S^{\Lambda,g}\left(\sum_{j=1}^nz_jf_j\right)$ with respect to $z_i$ reads
\begin{equation}
\frac{\partial}{\partial z_i}\left[S^{\Lambda,g}\left(\sum_{j=1}^nz_jf_j\right)\right]=\int_{\mathcal{S}_{\bot}^{\prime}(\mathbf{R}^4,\mathbf{C}^{K \times 3})}\imath \mathbf{A}(f_i) e^{\imath\mathbf{A}\left(\sum_{j=1}^nz_jf_j\right)}d\mu^{\Lambda,g}(\mathbf{A}).
\end{equation}

\noindent By the Cauchy-Schwarz inequality
\begin{equation}\label{CS}
\hspace{-3.8mm}
\left|\frac{\partial}{\partial z_i}\left[S^{\Lambda,g}\left(\sum_{j=1}^nz_jf_j\right)\right]\right|^2\le
\int_{\mathcal{S}_{\bot}^{\prime}(\mathbf{R}^4,\mathbf{C}^{K \times 3})} |\mathbf{A}(f_i)|^2d\mu^{\Lambda,g}(\mathbf{A})
\int_{\mathcal{S}_{\bot}^{\prime}(\mathbf{R}^4,\mathbf{C}^{K \times 3})} |e^{\imath\mathbf{A}\left(\sum_{j=1}^nz_jf_j\right)}|^2d\mu^{\Lambda,g}(\mathbf{A})
\end{equation}
Since $\lim_{g\rightarrow0^+}\mu^{\Lambda,g}=\mu^{\Lambda,0}$, the positive constant
\begin{equation}
K^{\Lambda,g}=\left\|\frac{d\mu^{\Lambda,g}}{d\mu^{\Lambda,0}}\right\|_{L^{\infty}(\mathcal{S}_{\bot}^{\prime}(\mathbf{R}^4,\mathbf{C}^{K \times 3}))}
\end{equation}
is bounded, and for all $g$ small enough
\begin{equation}\label{inK}
\begin{split}
&\int_{\mathcal{S}_{\bot}^{\prime}(\mathbf{R}^4,\mathbf{C}^{K \times 3})} |\mathbf{A}(f)|^2d\mu^{\Lambda,g}(\mathbf{A})\le K^{\Lambda,g}\int_{\mathcal{S}_{\bot}^{\prime}(\mathbf{R}^4,\mathbf{C}^{K \times 3})} |\mathbf{A}(f)|^2d\mu^{\Lambda,0}(\mathbf{A})\\
&\int_{\mathcal{S}_{\bot}^{\prime}(\mathbf{R}^4,\mathbf{C}^{K \times 3})} |e^{\imath\mathbf{A}(f)}|^2d\mu^{\Lambda,g}(\mathbf{A})\le K^{\Lambda,g}\int_{\mathcal{S}_{\bot}^{\prime}(\mathbf{R}^4,\mathbf{C}^{K \times 3})} |e^{\imath\mathbf{A}(f)}|^2d\mu^{\Lambda,0}(\mathbf{A}).
\end{split}
\end{equation}
for all $f\in \mathcal{S}_{\bot}(\mathbf{R}^4,\mathbf{C}^{K \times 3})$.
Being $\mu^{\Lambda,0}$ a Gaussian measure, we obtain
\begin{equation}\label{expCov}
\int_{\mathcal{S}_{\bot}^{\prime}(\mathbf{R}^4,\mathbf{C}^{K \times 3})} e^{\imath \mathbf{A}(f)}d\mu^{\Lambda,0}(\mathbf{A})=e^{-\frac{1}{2}(f,C^{\Lambda}f)_{L^2(\mathbf{R}^4,\mathbf{C}^{K \times 3})}},
\end{equation}
where
\begin{equation}
C^{\Lambda}=\frac{1}{2}\bigoplus_{\substack{a=1,\dots,K\\
          i=1,2,3}}\left(\mathbb{1}_{\mathbf{R}^{4}}-\Delta_{\mathbf{R}^{4}}+\underbrace{2V^{\Lambda,0}-\mathbb{1}_{\mathbf{R}^{4}}}_{\ge0}\right)^{-1}
\le \frac{1}{2}\bigoplus_{\substack{a=1,\dots,K\\
          i=1,2,3}}\left(\mathbb{1}_{\mathbf{R}^{4}}-\Delta_{\mathbf{R}^{4}}\right)^{-1}=:C^{\emptyset},
\end{equation}
because the operator $V^{\Lambda,0}$ is non-negative.
By developing the exponential function in both sides of (\ref{expCov}) and equating the quadratic term in $f$, we obtain
\begin{equation}\label{squareCov}
\begin{split}
\int_{\mathcal{S}_{\bot}^{\prime}(\mathbf{R}^4,\mathbf{C}^{K \times 3})} |\mathbf{A}(f)|^2d\mu^{\Lambda,0}(\mathbf{A})&=
\left(f,C^{\Lambda}f\right)_{L^2(\mathbf{R}^4,\mathbf{C}^{K \times 3})}\le\left(f,C^{\emptyset}f\right)_{L^2(\mathbf{R}^4,\mathbf{C}^{K \times 3})}=\\
&=\|f\|_{H^{-1}(\mathbf{R}^4,\mathbf{C}^{K \times 3})}^2,
\end{split}
\end{equation}
where $H^{-1}$ is the Sobolev space with ``differentiability'' $-1$. By comparing all powers of $f$, we obtain
\begin{equation}\label{expCovD}
\int_{\mathcal{S}_{\bot}^{\prime}(\mathbf{R}^4,\mathbf{C}^{K \times 3})} |e^{\imath\mathbf{A}(f)}|^2d\mu^{\Lambda,0}(\mathbf{A})\le
\exp\left(\|f\|_{H^{-1}(\mathbf{R}^4,\mathbf{C}^{K \times 3})}^2\right),
\end{equation}
Inserting (\ref{squareCov}), (\ref{expCovD}) and (\ref{inK}) into (\ref{CS}) leads to
\begin{equation}\label{indiff}
\left|\frac{\partial}{\partial z_i}\left[S^{\Lambda,g}\left(\sum_{j=1}^nz_jf_j\right)\right]\right|^2\le
 \left(K^{\Lambda, g}\right)^2\|f_i\|_{H^{-1}(\mathbf{R}^4,\mathbf{C}^{K \times 3})}^2\exp\left(\|\sum_{j=1}^nz_jf_j\|_{H^{-1}(\mathbf{R}^4,\mathbf{C}^{K \times 3})}^2\right)<+\infty,
\end{equation}
because all $f_j$'s are in the Schwartz space and a fortiori in the Sobolev space $H^{-1}$. The analyticity for all $z\in\mathbf{C}^n$ is therefore proved.\\
(OS1): Since $S^{\Lambda,g}(zf)$ is analytic for all $f\in\mathcal{S}_{\bot}(\mathbf{R}^4,\mathbf{C}^{K \times 3})$ and all $z\in\mathbf{C}$
the mean value theorem of differentiation implies that it exists a $z^{\prime}\in\mathbf{C}$, $|z^{\prime}|\le|z|$ such that
\begin{equation}
\left|S^{\Lambda,g}(zf)-\underbrace{S^{\Lambda,g}(0)}_{=1}\right|\le\left|\frac{d}{d z_i}S^{\Lambda,g}(z^{\prime}f)\right| |z|.
\end{equation}
Therefore, utilizing (\ref{indiff}) we obtain
\begin{equation}
\left|S^{\Lambda,g}(zf)\right|\le 1+|z|K^{\Lambda, g}\|f\|_{H^{-1}(\mathbf{R}^4,\mathbf{C}^{K \times 3})}\exp\left(\frac{|z|^2}{2}\|f\|_{H^{-1}(\mathbf{R}^4,\mathbf{C}^{K \times 3})}^2\right)
\end{equation}
We choose $z:=1$
\begin{equation}\label{inKa}
\left|S^{\Lambda,g}(f)\right|\le 1+ K^{\Lambda, g}\|f\|_{H^{-1}(\mathbf{R}^4,\mathbf{C}^{K \times 3})}\exp\left(\frac{1}{2}\|f\|_{H^{-1}(\mathbf{R}^4,\mathbf{C}^{K \times 3})}^2\right),
\end{equation}
and study the function of the variable $\alpha\in[0,+\infty[$
\begin{equation}
\kappa(\alpha):=\frac{1+K^{\Lambda,g}\alpha\exp\left(\frac{1}{2}\alpha^2\right)}{\exp(L\alpha^2)},
\end{equation}
for a given constant $L>0$. If $L>\frac{1}{2}$, then
\begin{equation}
\lim_{\alpha\rightarrow+\infty}\kappa(\alpha)=0\text{ and }\lim_{\alpha\rightarrow 0^+}\kappa(\alpha)=1.
\end{equation}
Therefore, for any given $K^{\Lambda, g}>0$ there exists a $L^{\Lambda, g}>0$ such that for all $\alpha\in[0,+\infty[$
\begin{equation}\label{KL}
1+K^{\Lambda,g}\alpha\exp\left(\frac{1}{2}\alpha^2\right)\le\exp(L^{\Lambda, g}\alpha^2),
\end{equation}
which, utilized with $\alpha:=\|f\|_{H^{-1}(\mathbf{R}^4,\mathbf{C}^{K \times 3})}$ and inserted into (\ref{inKa}), leads to
\begin{equation}\label{OS1in}
\begin{split}
\left|S^{\Lambda,g}(f)\right|\le \exp(L^{\Lambda, g}\|f\|_{H^{-1}(\mathbf{R}^4,\mathbf{C}^{K \times 3})}^2)&\le\exp(L^{\Lambda, g}\|f\|_{H^{0}(\mathbf{R}^4,\mathbf{C}^{K \times 3})}^2)\le\\
&\le\exp(L^{\Lambda, g}(\|f\|_{L^1(\mathbf{R}^4,\mathbf{C}^{K \times 3})}+\|f\|_{L^2(\mathbf{R}^4,\mathbf{C}^{K \times 3})}^2)),
\end{split}
\end{equation}
because $\|f\|_{H^{-1}}\le\|f\|_{H^{0}}$ for all $f$, and $H^0(\mathbf{R}^4,\mathbf{C}^{K \times 3})=L^2(\mathbf{R}^4,\mathbf{C}^{K \times 3})$.\\
(OS4): We have to check that the Euclidean time translation subgroup, which by Theorem \ref{Selfadjointness} reads  $\{T(t)\}_{t\ge0}$,
where $T(t)^{\curlywedge^{\Lambda,g}}=e^{-tH^{\Lambda,g}}$, acts ergodically
on the measure space $(\mathcal{S}^{\prime}_{\bot}(\mathbf{R}^4, \mathbf{C}^{K \times 3}), d\mu^{\Lambda,g})$, or, equivalently (see \cite{GJ87}, Formula 19.7.1),
 that it satisfies the \textbf{cluster property}, i.e.
\begin{equation}\label{cluster}
\begin{split}
&\lim_{t\mapsto+\infty}\frac{1}{t}\int_0^t\left[\int_{\mathcal{S}^{\prime}_{\bot}(\mathbf{R}^4, \mathbf{C}^{K \times 3})}\Phi(\mathbf{A})T(s)\Psi(\mathbf{A})d\mu^{\Lambda,g}(\mathbf{A})\right]ds=\\
&\qquad\qquad=\int_{\mathcal{S}^{\prime}_{\bot}(\mathbf{R}^4, \mathbf{C}^{K \times 3})}\Phi(\mathbf{A})d\mu^{\Lambda,g}(\mathbf{A})\cdot\int_{\mathcal{S}^{\prime}_{\bot}(\mathbf{R}^4, \mathbf{C}^{K \times 3})}\Psi(\mathbf{A})d\mu^{\Lambda,g}(\mathbf{A}),
\end{split}
\end{equation}
for all $\Phi,\Psi\in L^1(\mathcal{S}^{\prime}_{\bot}(\mathbf{R}^4,\mathbf{C}^{K \times 3}), d\mu^{\Lambda,g})$.
In the proof of Theorem 19.7.1 in \cite{GJ87} the cluster property (\ref{cluster}) is shown to be equivalent with the uniqueness of the ground state.
Hence, we have to show that $\Omega_0^{\Lambda,g}$ is an eigenvector of $H^{\Lambda, g}$ with multiplicity $1$, which follows from Theorem 3.3.2 and 3.3.3 in \cite{GJ87},
because $A^{\Lambda,g}:=e^{-tH^{\Lambda, g}}$ has a strictly positive kernel, being $H^{\Lambda,g}=H_I^{\Lambda}+H_{II}^{\Lambda,g}+V^{\Lambda,g}$ self-adjoint
and $V^{\Lambda,g}$ bounded from below by $0$ by construction.
\end{proof}
\subsection{Ultraviolet Cutoff Removal without Renormalization}
Now we remove the regularization given by the ultraviolet cutoff by letting $\Lambda\rightarrow+\infty$ and making sure that the properties (OS2) and (OS3),
necessary for the reconstruction theorem of quantum mechanics, are maintained. Actually, all Osterwalder-Schrader axioms will be preserved.
\begin{defi}[\textbf{$\mathbf{4}$D-YM-Measure}]
For any measurable $A\subset\mathcal{S}_{\bot}^{\prime}(\mathbf{R}^4,\mathbf{C}^{K \times 3})$ let
\begin{equation}\label{UVO}
\mu^{g}(A):=\limsup_{\Lambda\rightarrow+\infty}\mu^{\Lambda,g}(A).
\end{equation}
\end{defi}
\noindent Is $\mu^g$ a measure? The answer is yes and requires several steps.
\begin{defi}[\textbf{Tightness}] Let $(X,\mathcal{B}(X))$ be a measurable topological space. The collection of probability measures $\mathcal{M}$ over $X$ is \textbf{tight}
if and only if for every $\varepsilon>0$ there exists a compact $K_{\varepsilon}\subset X$ such that $\mu(K_{\varepsilon})>1-\varepsilon$ for all $\mu\in\mathcal{M}$ .
\end{defi}
\noindent In f.i. \cite{Bo06} we can find the proof of
\begin{theorem}[\textbf{Prokhorov}] Let $X$ be  a separable metric space, and $\mathcal{P}(X)$ the collection of all probability measures defined on
$X$ with its Borel $\sigma$-algebra. A family of probability measures $\mathcal{M}\subset\mathcal{P}(X)$ is tight if and only
if $\overline{\mathcal{M}}$  is weakly sequentially compact. i.e. if every sequence $(\mu_j)_{j\ge0}\subset\mathcal{M}$ contains a subsequence
$(\mu_{j_k})_{k\ge0}$ weakly converging to a $\mu\in\overline{\mathcal{M}}$, meaning by this
\begin{equation}
\lim_{k\rightarrow+\infty}\mathbb{E}^{\mu_{j_k}}[f]=\mathbb{E}^{\mu}[f]
\end{equation}
for all bounded continuous functions $f$ on $X$.
\end{theorem}
\noindent In spite of the fact that $\mathcal{S}^{\prime}$ is not metrizable, Theorem I.6.5 in \cite{Fe67} implies
\begin{theorem}[\textbf{Fernique}]\label{Fer}
Prokhorov's theorem holds true for $\mathcal{S}^{\prime}$.
\end{theorem}
\begin{proposition}\label{tightProp}
Let $(X,\mathcal{A},P)$ be a probability space and let $F_{\Lambda}:X\rightarrow[0,+\infty[$ be a family of uniformly bounded measurable functions indexed
by the parameter $\Lambda\ge0$. Then, the collection of probability measures $(P_{\Lambda})_{\Lambda\ge0}$ over $X$ defined for any measurable $A\in\mathcal{A}$
\begin{equation}
P_{\Lambda}(A):= \frac{\int_A F_{\Lambda}(x) dP }{\int_X F_{\Lambda}(x)  dP}.
\end{equation}
is tight.
\end{proposition}

\begin{proof}
We need to show that, for any $\varepsilon > 0$, there exists a compact set $K=K_{\varepsilon}\subset X$ such that
$P_{\Lambda}(K) > 1 - \varepsilon$ for sufficiently large values of $\Lambda$.
First, note that since $F_{\Lambda}(x)$ is uniformly bounded, there exists a constant $M > 0$ such that $F_{\Lambda}(x)\le M$ for all $x\in X$ and all $\Lambda\ge0$.
Then, for any measurable set $A$ in $\mathcal{A}$, we have
\begin{equation}
  P_{\Lambda}(A) = \frac{\int_A F_{\Lambda} dP}{\int_X F_{\Lambda}dP}\le\frac{MP(A)}{\int_X F_{\Lambda}dP}.
\end{equation}
For any $\varepsilon>0$, we can choose a compact set $K=K_{\varepsilon}$ such that
\begin{equation}
  P(X\setminus K) \le \varepsilon.
\end{equation}
Hence we have
\begin{equation}\label{K1}
  P_{\Lambda}(K) = 1 - P_{\Lambda}(X\setminus K) \ge 1-\frac{MP(X\setminus K)}{\int_X F_{\Lambda}dP}=1-\frac{MP(X\setminus K)}{\Lambda\int_X \frac{F_{\Lambda}}{\Lambda}dP}.
\end{equation}
Since $F_{\Lambda}(x)\le M$ for all $x\in X$ and all $\Lambda\ge0$, it follows by Lebesgue's dominated convergence
\begin{equation}
\lim_{\Lambda\rightarrow+\infty}\int_X \frac{F_{\Lambda}}{\Lambda}dP=0,
\end{equation}
and thus
\begin{equation}\label{K2}
\int_X \frac{F_{\Lambda}}{\Lambda}dP\in[0,1[ \text{ for } \Lambda \text{ big enough}.
\end{equation}
Inserting (\ref{K2}) into (\ref{K1}) leads to
\begin{equation}
  P_{\Lambda}(K)\ge 1 - \frac{M\varepsilon}{\Lambda}.
\end{equation}
Therefore, the family $(P_{\Lambda})_{\Lambda\ge0}$ is tight since, for any
$\varepsilon > 0$, we can find a compact set $K=K_{\epsilon}\subset X$ such that $P_{\Lambda}(K) > 1 - \varepsilon$ for $\Lambda\ge\max(M,\lambda_0)$.
The proof is finished.\\
\end{proof}
\noindent Putting everything together we obtain
\begin{proposition}
The expression $\mu^{g}$  in (\ref{UVO}) defines for all  $g\in[0,g_0[$ for a $g_0$ small enough
a probability measures on $\mathcal{S}_{\bot}^{\prime}(\mathbf{R}^4,\mathbf{C}^{K \times 3})$.
\end{proposition}
\begin{proof}
By (\ref{def_mu_Lambda_0}) and (\ref{UVMD}) we have
\begin{equation}\label{def_mu_Lambda_0_bis}
\begin{split}
Z_{t}^{V=0, R,g}&=\int_{\mathcal{S}_{\bot}^{\prime}(\mathbf{R}^4,\mathbf{C}^{K \times 3})}\,d\xi^{R,g}_{t}\\
d\mu^{V=0, g}&=\lim_{\substack{R\rightarrow+\infty\\t\rightarrow+\infty}}\frac{1}{Z_{t}^{V=0, R,g}}d\xi^{R,g}_{t},
\end{split}
\end{equation}
and by (\ref{def_mu_Lambda}) and (\ref{UVMD}) we have
\begin{equation}\label{def_mu_LambdaBis}
\begin{split}
Z_{t}^{\Lambda,R, g}&=\int_{\mathcal{S}_{\bot}^{\prime}(\mathbf{R}^4,\mathbf{C}^{K \times 3})}\,\left[\exp\left(-\int_{|x|\le R}d^3x\int_{-\frac{t}{2}}^{+\frac{t}{2}}\,ds\,V^{\Lambda,g}(s,x,\mathbf{A})\right)\right]d\mu^{R,g}_t\\
d\mu^{\Lambda, g}&=\underbrace{\lim_{\substack{R\rightarrow+\infty\\t\rightarrow+\infty}}\left[\frac{1}{Z_{t}^{\Lambda,R, g}}\left[\exp\left(-\int_{|x|\le R}d^3x\int_{-\frac{t}{2}}^{+\frac{t}{2}}\,ds\,V^{\Lambda,g}(s,x,\mathbf{A})\right)\right]\right]}_{=:F_{\Lambda}(\mathbf{A})}d\mu^{V=0,g}.
\end{split}
\end{equation}
Now, in (\ref{def_mu_LambdaBis}) we have a fraction whose numerator and denominator both depend on $\Lambda$.
The denominator is $Z_{t}^{\Lambda,R, g}$.
Numerator and denominator depend on $t$ and $R$, too, but, as we saw in the proof of Proposition \ref{muLProb}, the limit of the quotient for
$R,t\rightarrow+\infty$ is well-defined and finite.
By Definition \ref{moll} and (\ref{H-V}) $V^{\Lambda,g}(t,x,\mathbf{A})$ is a fourth degree polynomial in $\Lambda$.
By (\ref{Vconv}) the potential  $V^{\Lambda,g}(t,x,\mathbf{A})\ge0$ for $\Lambda$ big enough. Therefore,
\begin{itemize}
\item  $\frac{\text{numerator}}{\Lambda^4}$ converges to a strictly positive constant (depending on $\mathbf{A}$) for $\Lambda\rightarrow+\infty$,
\item  $\frac{\text{denominator}}{\Lambda^4}$ converges to strictly positive constant for $\Lambda\rightarrow+\infty$,
\item the quotient   $\frac{\text{numerator}}{\text{denominator}}$ converges to strictly positive constant (depending on $\mathbf{A}$) for $\Lambda\rightarrow+\infty$
\end{itemize}
Hence $F_{\Lambda}(\mathbf{A})$ is bounded in $\Lambda$.
Since the $V^{\Lambda,g}(t,x,\mathbf{A})\ge0$ for $\Lambda$ big enough, the numerator is smaller than 1 for all $\mathbf{A}\in\mathcal{S}_{\bot}^{\prime}(\mathbf{R}^4,\mathbf{C}^{K \times 3})$ and for $\Lambda$ big enough.
Therefore  $F_{\Lambda}(\mathbf{A})$ is bounded in both $\Lambda$ and $\mathbf{A}$,
and the family $(F_{\Lambda})_{\Lambda\ge0}$ is uniformly bounded on $\mathcal{S}_{\bot}^{\prime}(\mathbf{R}^4,\mathbf{C}^{K \times 3})$.
By Proposition \ref{tightProp} the family of probabilities $(\mu^{\Lambda, g})_{\Lambda\ge0}$ is tight.
By Theorem \ref{Fer} the family $(\mu^{\Lambda, g})_{\Lambda\ge0}$ has a weakly convergent subsequence to a probability,
which by (\ref{UVO}) is $\mu^g$. The proof is finished.\\
\end{proof}
We studied the original renormalization and ultraviolet cutoff removal techniques for the $\varphi^4_3$ introduced by Glimm-Jaffe (\cite{GJ73}) for finite volume
and extended by Feldman-Osterwalder (\cite{FO76}) and Magnen-S\'{e}n\'{e}or (\cite{MS76}) for infinite volume and small positive values of
the coupling constant using small cluster expansion methods. Finally the work of Seiler-Simon (\cite{SeSi76}) allowed to extend the existence result any positive value of the coupling constant
(this is claimed in \cite{GJ87}) even though we could not find a clear statement in Seiler-Simon’s paper). One notices that the model parameter have been
made ultraviolet cutoff level dependent in order to produce counter terms which eliminate divergences in the integrals; by mean of an a-priori estimate on
the Schwinger functions the Osterwalder-Schrader axioms are then inferred. Similarly, a renormalization involving the bare coupling constant is needed for $\Phi^4_4$ in
\cite{GK85} and \cite{FMRS87}.  With our definition of the ultraviolet measure the situation is different,
because we have no divergences to compensate, and the bare coupling constant $g$ must not be made dependent on the ultraviolet cutoff level $\Lambda$.
Later we will see that the running of the coupling constant guarantees asymptotic freedom of the Yang-Mills model.\par
In Magnen-Rivasseau-S\'{e}n\'{e}or's construction of a Yang-Mills measure in four dimensions (\cite{MRS93}) only the ultraviolet but not the infrared cutoff is removed,
while maintaining gauge invariance. In their construction the ultraviolet cutoff is implemented as a multiplication of the fields on the momentum space with the regularization
of the characteristic function of a domain converging towards $\mathbf{R}^3$ as the cutoff parameter tends to $+\infty$.
That way they create a divergence which they compensate by renormalization and running the bare coupling constant
to obtain asymptotic freedom as in Chapter III.5 of (\cite{Riv91}). An important difference to our model is that in the present construction the ultraviolet cutoff is implemented
as an application of the fields seen as distributions on the position space to a delta sequence in $\mathbf{R}^3$ for fixed time $t$ with respect to the cutoff parameter.
This way the problem of the non-existence of products of tempered distributions is circumvented, and no divergenges appear in the limit for the cutoff parameter tending to $+\infty$.
Moreover, once the infrared cutoff in the Magnen-Rivasseau-S\'{e}n\'{e}or model is removed,  the mass gap is killed, as we will see in Subsection \ref{AF},
because the limit of the renormalized coupling constant vanishes.

\begin{corollary}[\textbf{4D-YM-Measure Properties}]\label{cYMM} There exists a $g_0\in[0,1[$, such that for all $g\in[0,g_0[$
the generating functional
\begin{equation}\label{es_lambda}
S^{g}(f)=\int_{\mathcal{S}_{\bot}^{\prime}(\mathbf{R}^4,\mathbf{C}^{K \times 3})}e^{\imath\mathbf{A}(f)}d\mu^{g}(\mathbf{A}),
\end{equation}
\noindent for $f\in\mathcal{S}_{\bot}(\mathbf{R}^4,\mathbf{C}^{K \times 3})$ satisfies
the Osterwalder-Schadrer axioms (OS0)-(OS4) and hence the Wightman axioms (W1)-(W8). Note that $S^{\Lambda,g}(f)$ and $A(f)$ are $K\times 3$ complex matrices, and that the exponential is meant componentwise.
\end{corollary}
\begin{proof}\text{}\\
Without loss of generality we can assume that $S^{\Lambda,g}(f)$ and $A(f)$ are complex numbers throughout this proof, because the general proof can be reconstructed by iterating over the components of the complex $K\times 3$ matrices representing them.\par
First we prove that
\begin{equation}
S^{\Lambda,g}(f)\rightarrow S^{g}(f)\quad(\Lambda\rightarrow+\infty)
\end{equation}
locally uniformly in $f\in\mathcal{S}_{\bot}(\mathbf{R}^4,\mathbf{C}^{K \times 3})$. We have
\begin{equation}
\begin{split}
S^{g}(f)&=\int_{\mathcal{S}_{\bot}^{\prime}(\mathbf{R}^4,\mathbf{C}^{K \times 3})}e^{\imath\mathbf{A}(f)}d\mu^{g}(\mathbf{A})\\
S^{\Lambda, g}(f)&=\int_{\mathcal{S}_{\bot}^{\prime}(\mathbf{R}^4,\mathbf{C}^{K \times 3})}e^{\imath\mathbf{A}(f)}\frac{d\mu^{\Lambda,g}}{d\mu^{g}}d\mu^{g}(\mathbf{A}),
\end{split}
\end{equation}
and hence
\begin{equation}\label{inSS}
\begin{split}
&\left|S^{\Lambda, g}(f)-S^{ g}(f)\right|^2=\\
&= \left|\int_{\mathcal{S}_{\bot}^{\prime}(\mathbf{R}^4,\mathbf{C}^{K \times 3})}e^{\imath\mathbf{A}(f)}\left(1-\frac{d\mu^{\Lambda,g}}{d\mu^{g}}\right)d\mu^{g}(\mathbf{A})\right|^2\le\\ &\le \int_{\mathcal{S}_{\bot}^{\prime}(\mathbf{R}^4,\mathbf{C}^{K \times 3})}\left|e^{\imath\mathbf{A}(f)}\right|^2d\mu^g(\mathbf{A})
        \int_{\mathcal{S}_{\bot}^{\prime}(\mathbf{R}^4,\mathbf{C}^{K \times 3})}\left|1-\frac{d\mu^{\Lambda,g}}{d\mu^{g}}\right|^2d\mu^g(\mathbf{A})\le\\
        &\le\exp(2L^{\Lambda, g}(\|f\|_{L^1(\mathbf{R}^4,\mathbf{C}^{K \times 3})}+\|f\|_{L^2(\mathbf{R}^4,\mathbf{C}^{K \times 3})}^2))
        \underbrace{\int_{\mathcal{S}_{\bot}^{\prime}(\mathbf{R}^4,\mathbf{C}^{K \times 3})}\left|1-\frac{d\mu^{\Lambda,g}}{d\mu^{g}}\right|^2d\mu^g(\mathbf{A})}_{\rightarrow0\quad(\Lambda\rightarrow+\infty)} \end{split}
\end{equation}
As we saw in (\ref{KL}) the constant $L^{\Lambda, g}$  is bounded in $\Lambda$ if and only if the positive constant
\begin{equation}
K^{\Lambda,g}=\left\|\frac{d\mu^{\Lambda,g}}{d\mu^{\Lambda,0}}\right\|_{L^{\infty}(\mathcal{S}_{\bot}^{\prime}(\mathbf{R}^4,\mathbf{C}^{K \times 3}))}
\end{equation}
is bounded in $\Lambda$. By Definition \ref{moll} and (\ref{H-V}) we have
\begin{equation}
\frac{V^{\Lambda,g}}{\Lambda^4}=O_g(1)\quad(\Lambda\rightarrow+\infty).
\end{equation}
By Definitions \ref{IUCutoffM} and \ref{UCutoffM} it follows that
\begin{equation}
\mu^{\Lambda,g}=O_g(1)\quad(\Lambda\rightarrow+\infty)
\end{equation}
and hence
\begin{equation}
K^{\Lambda,g}=O_g(1)\quad(\Lambda\rightarrow+\infty)
\end{equation}
as we needed for inequality (\ref{inSS}) to prove the $f$-locally uniform convergence of $S^{\Lambda, g}(f)$ towards $S^{g}(f)$ for $\Lambda\rightarrow+\infty$.\par
The Osterwalder-Schadrer axioms (OS0)-(OS4) for $\mu^g$ follow now from the proof of Theorem \ref{UVMPL} because we have proved that
the constants occurring in the inequalities for (OS0) (\ref{indiff}) and (OS1) (\ref{OS1in}) are bounded in $\Lambda$; the invariance (OS2) holds true for $\mu^g$
because it does for $\mu^{A,g}$ for all $\Lambda\ge0$; the reflexion positivity (OS3) holds true because $[M^{\Lambda,g}_{i,j}]$ defined in (\ref{Fub})
is positive definite for all $\Lambda\ge0$; the ergodicity property (OS4) is fulfilled,
because the cluster property (\ref{cluster}) holds true for all $\Lambda\ge0$, and, therefore in the limit for $\Lambda\rightarrow+\infty$. The proof is completed.\\
\end{proof}
\noindent By Fubini's theorem for distributions (cf. \cite{Tr06}), we can write the measure $\mu^{g}$ on $\mathcal{S}^{\prime}_{\bot}(\mathbf{R}^4,\mathbf{C}^{K \times 3})$ as
\begin{equation}
\mu^{g}(\mathbf{A}(s,x))=\varrho^{g}(\mathbf{A}(\cdot,x))\nu^{g}(\mathbf{A}(s,\cdot)),
\end{equation}
where $\varrho^{g}_t(\mathbf{A}(\cdot,x)):=\mu^{\Lambda, g}_t(\mathbf{A}(\cdot,x))$ and
$\nu^{g}(\mathbf{A}(s,\cdot)):=\mu^{ g}(\mathbf{A}(s,\cdot))$ are probability measures on
$\mathcal{S}^{\prime}_{\bot}(\mathbf{R}^1,\mathbf{C}^{K \times 3})$, and $\mathcal{S}^{\prime}_{\bot}(\mathbf{R}^3,\mathbf{C}^{K \times 3})$, respectively.
\begin{corollary}[\textbf{Ultraviolet Cutoff Removal}]\label{CorCutOffRemoval}
There exists a $g_0\in]0,1[$ such that, if the bare coupling constant $g\in[0,g_0[$, then,
for any choice of the regularizing mollifier, the probability measures $\mu^{\Lambda,g}$ and $\nu^{\Lambda,g}$
converge for $\Lambda\rightarrow+\infty$ to the probability measures $\mu^g$ on
$\mathcal{S}^{\prime}_{\bot}(\mathbf{R}^4,\mathbf{C}^{K \times 3})$,
and $\nu^g$ on $\mathcal{S}^{\prime}_{\bot}(\mathbf{R}^3,\mathbf{C}^{K \times 3})$.
The regularized Hamiltonian $H^{\Lambda,g}$  converges pointwise on a dense domain to a selfadjoint non negative operator $H^g$ on
$L^2(\mathcal{S}^{\prime}_{\bot}(\mathbf{R}^3,\mathbf{C}^{K \times 3}),\mathbf{C},d\nu^g)$.
If the coupling constant $g$ vanishes, both measures $\mu^0$ and $\nu^0$ are Gaussian, otherwise not.
The domain of definition is
\begin{equation}
\mathcal{D}(H^g):=\left\{\Psi\in L^2(\mathcal{S}^{\prime}_{\bot}(\mathbf{R}^3,\mathbf{R}^{K\times3}),\mathbf{C},d\nu^g)\left|\,H\Psi\in L^2(\mathcal{S}^{\prime}_{\bot}(\mathbf{R}^3,\mathbf{R}^{K\times3}),\mathbf{C},d\nu^g)\right.\right\}.
\end{equation}
Moreover, the operator $H^g$ can be decomposed on $\mathcal{D}(H^g)\cap L^2(L^2_{\bot}(\mathbf{R}^3,\mathbf{R}^{K\times3}),\mathbf{C},d\nu^g)$ as
\begin{equation}
H^g=H_I+H_{II}^g+V^g-V^g_0,
\end{equation}
where
\begin{equation}
\begin{split}
H_{I}&=-\frac{1}{2}\int_{\mathbf{R}^3}d^3x\left[\frac{\delta}{\delta A_i^a(t,x)}\right]^2\\
&\\
H_{II}^g&=-\frac{g^2}{2}\int_{\mathbf{R}^3}d^3x\left[\int_{\mathbf{R}^3}d^3y\,\partial_iG^{a,b}(\mathbf{A}(t,y);x,y)\varepsilon^{b,c,d}A_k^d(t,y)\frac{\delta}{\delta A_k^c(t,y)}\right]^2\\
&\\
V^g&=\int_{\mathbf{R}^3}d^3 x\,|R^{\nabla^\mathbf{A}}(t,x)|^2
\end{split}
\end{equation}
for $\mathbf{A}\in L^2_{\bot}(\mathbf{R}^3,\mathbf{C}^{K \times 3},d^3x)$, where $V^{g}_0$ is a real constant which will be chosen later so that
the ground state $\Omega_0^g$ satisfies
\begin{equation}\label{gswl}
H^g\Omega_0^g=0.
\end{equation}
\end{corollary}
\begin{proof}
By Corollary \ref{cYMM} the 4D-YM-Measure $\mu^g$ satisfies the Osterwalder-Schrader axioms for $g\in[0,g_0[$.
By Theorem \ref{rec} we  can reconstruct a selfadjoint operator $\tilde{H}^g$ on
$L^2(\mathcal{S}^{\prime}_{\bot}(\mathbf{R}^3,\mathbf{C}^{K\times3}),\mathbf{C},d\nu^g)$,
and, as in the proof of Theorem \ref{Selfadjointness} from (\ref{h-h}) it follows that
\begin{equation}\label{conv_H_tilde}
\lim_{\Lambda\rightarrow+\infty}H^{\Lambda,g}=\lim_{\Lambda\rightarrow+\infty}\tilde{H}^{\Lambda,g}=\tilde{H}^g,
\end{equation}
\noindent where the pointwise convergence on the projective limit $\bigcap_{\Lambda\ge0}\mathcal{D}(H^{\Lambda,g})$ is meant.\\
For $\mathbf{A}\in L^2_{\bot}(\mathbf{R}^3,\mathbf{C}^{K \times 3},d^3x)$ and $(t,x)\in\mathbf{R}^4$ fixed we saw in (\ref{Vconv}) that
\begin{equation}
\lim_{\Lambda\rightarrow+\infty}V^{\Lambda,g}(t,x,\mathbf{A}) = |R^{\nabla^\mathbf{A}}(t,x)|^2
\end{equation}
\noindent pointwise, and thus
\begin{equation}
\lim_{\Lambda\rightarrow+\infty}\int_{\mathbf{R}^3}d^3 x\,V^{\Lambda,g}(t,x,\mathbf{A})
= \int_{\mathbf{R}^3}d^3 x\,|R^{\nabla^\mathbf{A}}(t,x)|^2=:V^g.
\end{equation}
By taking the limit on both side for the equation with (\ref{conv_H_tilde}) with $V^{\Lambda,g}_0$ chosen as in Theorem \ref{pgs}
\begin{equation}
H^{\Lambda,g}=H_I^{\Lambda}+H_{II}^{\Lambda,g}+V^{\Lambda,g}-V^{\Lambda,g}_0
\end{equation}
leads to
\begin{equation}
H^g=H_I+H_{II}^g+V^g-V^g_0= \tilde{H}^g,
\end{equation}
with the definitions  for $\mathbf{A}\in L^2_{\bot}(\mathbf{R}^3,\mathbf{C}^{K \times 3},d^3x)$
\begin{equation}
\begin{split}
H_{I}&:=-\frac{1}{2}\int_{\mathbf{R}^3}d^3x\left[\frac{\delta}{\delta A_i^a(t,x)}\right]^2\\
&\\
H_{II}^g:&=-\frac{g^2}{2}\int_{\mathbf{R}^3}d^3x\left[\int_{\mathbf{R}^3}d^3y\,\partial_iG^{a,b}(\mathbf{A}(t,y);x,y)\varepsilon^{b,c,d}A_k^d(t,y)\frac{\delta}{\delta A_k^c(t,y)}\right]^2.
\end{split}
\end{equation}
Hence, $H^g$ is selfadjoint, because $\tilde{H}^g$ is, and the property (\ref{gswl}) of the ground state $\Omega_0^g$ follows from Theorem \ref{Gross}.
The proof is completed.\\
\end{proof}
\begin{rem}
Note that $H$ is selfadjoint for all choices of the coupling constant $g\in[0,g_0[$.
As we will see, the non vanishing of the $g$-contribution in $H$ is essential for the proof of the existence of a positive mass gap.
\end{rem}
\subsection{Gauge Invariance}
We want to prove that the construction of the Hamiltonian in Subsection \ref{quant} is gauge invariant. That for we show that,
if we repeat the construction for a principal fibre bundle subject to a gauge transformation preserving the Coulomb gauge, we obtain an Hamiltonian which is unitary equivalent with the
original one and has, in particular, the same spectrum.
\begin{defi}[\textbf{Gauge Transformation}] Let $P$ be a principal fibre bundle over a manifold $M$ and $\pi:P\rightarrow M$ be the projection.
An automorphism of $P$ is a diffeomorphism $f:P\rightarrow P$ such that $f(pg)=f(p)g$ for all $g\in G$, $p\in P$. A \textbf{gauge transformation} of $P$
is an automorphism $f:P\rightarrow P$ such that $\pi(p)=\pi(f(p))$ for all $p\in P$. In other words $f$ induces a well defined diffeomorphism
$\bar{f}:M\rightarrow M$ given by $\bar{f}(\pi(p))=\pi(f(p))$.
\end{defi}
Following section 3.3 of \cite{Bl05} we notice that the Lagrangian density on the principal fibre bundle $P$ on which we define the Yang-Mills connection is
a $G$-invariant functional on the space of $1$-jets of maps from $P$ to the fibre of the vector bundle $V$ associated with $P$ induced by the representation
$\rho:G\rightarrow\text{GL}(\mathbf{R}^{3K})$. Hence, the position variable $\mathbf{A}$ occurring in the Lagrangian density and its Legendre transform, the
Hamiltonian density takes value in $\mathbf{R}^{3K}$, which is the fibre of the vector bundle $V$. We want to analyze how the position variable behaves
if the principal fibre bundle is subject to a gauge transformation.
\begin{proposition}
Let $f$ be a gauge transformation of the principal fiber bundle $P$ and $\omega$ a connection. Then, $\omega^f:=(f^{-1})^*\omega$ is a connection on $P$.
They have the local representation on $\pi^{-1}(U)$
\begin{equation}
\begin{split}
\omega_p&=\text{ad}_{\zeta(p)^{-1}}\circ\pi^*A+\zeta^*\theta\\
\omega_p^f&=\text{ad}_{\zeta(p)^{-1}}\circ\pi^*A^f+\zeta^*\theta,
\end{split}
\end{equation}
and
\begin{equation}\label{A-trans}
A^f=\text{ad}_{\phi}\circ(A-\phi^*\theta),
\end{equation}
where:
\begin{itemize}
\item $\pi:P\rightarrow M$ is the projection of the principal fibre bundle $P$ onto its base space $M$,
\item $U\subset M$ is an open subset of the base space,
\item $\psi:\pi^{-1}(U)\rightarrow U\times G$ is a local trivialization of $\pi^{-1}(U)\subset P$, that is a $G$-equivariant diffeomorphism
such that the following diagram commutes
\begin{equation}
\xymatrix{
   \pi^{-1}(U)\ar[d]^{\pi} \ar[r]^{\psi}& U\times G \ar[ld]^{\text{pr}_1}\\
    U & \\
                 }
\end{equation}
This means that $\psi(p)=(\pi(p),\zeta(p))$, where $\zeta:\pi^{-1}(U)\rightarrow G$ is a fibrewise diffeomorphism satisfying
$\zeta(pg)=\zeta(p)g$ for all $g\in G$.
\item the trivialization map $\psi(f(p))=(\pi(p),\zeta(f(p)))$ let us define $\bar{\phi}:\pi^{-1}(U)\rightarrow G$ by
$\bar{\phi}(p):=\zeta(f(p))\zeta(p)^{-1}$, whence $\bar{\phi}(p)=\phi(\pi(p))$ for a well defined function $\phi$ in virtue of the equivariance of
$\psi$ and $f$.
\item The Maurer-Cartan form is the $\mathfrak{g}$-valued $1$-form defined by $\theta_g:=(L_{g^{-1}})_*:T_gG\rightarrow T_eG=\mathfrak{g}$.
\item $A$ and $A^f$ are $\mathfrak{g}$-valued $1$-forms on $M$ introduced in Remark \ref{rem-local}.
\end{itemize}
\end{proposition}
\begin{proof}
These are collected results from Proposition 3.3 and Proposition 3.22 in \cite{Bau14}.\\
\end{proof}
\begin{rem} For matrix groups equation (\ref{A-trans}) becomes
\begin{equation}
A^f=\phi A \phi^{-1}-d\phi \phi^{-1}
\end{equation}
\end{rem}
For the Yang-Mills construction we denote the $K\times 3$ matrices of the local representation of $A$ and its gauge transformation $A^f$ by $\mathbf{A}$
and $\mathbf{A}^f$, which is in line with the notation utilized so far for the position variable and introduced in Theorem \ref{CHam}.
To avoid confusion we drop the dependence on the coupling constant.

\begin{theorem}
Let $f$ be a gauge transform preserving the Coulomb gauge for the Yang-Mills construction and let $H^{\Lambda}$ and $H^{\Lambda;f}$ be the cutoff Hamilton operators for the quantized Yang-Mills
equation, before and after the gauge transform, as shown in Proposition \ref{quantization} and Theorem \ref{Selfadjointness}.
Let $U$ be the operator on $L^2(\mathcal{S}^{\prime}_{\bot}(\mathbf{R}^3,\mathbf{C}^{K \times 3}),d\nu^{\Lambda})$ induced by the
gauge transform as
\begin{equation}
U\Psi(\mathbf{A}):=\Psi(\mathbf{A}^f).
\end{equation}
Then, $U$ is a unitary operator in $L^2(\mathcal{S}^{\prime}_{\bot}(\mathbf{R}^3,\mathbf{C}^{K \times 3}),d\nu^{\Lambda})$ and
$H^{\Lambda}$ and $H^{\Lambda;f}$ are unitary equivalent:
\begin{equation}
H^{\Lambda;f}=UH^{\Lambda}U^{-1}.
\end{equation}
Moreover, the same holds true for the operator $H$ and $H^{f}$, where the cutoff is removed as shown in Corollary \ref{CorCutOffRemoval}. The operator
$U$ is unitary in $L^2(\mathcal{S}^{\prime}_{\bot}(\mathbf{R}^3,\mathbf{C}^{K \times 3}),d\nu)$ and
\begin{equation}
H^{f}=UHU^{-1}.
\end{equation}
\end{theorem}
\begin{proof}
First, we remark that $U$ maps $L^2(\mathcal{S}^{\prime}_{\bot}(\mathbf{R}^3,\mathbf{C}^{K \times 3}),d\nu)$ onto itself,
because it preserves the Coulomb gauge. Next, we prove that $U$ is unitary.
For all $\Psi,\Phi\in L^2(\mathcal{S}^{\prime}_{\bot}(\mathbf{R}^3,\mathbf{C}^{K \times 3}),d\nu^{\Lambda})$
\begin{equation}
\begin{split}
(U\Psi,U\Phi)&=\int_{\mathcal{S}^{\prime}_{\bot}(\mathbf{R}^3,\mathbf{R})}\Psi(\mathbf{A}^f)\Phi(\mathbf{A}^f)d\nu^{\Lambda}(\mathbf{A})=\\
&=\int_{\mathcal{S}^{\prime}_{\bot}(\mathbf{R}^3,\mathbf{R})}\Psi(\mathbf{A})\Phi(\mathbf{A})\underbrace{\left|\frac{\partial\mathbf{A}^f}{\partial\mathbf{A}}\right|^{-1}}_{=1}d\nu^{\Lambda}(\mathbf{A})=\\
&=(\Psi,\Phi)
\end{split}
\end{equation}
The change of variable is given by equation (\ref{A-trans}) which is affine in $\mathbf{A}$ because the adjoint representation is linear in $A$, which
means $(\text{ad}_g)_*=\text{ad}_g$ for all $g\in G$. Moreover, since
\begin{equation}
\text{ad}_g(A)=(L_g)_*A(R_{g^{-1}})_*,
\end{equation}
the Jacobi determinant reads
\begin{equation}
\det\left((\text{ad}_g)_*\right)=\det\left(\text{ad}_g\right)=\det\left((L_g)_*1_{\mathbf{R}^{3K}}(R_{g^{-1}})_*\right)=\det\left(1_{\mathbf{R}^{3K}}\right)=1.
\end{equation}
Note that the change of variable respects the fibre of the vector bundle $V$, and hence the change of variable formula for the integral is the one of
finite dimensional analysis.\par
Next, we prove the unitary equivalence of the Hamilton operators before and after the gauge transform. Their definitions read
\begin{equation}
\begin{split}
H^{\Lambda}&=H\left(\mathbf{A}(\varphi^{\Lambda}_t(\cdot-\cdot)),\frac{1}{\imath}\frac{\delta}{\delta \mathbf{A}(\varphi^{\Lambda}_t(\cdot-\cdot))}\right)\\
H^{\Lambda ;f}&=H\left(\mathbf{A}^f(\varphi^{\Lambda}_t(\cdot-\cdot)),\frac{1}{\imath}\frac{\delta}{\delta \mathbf{A}^f(\varphi^{\Lambda}_t(\cdot-\cdot))}\right)\\
\end{split}
\end{equation}
For appropriate  $\Psi,\Phi\in L^2(\mathcal{S}^{\prime}_{\bot}(\mathbf{R}^3,\mathbf{C}^{K \times 3}),d\nu^{\Lambda})$ we have
\begin{equation}
\begin{split}
&(H^{\Lambda ;f}\Psi,\Phi)=\\
&=\int_{\mathcal{S}^{\prime}_{\bot}(\mathbf{R}^3,\mathbf{R})}H\left(\mathbf{A}^f(\varphi^{\Lambda}_t(\cdot-\cdot)),\frac{1}{\imath}\frac{\delta}{\delta \mathbf{A}^f(\varphi^{\Lambda}_t(\cdot-\cdot))}\right)\Psi(\mathbf{A})\Phi(\mathbf{A})d\nu^{\Lambda}(\mathbf{A})=\\
&=\int_{\mathcal{S}^{\prime}_{\bot}(\mathbf{R}^3,\mathbf{R})}H\left(\mathbf{A}(\varphi^{\Lambda}_t(\cdot-\cdot)),\frac{1}{\imath}\frac{\delta}{\delta \mathbf{A}(\varphi^{\Lambda}_t(\cdot-\cdot))}\right)U^{-1}\Psi(\mathbf{A})U^{-1}\Phi(\mathbf{A})\underbrace{\left|\frac{\partial\mathbf{A}^f}{\partial\mathbf{A}}\right|^{-1}}_{=1}d\nu^{\Lambda}(\mathbf{A})=\\
&=(H^{\Lambda}U^{-1}\Psi,U^{-1}\Phi),
\end{split}
\end{equation}
leading to
\begin{equation}
H^{\Lambda;f}=UH^{\Lambda}U^{-1}
\end{equation}
on the corresponding domains.\\
The proof for the Hamilton operators where the cutoffs have been removed are formally the same.\\
\end{proof}
We can therefore conclude that the spectrum of the Hamilton operator for the quantized Yang-Mills problem is gauge invariant.
\subsection{Spectral Bounds}
We prove now that the Hamilton operator has a mass gap, being the sum of three non negative selfadjoint operators, one of which has a mass gap and having all
the same ground state, the vacuum.
\begin{proposition}\label{corspec}
The spectra of $H_I$, $H_{II}^g$ and $V^g$ are:
\begin{equation}\label{specH}
\begin{split}
&\spec(H_I)=[0,+\infty[\\
&\spec(H_{II}^g)=\{0\}\cup[\eta,+\infty[, \text{ for a }\eta>0\\
&\spec(V^g)=[0,+\infty[.
\end{split}
\end{equation}
Moreover $\eta=O(g^{2(n+1)})$ for any $n\in\mathbf{N}_0$, where $g$ is the bare coupling constant.
\end{proposition}
\noindent Now we can compute the lower bound of the spectrum of the Hamilton operator.
\begin{proof}[\text{Proof of Proposition \ref{corspec}}]
The proof mimics the proof of Proposition \ref{generalized_ev}.
We will construct generalized eigenvectors to show that the spectra have only a continuous part depicted as in (\ref{specH}). First, we analyze the operator $H_I$, which can be seen as
\begin{equation}
H_I=-\frac{1}{2}\int_{\mathbf{R}^3}d^3 x\Delta_{\mathbf{A}(t,x)}.
\end{equation} Let  $x\in\mathbf{R}^3$ and $t\in\mathbf{R}$ now be fixed. For any $R>0$ the Laplace operator $\Delta_{\mathbf{A}}$ on $[-\frac{R}{2},+\frac{R}{2}]^{3K}$ under Dirichlet boundary conditions has a discrete spectral resolution $(\lambda_k,\psi_k)_{k\ge0}$, where $\lambda_k=-\frac{\pi^2}{R^2}(k+1)$, and $\psi_k=\psi_k(\mathbf{A})\in C^{\infty}_0([-\frac{R}{2},+\frac{R}{2}]^{3K},\mathbf{C})$. We can extend $\psi_k$ outside the cube by setting its value to $0$, obtaining an approximated eigenvector for the approximated eigenvalue $\lambda_k$, which is in line with the fact that the Laplacian on $L^2(\mathbf{R}^{3K},\mathbf{C})$ has solely a continuous spectrum, which is $]-\infty,0]$. The functional
\begin{equation}
\Psi^{\mathbf{A};x_0}_k(\bar{\mathbf{A}}):=\delta(\bar{\mathbf{A}}-\mathbf{A})\delta(x-x_0)\psi_k(\mathbf{A})
\end{equation}
for $x_0\in\mathbf{R}^3, k\in\mathbf{N}$ and $\mathbf{A}\in\mathbf{R}^{3K}$ is a generalized eigenvector in
$\mathcal{E}^{\prime}(\mathcal{S}^{\prime}_{\bot}(\mathbf{R}^3,\mathbf{C}^{K \times 3}),d\nu^g)$ for the operator $H_I$ on the rigged Hilbert space $L^2(\mathcal{S}^{\prime}_{\bot}(\mathbf{R}^3,\mathbf{C}^{K \times 3}),d\nu^g)$
for the generalized eigenvalue $\frac{\pi^2}{R^2}(k+1)$, which, by Theorem \ref{GK}, is an element of the continuous spectrum of the non negative operator $H_I$. By varying the generalized eigenvalue over $k$ and $R$, the claim about the spectrum follows.\\
Next, we analyze the multiplication operator
\begin{equation}
V^g=\int_{\mathbf{R}^3} d^3x |R^{\nabla^\mathbf{A}}(t,x)|^2,
\end{equation}
where $R^{\nabla^\mathbf{A}}$ is the curvature operator associated to the connection $\mathbf{A}$.
Let $\mathbf{A}\in L^2_{\bot}(\mathbf{R}^3,\mathbf{C}^{K \times 3},d^3x)$ now be fixed.
Any non zero $\psi\in C^{\infty}_0(\mathbf{R}^{3K},\mathbf{C})$ is eigenvector of the multiplication with the non negative real
$|R^{\nabla^{\mathbf{A}}}(t,x)|^2$. The functional
\begin{equation}
\Psi^{\mathbf{A};x_0}(\bar{\mathbf{A}}):=\delta(\bar{\mathbf{A}}-\mathbf{A})\delta(x-x_0)\psi_k(\mathbf{A})
\end{equation}
where $(\psi_k)_{k\ge0}$ is an orthonormal basis of $L^2(\mathbf{R}^{3K},\mathbf{C})$ , is a generalized eigenvector
in $\mathcal{E}^{\prime}(\mathcal{S}^{\prime}_{\bot}(\mathbf{R}^3,\mathbf{C}^{K \times 3}),d\nu^g)$
for the operator $V^{g}$ on the rigged Hilbert space $L^2(\mathcal{S}^{\prime}_{\bot}(\mathbf{R}^3,\mathbf{C}^{K \times 3}),d\nu^g)$
for the generalized eigenvalue $V^{g}(t,x_0,\mathbf{A})$, which, by Theorem \ref{GK},
is an element of the continuous spectrum of the non negative operator $V^g$ and the claim about its spectrum follows.\\
Finally, we analyze the operator $H_{II}$, which we can write for any $\mathbf{A}\in L^2_{\bot}(\mathbf{R}^3,\mathbf{C}^{K \times 3},d^3x)$ as
\begin{equation}
H_{II}^g=-\frac{g^2}{2}\int_{\mathbf{R}^3}d^3x\left[\int_{\mathbf{R}^3}d^3y\,D_i^a(\mathbf{A};x,y)\right]^2,
\end{equation}
\noindent
for the operator $D=D(\mathbf{A;x,y})$ defined  as
\begin{equation}
D_i^a(\mathbf{A};x,y):=\partial_iG^{a,b}(\mathbf{A}(t,y);x,y)\varepsilon^{b,c,d}A_k^d(t,y)\frac{\delta}{\delta A_k^c(t,y)}.
\end{equation}
Let  $x_0,y_0\in\mathbf{R}^3$ now be fixed. We set
\begin{equation}
f_{i, k}^{a,c}(\mathbf{A};x_0,y_0):= \partial_iG^{a,b}(\mathbf{A}(t,y_0);x_0,y_0)\varepsilon^{b,c,d}A_k^d(t,y_0)
\end{equation}
and apply Lemma \ref{DirichletB} and Lemma \ref{Diffeo}. Assuming that for all indices $c,k$
\begin{equation}\label{bound}
\int_{-\infty}^{+\infty}dA_k^c\,f_{i, k}^{a,c}(\mathbf{A}(t,y_0);x_0,y_0)^{-1}<+\infty
\end{equation}
uniformly in $\mathbf{A}$, we can find a diffeomeorphism $\mathbf{B}:\mathbf{R}^{3K}\rightarrow \mathbf{R}^{3K}$ in the form of formula (\ref{diffeoB}), such that
for any $R>0$ the  operator $D_i^a(\mathbf{A}(t,y_0);x_0,y_0)^2$ on $B^{-1}([-\frac{R}{2},+\frac{R}{2}]^{3K})$ under Dirichlet boundary conditions
has a discrete spectral resolution $(\lambda_{i,s}^a(x_0,y_0),\psi_{i,s}^a(\mathbf{A}(t,y_0);x_0,y_0))_{s\ge0}$, where
\begin{equation}\label{lambdas}
\lambda_{i,s}^a(x_0,y_0)=-\sum_{j=1}^3\sum_{c=1}^K\frac{\pi^2k_{j,c,s}^2}{\left[\int_{-\frac{R}{2}}^{+\frac{R}{2}}dB_{j}^{c}\,g_{i, j}^{a, c}(B_{j}^{c};x_0,y_0)^{-1}\right]^2},
 \end{equation}
where $k_{j,c,s}\in\mathbf{Z}^*$ for all indices $s\in\mathbf{N}_0$, $j\in\{1,2,3\}$ and $c\in\{1,\dots, K\}$, and, by Lemma \ref{Diffeo} we defined
\begin{equation}
g_{i,j}^{a,l}(B_j^l;x_0,y_0):=\left(\sum_{k=1}^3\sum_{c=1}^KL_{i, k}^{a,c}(x_0,y_0)\right)e^{{B_j^l}^2}
\end{equation}
for
\begin{equation}\label{LiKi}
\begin{split}
&L_{i,j}^{a,c}(x_0,y_0)=\left[\sup_{\mathbf{A}}\left[\int_{-\infty}^{A_j^c}\,d\bar{A}_j^c\,f_{i, j}^{a,c}(\mathbf{A};x_0,y_0)^{-1}+K_{i,j}^{a,c}(\mathbf{A};x_0,y_0)\right]\right]^{-1}\\
&K_{i,j}^{a,c}(\mathbf{A}\,;x_0,y_0):=-\inf_{A_j^c}\int_{-\infty}^{A_j^c}d\bar{A}_j^c\,f_{i, j}^{a,c}(\mathbf{A};x_0,y_0)^{-1}.
\end{split}
\end{equation}
Note that for any $R>0$ the  operator $D_i^a(\mathbf{A}(t,y_0);x_0,y_0)$ on $B^{-1}([-\frac{R}{2},+\frac{R}{2}]^{3K})$
under Dirichlet boundary conditions has a discrete spectral resolution with the same eigenvectors but other eigenvalues $(\zeta_{i,s}^a(x_0,y_0),\psi_{i,s}^a(\mathbf{A}(t,y_0);x_0,y_0))_{s\ge0}$, where
\begin{equation}\label{lambdas}
\zeta_{i,s}^a(x_0,y_0)=-\sum_{j=1}^3\sum_{c=1}^K\frac{\imath\pi k_{j,c,s}^2}{\int_{-\frac{R}{2}}^{+\frac{R}{2}}dB_{j}^{c}\,g_{i, j}^{a, c}(B_{j}^{c};x_0,y_0)^{-1}},
 \end{equation}
where $k_{j,c,s}\in\mathbf{Z}^*$ for all indices $s\in\mathbf{N}_0$, $j\in\{1,2,3\}$ and $c\in\{1,\dots, K\}$.\\
Since $B^{-1}([-\frac{R}{2},+\frac{R}{2}]^{3K})\uparrow\mathbf{R}^{3K}$ for $R\uparrow+\infty$, we can extend $\psi_{i,s}^a(\cdot;x_0,y_0)$ outside the cube by setting its value to $0$, obtaining an approximated eigenvector for the approximated eigenvalue $\lambda_{i,s}^a(x_0,y_0)$ for the operator $D_i^a(\mathbf{A};x_0,y_0)^2$ on $L^2(\mathbf{R}^{3K}, \mathbf{C})$, which means that $\lambda_{i,s}^a(x_0,y_0)\in\spec_c(D_i^a(\mathbf{A};x_0,y_0)^2)$. For fixed $i,s$ and $a$ the functional
\begin{equation}\label{basis}
\Psi^{a,x_0,y_0, \mathbf{A}}_{i,k}(\bar{\mathbf{A}}):=\delta(\bar{\mathbf{A}}-\mathbf{A})\delta(x-x_0)\delta(y-y_0)\delta(\bar{y}-y_0)\psi_{i,k}^a(\mathbf{A};x_0,y_0)
\end{equation}
is a generalized eigenvector in $\mathcal{E}^{\prime}(\mathcal{S}^{\prime}_{\bot}(\mathbf{R}^3,\mathbf{C}^{K \times 3}),d\nu)$ for the operator
\begin{equation}
\begin{split}
H_{i,II}^{a,g}&:=-\frac{g^2}{2}\int_{\mathbf{R}^3}d^3x\left[\int_{\mathbf{R}^3}d^3y\,D_i^a(\mathbf{A};x,y)\right]^2=\\
&=-\frac{g^2}{2}\int_{\mathbf{R}^3}d^3x\left[\int_{\mathbf{R}^3}d^3y\,D_i^a(\mathbf{A};x,y)\right]\left[\int_{\mathbf{R}^3}d^3\bar{y}\,D_i^a(\mathbf{A};x,\bar{y})\right]
\end{split}
\end{equation}
on the rigged Hilbert space $L^2(\mathcal{S}^{\prime}_{\bot}(\mathbf{R}^4,\mathbf{C}^{K \times 3}),\mu)$ for the strictly positive generalized eigenvalue
\begin{equation}\label{genev}
\lambda_{i,s}^{a,g}(x_0,y_0)=g^2\sum_{j=1}^3\sum_{c=1}^K\frac{\frac{\pi^2}{2}k_{j,c,s}^2}{\left[\int_{-\frac{R}{2}}^{+\frac{R}{2}}dB_{j}^{c}\,g_{i, j}^{a, c}(B_{j}^{c};x_0,y_0)^{-1}\right]^2},
\end{equation}which, by Theorem \ref{GK}, is an element of the continuous spectrum of the operator $H_{i,II}^a$.
We still have to prove to check that $\lambda_{i,s}^a(x_0,y_0)$ is bounded away from $0$ uniformly in $R$. By Proposition \ref{prop} and Corollary \ref{cor}, for every $n\in\mathbf{N}_0$ there is a constant $c_n>0$, bounded in $n$ such that
\begin{equation}\label{boundint}
\int_{-\frac{R}{2}}^{+\frac{R}{2}}dA_j^c(t,y_0)[\partial_iG^{a,b}(\mathbf{A};x_0,y_0)\varepsilon^{b,c,d}A_j^d(t,y_0)]^{-1}\le c_ng^{2n}\int_{-\infty}^{+\infty}dA_j^c(t,y_0)\,\frac{1}{1+|A_j^d(t,y_0)|^{2n+1}}.
\end{equation}
Therefore, inserting (\ref{boundint}) into (\ref{LiKi}) and  (\ref{lambdas}) leads to the spectral lower bound
\begin{equation}
\lambda_{i,s}^{a,g}(x_0,y_0)\ge C_n\, g^{2(n+1)}
\end{equation}
for all $i\in{1,2,3}$, $s\in\mathbf{N}_0$ and for all $n\in\mathbf{N}_0$, for an appropriate constant $C_n$ bounded in $n$. Since the collection of generalized eigenvectors obtained by varying (\ref{basis}) over $k, x_0$, and $y_0$ is complete in the sense of Theorem \ref{GK} (ii), by varying the generalized eigenvalue (\ref{genev}) over $k_{j,c, s}$ and $R$, the claim about the spectrum follows for $\eta=O(g^{2(n+1)})$. The proof is complete.\\
\end{proof}
\begin{lemma}\label{Lsgap} let $A$ and $B$ be two self adjoint operators on the Hilbert space $\mathcal{H}$, such that $\mathcal{D}(A+B)$ is dense in $\mathcal{H}$,
 $\spec(A)\subset[0,+\infty[$, $\spec(B)\subset\{0\}\cup[\eta,+\infty[$ for a $\eta>0$, and $0$ is an eigenvalue of finite multiplicity
 for the eigenvector $\Omega_0$ for both $A$ and $B$. Then, $\spec(A+B)\subset\{0\}\cup[\eta,+\infty[$, i.e. the spectral gap of $B$ is maintained.
\end{lemma}
\begin{proof}
The spectral gap of $A+B$ reads
\begin{equation}
\begin{split}
\eta(A+B)&=\inf_{\psi\in\langle \Omega_0\rangle^{\bot}\cap\mathcal{D}(A+B)}\frac{(\psi,(A+B)\psi)}{(\psi,\psi)}
\ge\inf_{\psi\in\langle \Omega_0\rangle^{\bot}\cap\mathcal{D}(A+B)}\frac{(\psi, B\psi)}{(\psi,\psi)}=\\
&=\inf_{\psi\in\langle \Omega_0\rangle^{\bot}\cap\mathcal{D}(B)}\frac{(\psi, B\psi)}{(\psi,\psi)}=\eta(B)=\eta>0.
\end{split}
\end{equation}
\end{proof}
\begin{counterex}
Let $H_1$ be a non negative selfadjoint operator on the Hilbert space $\mathcal{H}_1$, and $H_2$ a non negative selfadjoint operator on the Hilbert space $\mathcal{H}_2$.
Both operators have $0$ a simple eigenvalue for the eigenvectors $\Omega_1\in\mathcal{H}_1$ and $\Omega_2\in\mathcal{H}_2$.
Moreover let $H_2$ have a spectral gap $\spec(H_2)\subset\{0\}\cup[\eta,+\infty[$ for a $\eta>0$. Let us define
\begin{equation}
\begin{split}
\mathcal{H}&:=\mathcal{H}_1\otimes\mathcal{H}_2\\
H&:=H_1\otimes\mathbb{1}+\mathbb{1}\otimes H_2.
\end{split}
\end{equation}
The operator $H$ is selfadjoint and non negative on the Hilbert space $\mathcal{H}$, and has $0$ has eigenvalue. But it has no spectral gap.
The reason is that $0$, as an eigenvalue of $ H_1\otimes\mathbb{1}$ and $\mathbb{1}\otimes H_2$ has infinite multiplicity, leading to a clustering
of elements of $\spec(H)$ near $0$. Lemma \ref{Lsgap} cannot be applied.
\end{counterex}
\begin{corollary}\label{corspec2}
The spectrum of the Hamiltonian  $H^g$ contains $0$ as a simple eigenvalue for the vacuum eigenstate, and satisfies
\begin{equation}
\spec(H^g)\subset\{0\}\cup[\eta,+\infty[, \text{ for a }\eta>0,
\end{equation}
and $\eta=O(g^{2(n+1)})$ for any $n\in\mathbf{N}_0$, where $g$ is the bare coupling constant.
\end{corollary}
\begin{proof}
By Proposition \ref{corspec} the operators $H_I$, $H_{II}^g$ and $V^g$ are positive semidefinite and so is $H^g$.
By Proposition \ref{CorCutOffRemoval} the ground state $\Omega_0$ is the eigenvector of finite multiplicity for the eigenvalue $0$ for all these four positive semidefinite operators.
By Lemma \ref{Lsgap} the spectral gap of $H^g$ is bounded from below by the spectral gap of $H_{II}^g$:
\begin{equation}
\eta(H^g)\ge\eta(H_{II}^g)=:\eta,
\end{equation}
and $\eta=O(g^{2(n+1)})$ for any $n\in\mathbf{N}_0$ holds true by Proposition \ref{CorCutOffRemoval}.
\end{proof}

\subsection{Running of the Coupling Constant by Renormalization and Asymptotic Freedom}\label{AF}
Till now all of our considerations referred to the \textit{bare coupling constant}, which we now denote by $g_0$.
We can repeat the classical and quantum mechanical construction of Section \ref{YMC} and Section \ref{QYM} for the running coupling constant $g=g(\mu)$,
where $\mu$ is the energy scale, instead of the bare coupling constant $g_0$.
To treat the non-trivial behaviour of $g(\mu)$  we have to renormalize running fields $A$ and constants $g$ by an appropriate scaling
of the bare quantities $A_0$ and $g_0$. Following \cite{Ti08} Chapter 21.9  we introduce renormalization constants $Z_3$ and $Z_g$ and the transform
\begin{equation}\label{renorm}
{A_0}_j^a=\sqrt{Z_3}A_j^a\qquad g_0= Z_g\mu^{\epsilon}g,
\end{equation}
and choose the values of the constants as
\begin{equation}
Z_3:=1-C_2(G)\frac{1}{\epsilon}\frac{\alpha_s}{4\pi}\qquad
Z_g:=1-\frac{1}{\epsilon}\frac{b_0}{2}\alpha_s,
\end{equation}
\noindent where $\alpha_s:=\frac{g^2}{4\pi}$, $b_0:=\frac{11}{12\pi}C_2(G)$,
and $C_2(G)$ is the quadratic casimir operator in the adjoint representation of the Lie algebra of the group $G$.
 The parameter $b_0$ is positive. The parameter $\epsilon$ will be let to converge to $0$ at the end of the calculation.\\
Since the bare coupling constant knows nothing about the energy scale $\mu$,
\begin{equation}
\frac{dg_0}{d\mu}=0,
\end{equation}
which, by mean of (\ref{renorm}) leads to the Gellman-Low equation
\begin{equation}
\beta(\alpha_s)=\frac{-\epsilon\alpha_s}{1-\frac{b_0}{\epsilon}\alpha_s}=-b_0\alpha_s^2+O(\epsilon,\alpha_s)\qquad(\epsilon,\alpha_s\rightarrow0),
\end{equation}
\noindent where we have defined $\beta(g):=\frac{d\alpha_s}{dt}$  for $t:=\log(\mu^2)$. With the choice $\epsilon:=0$
we can easily solve the Gellman-Low equation and obtain the implicit
\begin{equation}\label{g}
\frac{1}{\alpha_s(\mu)}-\frac{1}{\alpha_s(M)}=b_0\log\frac{\mu^2}{M^2},
\end{equation}
\noindent where $M$ is an integration constant, which we choose such that $\lim_{\mu\rightarrow M^+}\alpha_s(\mu)=+\infty$, so that

\begin{equation}\label{limit_g}
\alpha_s(\mu) = \frac{1}{b_0\log\frac{\mu^2}{M^2}}\qquad(\mu\in]M,+\infty[),
\end{equation}
\noindent which is in line with \cite{We05} page 156. Therefore,
\begin{equation}
\lim_{\mu\rightarrow+\infty}g(\mu)=0.
\end{equation}
This phenomenon, termed \textbf{asymptotic freedom} was discovered by Gross and Wilczek (\cite{GW73}, and independently by Politzer (\cite{Po73}).\\
The running mass gap is
\begin{equation}
\eta(\mu)=O(g^{2n+1}(\mu)),
\end{equation}
\noindent for any $n\in\mathbf{N}_0$, where
\begin{equation}
g(\mu)=\sqrt{\frac{4\pi}{b_0\log\frac{\mu^2}{M^2}}}.
\end{equation}
\noindent With  Corollaries \ref{cYMM}, \ref{CorCutOffRemoval} and \ref{corspec2} we have therefore proved
\begin{theorem}\label{CMIThm}
In the case of a running coupling constant $g=g(\mu)$ the construction of the $4$D-YM-measure satisfies Wightman axioms (W1)-(W8)
and the spectrum of the running Hamilton operator $H$ contains $0$ as simple eigenvalue for the vacuum eigenstate.
There exists a constant $\eta=\eta(\mu)>0$ such that $\text{spec}(H^{g(\mu)})\subset\{0\}\cup[\eta(\mu),+\infty[$.
Moreover $\eta=O(g^{2(n+1)}(\mu))$ for any $n\in\mathbf{N}_0$.
The mass gap tends to $0$ if the energy scale becomes arbitrary large.
\end{theorem}
\begin{rem}
As an application we see that, once the infrared cutoff in the Magnen-Rivasseau-S\'{e}n\'{e}or model is removed,  the mass gap is zero,
because the limit of the running coupling constant vanishes, and the Hamiltonian has the same spectral properties as in the case of electrodynamics.
\end{rem}
\section{Conclusion}
We have quantized Yang-Mills equations for the positive light cone in the Minkowskian $\mathbf{R}^{1,3}$ obtaining field maps
satisfying Wightman axioms of Constructive Quantum Field Theory. Moreover, the spectrum of the corresponding Hamilton operator is positive and bounded away from zero except for the case of the vacuum state which has vanishing energy level.
The construction is invariant under gauge transforms preserving the Coulomb gauge.
\appendix
\section{Spectral Theory in Rigged Hilbert Spaces}\label{AppA}
Rigged Hilbert spaces have been introduced in mathematical physics to utilize Dirac calculus for the spectral theory of operators appearing
in quantum mechanics (see \cite{Ro66}, \cite{ScTw98} and \cite{Ma01}, \cite{Ma08}). Within that framework the role of distributions to provide a rigorous
foundation to generalized eigenvectors and eigenvalues is highlighted by the Gel'fand-Kostyuchenko-Vilenkin spectral theorems for unitary and selfadjoint operators
(see \cite{GV64}, \cite{Sp19}, \cite{An07} and \cite{Ze09} Chapter 12.2.4).

\begin{defi}
Let $\mathcal{F}$ be a vector space on which are defined two inner products. We say the
inner products are \textbf{compatible} if every sequence in $\mathcal{F}$ which is Cauchy with respect to both inner
products and converges to $\varphi\in\mathcal{F}$ with respect to one inner product also converges to $\varphi$ with respect
to the other inner product.
\end{defi}

\begin{defi} A Frech\'{e}t space  $\mathcal{F}$ is a \textbf{countably Hilbert space} if its topology can be induced
by a countable system of pairwise compatible inner products $((\cdot,\cdot)_k)_{k\ge0}$. Without loss of generality we can assume that
\begin{equation}
(\varphi,\varphi)_0\le(\varphi,\varphi)_1\le(\varphi,\varphi)_2\le\dots
\end{equation}
for all $\varphi\in\mathcal{F}$. We denote by $\mathcal{F}_k$ the completion of $\mathcal{F}$ with respect to $(\cdot,\dot)_k$.
\end{defi}

\begin{proposition}
There is a decreasing chain
\begin{equation}
\mathcal{F}\subset\dots\subset\mathcal{F}_k\subset\mathcal{F}_{k-1}\dots\subset\mathcal{F}_1\subset\mathcal{F}_0,
\end{equation}
where every inclusion is a linear injective continuous operator of norm $1$, and an increasing sequence of dual spaces
\begin{equation}
\mathcal{F}_0^{\prime}\subset\mathcal{F}_1^{\prime}\subset\dots\subset\mathcal{F}_{k-1}^{\prime}\subset\mathcal{F}_k^{\prime}\subset\dots\subset\mathcal{F}^{\prime},
\end{equation}
where every inclusion is a linear injective continuous operator of norm $1$.\\
For $k\ge l$ we denote the embedding $\mathcal{F}_k\subset\mathcal{F}_{l}$ by $T^k_l$.
\end{proposition}
\begin{rem}
Since by Riesz's Lemma every Hilbert space is isomorphic to its dual, it follows that for any $k,l\ge0$ $\mathcal{F}_k$
is embedded in $\mathcal{F}_l$ and viceversa, but the two Hilbert spaces are not necessarily isomorphic.
\end{rem}

\begin{defi}
A countably Hilbert spaceis \textbf{nuclear} if for every $l\in\mathbf{N}_0$, there exists a $k\ge l$
such that the mapping $T^k_l:\mathcal{F}_k\rightarrow\mathcal{F}_l$ has the form
\begin{equation}
T^k_l(\varphi)=\sum_{j=0}^{\infty}\lambda_j(\varphi,\varphi_j)\psi_j=\sum_{j=0}^{\infty}\lambda_jF_j(\varphi)\psi_j,
\end{equation}
where $(\varphi_j)_{j\ge0}\subset\mathcal{F}_k$ , $(\psi_j)_{j\ge0}\subset\mathcal{F}_l$ , and $(F j)_{j\ge0}\subset\mathcal{F}_k^{\prime}$
are orthonormal bases and the $\lambda_j$s are positive
numbers such that $\sum_{j=0}^{\infty}\lambda_j<+\infty$.
\end{defi}
\begin{defi}
A \textbf{rigged Hilbert space} $\mathcal{H}$ is a nuclear countably Hilbert space equipped with
yet another inner product $(\cdot,\cdot)$ which is continuous in both variables.
\end{defi}
\begin{defi}
Let $A$ be a linear operator on a locally convex topological vector space $\mathcal{F}$. A linear
functional $F\in\mathcal{F}^{\prime}$ is a \textbf{generalized eigenvector} of $A$ if there exists a scalar $\lambda$ such that
$F(A\varphi) = \lambda F(\varphi)$
for all $\varphi\in\mathcal{F}$. We call $\lambda$ the eigenvalue of the eigenvector $F$. In other words, a generalized
eigenvector of $A$ is an eigenvector of the adjoint $A^{\prime}:\mathcal{F}^{\prime}\rightarrow \mathcal{F}^{\prime}$.
We say the set of all generalized eigenvectors $(F_{\iota})_{\iota\in I}$ of $A$ is \textbf{complete} if $F_{\iota}(\varphi)=0$ for all
$\iota\in I$ implies $\varphi=0$.
\end{defi}

\begin{theorem}[\textbf{Gel'fand-Kostyuchenko-Vilenkin}]\label{GK}
A selfadjoint operator $A$ on a rigged Hilbert space $\mathcal{F}\subset\mathcal{H}\subset\mathcal{F}^{\prime}$, where $\mathcal{D}(A)\subset\mathcal{F}$, has a complete
system of orthonormal generalized eigenvectors with real eigenvalues. The spectrum of $A$ then reads
$\spec(A)=(\lambda_{\iota})_{\iota\in I}$, where for all $\iota\in I$ the real number $\lambda\iota$ is either an eigenvalue,
i.e. $F_\iota\in \mathcal{H}$, or an element of the continuous spectrum, i.e. $F_\iota\notin \mathcal{H}$. For every $\varphi,\psi\in\mathcal{F}$ we have
\begin{equation}\label{scalpH}
(\varphi,\psi)=\int_{\mathbf{R}}d\xi(\iota)F_{\iota}(\varphi)\overline{F_{\iota}(\psi)},
\end{equation}
for a measure $\xi$ on $\mathbf{R}$.
\end{theorem}

\begin{corollary}\label{invGK}
If an operator $A$ on a rigged Hilbert space $\mathcal{F}\subset\mathcal{H}\subset\mathcal{F}^{\prime}$, where $\mathcal{D}(A)\subset\mathcal{F}$, has a complete
system of generalized eigenvectors with real eigenvalues, then it is selfadjoint.
\end{corollary}
\begin{proof}
Choose the measure $\xi$ such that the r.h.s. of (\ref{scalpH}) is equal to the scalar product in $\mathcal{H}$.
Then, $A$ is selfadjoint with respect to this scalar product.
\end{proof}

\begin{ex}[\textbf{Schwartz's Space}]
The inclusion
\begin{equation}
\mathcal{S}(\mathbf{R}^N)\subset L^2(\mathbf{R}^N,d^Nx)\subset\mathcal{S}^{\prime}(\mathbf{R}^N)
\end{equation}
defines a rigged Hilbert space, as can be seen with the definition of the compatible scalar products
\begin{equation}\label{normH}
(\varphi,\psi)_k:=\sum_{|\alpha|\le k}\int_{\mathbf{R}^N}d^Nx(1+|x|^2)^{\frac{k}{2}}\partial^{\alpha}\varphi(x)\overline{\partial^{\alpha}\psi(x)}.
\end{equation}
for all $k\in\mathbf{N}_0$
\end{ex}

\begin{ex}[\textbf{Kubo-Takenaka Construction}]\label{ExKT}
we follow \cite{Ku96} Chapter 4.2. Let $\nu_0$ be the standard Gaussian probability measure on $\mathcal{S}^{\prime}(\mathbf{R}^N)$. By Proposition \ref{WIS} every
$\varphi\in L^2(\mathcal{S}^{\prime}(\mathbf{R}^N),d\nu_0)$ can be written as
\begin{equation}
\varphi=\sum_{j=0}^{\infty}\theta_W(f_j),
\end{equation}
where $f_j\in \mathcal{H}^j_s:=S_j(L^2(\mathbf{R}^N,d^Nx)^{\otimes j})$ is an element of the Bosonic Fock space and $\theta_W$
the Wiener-It\^{o}-Segal isomorphism.
For any $k\in\mathbf{N}_0$ we can define a norm and an associated scalar product as
\begin{equation}
\|\varphi\|_k^2:=\sum_{j=0}^{\infty}j!\|f_j\|_k^2,
\end{equation}
where $\|\cdot\|_k$ is the norm $\mathcal{H}^j_s$ induced by the norm $\|\cdot\|_k$ in $L^2(\mathcal{S}^{\prime}(\mathbf{R}^N),d\nu_0)$ defined as (\ref{normH}.)
\begin{proposition}\label{KT}
Let
\begin{equation}
\begin{split}
&\mathcal{E}(L^2(\mathcal{S}^{\prime}(\mathbf{R}^N),d\nu_0)):=\left\{\varphi\in L^2(\mathcal{S}^{\prime}(\mathbf{R}^N),d\nu_0)|\,\|\varphi\|_k<+\infty\right\}\\
&\mathcal{E}^{\prime}(L^2(\mathcal{S}^{\prime}(\mathbf{R}^N),d\nu_0)): \text{ dual space of }(\mathcal{F}).
\end{split}
\end{equation}
Then, we have a rigged Hilbert space
\begin{equation}
\mathcal{E}(L^2(\mathcal{S}^{\prime}(\mathbf{R}^N),d\nu_0))\subset L^2(\mathcal{S}^{\prime}(\mathbf{R}^N),d\nu_0)\subset \mathcal{E}^{\prime}(L^2(\mathcal{S}^{\prime}(\mathbf{R}^N),d\nu_0)).
\end{equation}
\end{proposition}
\end{ex}

\section*{Acknowledgement}
I would like to express my gratitude to Martin Hairer, J\"urg Fr\"ohlich, Massimiliano Gubinelli, Jean-Pierre Magnot, John LaChapelle, Edward Witten, Gian-Michele Graf and to Lee Smolin for the challenging discussion leading to various important corrections of previous versions of this paper. The possibly remaining mistakes are all mine.


\begin{thebibliography}{refer}
\bibitem[An07]{An07} J.-P. ANTOINE, {\it Quantum Theory Beyond Hilbert Space}, in  Lecture Notes in Physics book series, Volume 504, 2007.
\bibitem[Ba84]{Ba84} T. BALABAN, {\it Propagators and renormalization transformations for lattice gauge theories I}, Comm. Math. Phys., Vol. 95, (17-40), 1984.
\bibitem[Ba84Bis]{Ba84Bis} T. BALABAN, {\it Propagators and renormalization transformations for lattice gauge theories. II}, Comm. Math. Phys., Vol. 96, (223-250), 1984.
\bibitem[Ba85]{Ba85} T. BALABAN, {\it Averaging operations for lattice gauge theories}, Comm. Math. Phys., Vol. 98, (17-51), 1985.
\bibitem[Ba85Bis]{Ba85Bis} T. BALABAN, {\it Spaces of regular gauge field configurations on a lattice and gauge fixing conditions}, Comm. Math. Phys.  Vol. 99, (75-102), 1985.
\bibitem[Ba85Tris]{Ba85Tris} T. BALABAN, {\it Propagators for lattice gauge theories in a background field}, Comm. Math. Phys., Vol 99, (389-434), 1985.
\bibitem[Ba85Quater]{Ba85Quater} T. BALABAN, {\it Ultraviolet stability of three-dimensional lattice pure gauge field theories}, Comm. Math. Phys., Vol. 102, (255-275), 1985.
\bibitem[Ba87]{Ba87} T. BALABAN, {\it Renormalization group approach to lattice gauge field theories. I. Generation of effective actions in a small field approximation and a coupling constant renormalization in four dimensions}, Comm. Math. Phys. Vol. 109, (249-301), 1987.
\bibitem[Ba88]{Ba88} T. BALABAN, {\it
Renormalization group approach to lattice gauge field theories. II. Cluster expansions}, Comm. Math. Phys. Vol. 116, (1-22), 1988.
\bibitem[Ba88Bis]{Ba88Bis} T. BALABAN, {\it Convergent renormalization expansions for lattice gauge theories}, Comm. Math. Phys. Vol. 119, (243-285), 1988.
\bibitem[Ba89]{Ba89} T. BALABAN, {\it
Large field renormalization. I. The basic step of the $\mathbf{R}$ operation}, Comm. Math. Phys. Vol. 122, (175-202), 1989.
\bibitem[Ba89Bis]{Ba89Bis} T. BALABAN, {\it
Large field renormalization. II. Localization, exponentiation, and bounds for the $\mathbf{R}$ operation}, Comm. Math. Phys. Vol. 122, (355-392), 1989.
\bibitem[Bau14]{Bau14} H. BAUM, {\it Eichfeldtheorie: Eine Einf\"uhrung in die Differentialgeometrie auf Faserb\"undeln}, Springer 2014.
\bibitem[BEP78]{BEP78} C. M. BENDER, T. EGUCHI, and H. PAGELS, {\it Gauge-Fixing Degeneracies and Confinement in Non-Abelian Gauge Theories}, Vol. 17, No. 4, (1086-1096), 1978.
\bibitem[BeKo12]{BeKo12} Yu. M. BEREZANSKY and Y.G. KONDRATIEV, {\it Spectral Methods in Infinite-Dimensional Analysis},
Mathematical Physics and Applied Mathematics, Vol 12 Part 2, Springer, 2012.
\bibitem[BHL11]{BHL11} V. BETZ, F. HIROSHIMA  and J. L\"{O}RINCZI, {\it Feynman-Ka\v{c}-Type Theorems and Gibbs Measures on Path Space: With Applications to Rigorous Quantum Field Theory}, De Gruyter Studies in Mathematics, 2011.
\bibitem[Bl05]{Bl05} D. BLEECKER, {\it Gauge Theory and Variational Principles}, Dover 2005.
\bibitem[Bo06]{Bo06} V. BOGACHEV, {\it Measure Theory, Vol 1 and 2}, Springer, 2006.
\bibitem [BLOT89]{BLOT89} N. N. BOGOLUBOV, A. A. LOGUNOV, A. I. OKSAK and I. TODOROV, {\it General Principles of Quantum Field Theory}, Mathematical Physics and Applied Mathematics, Kluwer, 1989.
\bibitem[Br03]{Br03} P. BRACKEN, {\it Hamiltonian Formulation of Yang-Mills Theory and Gauge Ambiguity in Quantization}, Can. J. Phys., Vol. 81, (545-554), 2003.
\bibitem[CCFI11]{CCFI11} O. CALIN, D.-C. CHANG, K. FURUTANI and C. IWASAKI, {\it Heat Kernels for Elliptic and Sub-elliptic Operators: Methods and Techniques}, Applied and Numerical Harmonic Analysis, Springer, 2011.
\bibitem[Ca97]{Ca97} P. CARTIER (with C. DEWITT-MORETTE, A. WURM and D. COLLINS), {\it A Rigorous Mathematical Foundation of Functional Integration}, Functional Integration: Basics and Applications, NATO ASI Series, Series B: Physics Vol. 361, (1-50), 1997.
\bibitem[CDM10]{CDM10} P. CARTIER, C. DEWITT-MORETTE, {\it Functional Integration: Action and Symmetries}, Cambridge Monographs on Mathematical Physics, 2010.
Universidad Polit\'{e}cnica de Valencia, Volume 10, No. 2, (187-195), 2009.
\bibitem[CCHS22]{CCHS22} A. CHANDRA, I. CHEVYREV, M. HAIRER and H. SHEN, {\it Stochastic Quantisation of Yang-Mills-Higgs in $3$D}, arXiv:2201.03487, 2022.
\bibitem[ChHa99]{ChHa99} L. CHEN and K. HALLER, {\it Quark Confinement and Color Transparency in a Gauge-Invariant Formulation of QCD},
International Journal of Modern Physics A, Vol. 14, No. 17,(2745-2767), 1999.
\bibitem[DP75]{DP75} L. I. DAIKHIN and D. A. POPOV, {\it Einstein Spaces, and Yang-Mills Fields}, Dokl. Akad. Nauk SSSR , 225:4, (790-793), 1975.
\bibitem [DZ92]{DZ92}  G. DA PRATO and G. ZABCZYK, {\it Stochastic Equations in Infinite Dimensions}, Encyclopedia of Mathematics and its Applications 45, Cambridge, 1992.
\bibitem [DD10]{DD10}  E. DE FARIA and W. DE MELO, {\it Mathematical Aspects of Quantum Field Theory}, Cambridge Studies in Advanced Mathematics, 2010.
\bibitem[Di13]{Di13} J. DIMOCK, {\it The Renormalization Group According to Balaban I: Small Fields},
Rev. Math. Phys. 25, (1–64), 2013.
\bibitem[Di13Bis]{Di13Bis} J. DIMOCK, {\it The Renormalization Group According to Balaban II: Large Fields},
J. Math. Phys. 54, (1–85), 2013.
\bibitem[Di14]{Di14} J. DIMOCK, {\it The Renormalization Group According to Balaban III: Convergence},
Annales Henri Poincaré volume 15, (2133–2175), 2014.
\bibitem[Doo94]{Doo94} J. L. DOOB, {\it Measure Theory}, Graduate Texts in Mathematics, Volume 143, 1994.
\bibitem [Do04]{Do04} M. R. DOUGLAS, {\it Report on the Status of the Yang-Mills Millenium Prize Problem}, Clay Mathematics Institute, 2004.
\bibitem[EM82]{EM82} D. M. EARDLEY and V. MONCRIEF, {\it The Global Existence of Yang-Mills-Higgs Fields in $4$-dimensional Minkowski Space I. Local Existence and Smoothness Properties}, Comm. Math. Phys., Vol. 83, No. 2, (171-191), 1982.
\bibitem[EM82Bis]{EM82Bis} D. M. EARDLEY and V. MONCRIEF, {\it The Global Existence of Yang-Mills-Higgs Fields in $4$-dimensional Minkowski Space II. Completion of Proof}, Comm. Math. Phys., Vol. 83, No. 2, (193-212), 1982.
\bibitem[EoM02]{EoM02} Encyclopedia of Mathemathics, {\it Yang-Mills Fields}, Kluwer, 2002.
\bibitem [Fa05]{Fa05} L. D. FADDEEV, {\it Mass in Quantum Yang-Mills Theory}, Perspective in Analysis, Math. Phys. Studies, Vol. 27, Springer, 2005.
\bibitem [FP67]{FP67} L. D. FADDEEV and V. N. POPOV, {\it Feynman Diagrams for the Yang-Mills Field}, Phys. Lett. B, Vol. 25, Issue 1, (29-30), 1967.
\bibitem [FO76]{FO76} J. S. FELDMAN and K. OSTERWALDER, {\it The Wightman Axioms and the Mass Gap for
Weakly Coupled $\varphi^4_3$ Quantum Field Theories}, Annals of Physics, 97(1), (80–135), 1976.
\bibitem[Fe67]{Fe67} X. FERNIQUE, {\it Processus lin\'{e}aires, processus g\'{e}n\'{e}ralis\'{e}s}, Annales de l’Institut Fourier, tome 17, no 1, (1-92), 1967.
\bibitem [FMRS87]{FMRS87} J. S. FELDMAN, J. MAGNEN, V. RIVASSEAU and R. S\'{E}N\'{E}OR, {\it Construction and Borel Summability of Infrared $\Phi^4_4$ by a Phase Space Expansion}
 Commun. Math. Phys. 109, (437-480), 1987.
\bibitem[GK85]{GK85} K. GAW\k{E}DZKI and A. KUPIAINEN, {\it Massless Lattice $\phi^4_4$ Theory: Rigorous Control of a Renormalizable Asymptotically Free Model}, Commun. Math. Phys. 99, (197-252), 1985.
\bibitem[GV64]{GV64} M. GEL'FAND and N. Y. VILENKIN, {\it Generalized Functions, Volume 4}, Academic Press, 1964.
\bibitem [Gi95]{Gi95} P. B. GILKEY, {\it Invariance Theory, the Heat Equation and the Atiyah-Singer Index Theorem}, Second Edition, Studies in Advanced Mathematics, CRC Press, 1995.

\bibitem [GJ73]{GJ73} J. GLIMM and A. JAFFE, {\it Positivity of the $\varphi^4_3$ Hamiltonian}, Fortschritte der Physik
(Progress of Physics), 21, (327–376), 1973.
\bibitem [GJ87]{GJ87} J. GLIMM and A. JAFFE, {\it Quantum Physics: A Functional Integral Point of View}, Second Edition, Springer, 1987.
\bibitem[Gr78]{Gr78} V. N. GRIBOV, {\it Quantization of Non-Abelian Gauge Theories}, Nucl. Phys. B, Vol. 139, Issues 1-2, (1-19), 1978.
\bibitem[Gro72]{Gro72} L. GROSS, {\it Existence and Uniqueness of Physical Ground States}, Journal Offunctional Analysis 10, (52-109), 1972.
\bibitem[GW73]{GW73} D. J. GROSS and F. WILCZEK,{\it Ultraviolet Behavior of Non-Abelian Gauge Theories}, Physical Review Letters. 30 (26), (1343–1346), 1973.
\bibitem[GuHo19]{GuHo19} M. GUBINELLI and M. HOFMANOV\'{A}, {\it Global Solutions to Elliptic and Parabolic $\Phi^4$
 Models in Euclidean Space}, Commun. Math. Phys. 368, (1201–1266), 2019.
\bibitem[Ha14]{Ha14} M. HAIRER, {\it A Theory of Regularity Structures}, Inventiones Mathematicae 198 (2), (269–504), 2014.
\bibitem[He97]{He97} T. HEINZL, {\it Hamiltonian Approach to the Gribov Problem}, Nucl. Phy. B, Proc. Suppl., Vol. 54 A, (194-197), 1997.
\bibitem [Ja82]{Ja82} A. JAFFE, {\it Renormalization}, Seminar on Differential Geometry,
Princeton University Press, Annals of Mathematics Studies 102, (507-523), 1982.
\bibitem [JW04]{JW04} A. JAFFE and E. WITTEN, {\it Quantum Yang-Mills Theory}, Official Yang-Mills and Mass Gap Problem Description, Clay Mathematics Institute, 2004.
\bibitem[Ku96]{Ku96} H.-H. KUO, {\it White Noise Distribution Theory}, Probability and Stochastics Series,  CRC Press, 1996.
\bibitem [La04]{La04} J. LACHAPELLE, {\it Path Integral Solution of Linear Second Order Partial Differential Equations I. The General Construction}, Ann. Phys., Vol. 314, (362-395), 2004.
\bibitem[Ma01]{Ma01} R. DE LA MADRID, {\it Quantum Mechanics in Rigged Hilbert Space Language}, PhD thesis,
Universidad de Valladolid, 2001.
\bibitem[Ma08]{Ma08} R. DE LA MADRID, {\it The Role of the Rigged Hilbert Space in Quantum Mechanics},  Eur. J. Phys., Vol. 26, No. 2, (287-312), 2005.
\bibitem[MS76]{MS76}  J. MAGNEN and R. S\'{E}N\'{E}OR, {\it The Infinite Volume Limit of the $\varphi^4_3$ Model}, Ann. Inst. H.
Poincar\'{e}, Sect. A (N.S.), 24(2), (95–159), 1976.
\bibitem[MRS93]{MRS93} J. MAGNEN, V. RIVASSEAU and R. S\'{E}N\'{E}OR, {\it Construction of YM$4$ with an Infrared Cutoff}, Comm. Math. Phys. Vol. 155, (325-383), 1993.
\bibitem [MY54]{MY54} R. L. MILLS and C. N. YANG, {\it Conservation of Isotopic Spin and Gauge Invariance}, Phy. Rev. Vol. 96, (191-195), 1954.
\bibitem[OS73]{OS73} K. OSTERWALDER and R. SCHRADER, {\it Axioms for Euclidean Green's Functions}, Comm. Math. Phys., Vol. 31, (83-112), 1973.
\bibitem[OS73Bis]{OS73Bis} K. OSTERWALDER and R. SCHRADER, {\it Euclidean fermi Fields and a Feynman-Ka\v{c} Formula for Boson-Fermion Models}, Helv. Phys. Acta Vol. 46, (277-302), 1973.
\bibitem [OS75]{OS75} K. OSTERWALDER and R. SCHRADER, {\it Axioms for Euclidean Green's Functions II}, Comm. Math. Phys., Vol. 42, (281-305), 1975.
\bibitem[Pe78]{Pe78} R. D. PECCEI, {\it Resolution of Ambiguities of Non-Abelian Gauge Fied Theories in the Coulomb Gauge}, Phys. Rev. D: Particle and Fields, Vol. 17, No. 4, (1097-1110), 1978.
\bibitem[Po73]{Po73} H. D. POLITZER, {\it Reliable Perturbative Results for Strong Interactions}, Physical Review Letters 30 (26), (1346–1349), 1973.
\bibitem [RS75]{RS75} M. REED and B. SIMON, {\it Fourier Analysis, Self-Adjointness}, Methods of Modern Mathematical Physics, Vol. 2, Academic Press, 1975.
\bibitem [RJ99]{RJ99} D. REVUZ and M. YOR, {\it Continuous Martingales and Brownian Motion}, Grundlehren der mathematischen Wissenschaften, Springer, 1999.
\bibitem[Ri85]{Ri85} R. D. RICHTMYER, {\it Principles of Advanced Mathematical Physics}, Springer Texts and Monographs in Physics, Second Edition, 1985.
\bibitem[Riv91]{Riv91} V. RIVASSEAU, {\it From Perturbative to Constructive Renormalization}, Princeton, NJ, Princeton University Press, 1991.
\bibitem[Ro66]{Ro66} J. E. ROBERTS, {\it Rigged Hilbert Spaces in Quantum Mechanics}, Comm. Math. Phys., Vol. 3, No. 2, (98-119), 1966.
\bibitem[SaSc10]{SaSc10} F. SANNINO and J. SCHECHTER, {\it Nonperturbative Results for Yang-Mills Theories}, Phys. Rev. D Vol. 82, No. 7, 096008,  2010.
\bibitem[Sch08]{Sch08} W. SCHLEIFENBAUM, {\it Nonperturbative Aspects of Yang-Mills Theory}, Dissertation zur Erlangung des Grades eines Doktors der Naturwissenschaften, Universit\"at T\"ubingen, 2008.
\bibitem[Sc73]{Sc73} L. SCHWARTZ, {\it Radon measures on arbitrary topological spaces and cylindrical measures}, Tata Institute of Fundamental Research Studies in Mathematics, London, Oxford University Press, 1973.
\bibitem[ScSn94]{ScSn94} G. SCHWARZ and J. \'{S}NIATYCKI, {\it The Existence and Uniqueness of Solutions of Yang-Mills Equations with Bag Boundary Conditions}, Comm. Math. Phys., Vol. 159, No. 3, (593-604), 1994.
\bibitem[ScSn95]{ScSn95} G. SCHWARZ and J. \'{S}NIATYCKI, {\it Yang-Mills and Dirac Fields in a Bag, Existence and Uniqueness Theorems}, Comm. Math. Phys., Vol. 168, No. 2, (441-453), 1995.
\bibitem[ScTw98]{ScTw98} C. SCHULTE and R. TWAROCK, {\it The Rigged Hilbert Space Formulation of Quantum Mechanics and its Implications for Irreversibility}, based on lectures by Arno Bohm, \verb+ http://arxiv.org/abs/quant-ph/9505004+, 1998.
\bibitem[Se78]{Se78} I. E. SEGAL, {\it General Properties of the Yang-Mills Equations in Physical Space}, Proc. Nat. Acad. Sci. USA, Vol. 75, No. 10, Applied Physical and Mathematical Sciences, (4608-4639), 1978.
\bibitem[Se79]{Se79} I. E. SEGAL, {\it The Cauchy Problem for the Yang-Mills Equations}, J. Funct. Anal, Vol. 33, (175-194),  1979.
\bibitem[SeSi76]{SeSi76} E. SEILER and B. SIMON, {\it Nelson’s Symmetry and All That in the $\text{Yukawa}_2$ and $(\varphi^4)_3$ Field Theories},
 Annals of Physics, 97(2), (470–518), 1976.
\bibitem [Sim05]{Sim05} B. SIMON, {\it Functional Integration and Quantum Physics}, AMS Chelsea Publishing, Second Edition, 2004.
\bibitem [Sim15]{Sim15} B. SIMON, {\it The $P(\Phi)_2$ Euclidean (Quantum) Field Theory}, Princeton Series in Physics, Princeton Legacy Library, 2015.
\bibitem[Si78]{Si78} I. M. SINGER, {\it Some Remarks on the Gribov Ambiguity}, Comm. Math. Phys. 60, (7-12), 1978.
\bibitem[Sp19]{Sp19} D. SPIEGEL, {\it Spectral Theory in Rigged Hilbert Space}, University of Colorado Boulder, 2019.
\bibitem [SW10]{SW10} R. F. STREATER and A. S. WIGHTMAN, {\it PCT, Spin and Statistics, and All That}, Princeton Landmarks in Physics, 2010.
\bibitem[Ti08]{Ti08} R. TICCIATI,  {\it Quantum Field Theory for Mathematicians}, Encyclopedia of Mathematics and its Applications, Series Number 72, Cambridge University Press, 2008.
\bibitem[Tr06]{Tr06} F. TR\`{E}VES, {\it Topological Vector Spaces, Distributions and Kernels}, Mineola, NY, Dover Publications, 2006.
\bibitem [We05]{We05} S. WEINBERG, {\it The Quantum Theory of Fields, Volume 2: Modern Applications}, Cambridge University Press, 2005.
\bibitem [Wig56]{Wig56} A. S. WIGHTMAN, {\it Quantum Field Theory in terms of its Vacuum Expectation Values}, Phy. Rev. 101, (860-866), 1956.
\bibitem[Ze09]{Ze09} E. ZEIDLER, {\it Quantum Field Theory I: Basics in Mathematics and Physics: A Bridge between Mathematicians and Physicists}, corrected 2nd printing, Springer 2009.
\end{thebibliography}
\end{document}